\documentclass[11pt]{amsart}

\usepackage[pdftex]{graphicx}
\usepackage{tikz}
\usepackage{ascmac}
\usepackage{amssymb,amsmath,amsthm}
\usepackage{enumerate}
\usepackage[margin=25truemm]{geometry}
\usepackage{mathrsfs}

\usepackage{mycommand}
\usetikzlibrary{arrows.meta}

\theoremstyle{plain}
\newtheorem{theo}{Theorem}[section]
\newtheorem{lemm}[theo]{Lemma}
\newtheorem{corr}[theo]{Corollary}
\newtheorem{prop}[theo]{Proposition}

\theoremstyle{definition}
\newtheorem{defi}[theo]{Definition}
\newtheorem{exam}[theo]{Example}
\newtheorem{rema}[theo]{Remark}
\newtheorem{ques}{Question}

\usepackage[pdftex]{hyperref}
\hypersetup{%
setpagesize=false,%
bookmarks=true,%
bookmarksdepth=tocdepth,%
bookmarksnumbered=true,%
colorlinks=true,%
linkcolor=purple,%
citecolor=teal,%
pdftitle={},%
pdfsubject={},%
}

\begin{document}
\title{A characterization of 2-fold torsion classes induced by $\tau$-rigid modules}

\author{Yuki Uchida}	
\address{Department~of~Mathematics, Graduate~School~of~Science, Osaka~Metropolitan~University, 3-3-138, sugimoto, sumiyoshi, Osaka~558-8585, Japan}
\email{yukiuchidamath@gmail.com}

\maketitle

\begin{abstract}
 We introduce $n$-fold torsion(-free) classes of an abelian category. These are a generalization of ordinary torsion(-free) classes in the sense that $1$-fold torsion(-free) classes coincide with torsion(-free) classes.
 In the category of finitely generated modules over a finite dimensional algebra, we can naturally construct $n$-fold torsion classes from $\tau$-rigid modules.
 We characterize the $2$-fold torsion classes induced by $\tau$-rigid modules.
\end{abstract}

\small
\tableofcontents
\normalsize


\section{Introduction}
In an abelian category, there are several kinds of subcategories which have been studied so far, for instance, \textit{Serre subcategories, torsion classes, torsion-free classes} and so on.
These subcategories are defined to be closed under certain operations, for example, taking extensions, subobjects, quotients, kernels, cokernels, etc.
Recently, the following characterization of \textit{$\KE$-closed subcategories} ($\resp$ \textit{$\CE$-closed subcategories}), which are subcategories closed under kernels ($\resp$ cokernels) and extensions, has been proved in \cite{kobayashi2024ke}:

\begin{theo}[{\cite[Propositions 3.1 and 3.2]{kobayashi2024ke}}] \label{thm A}
	Let $\ca$ be an abelian category and $\cx$ its subcategory.
	\begin{enumerate}[$(1)$]
		\setlength{\itemsep}{0pt}
		\item The following are equivalent$\colon$
		\begin{enumerate}[$(a)$]
			\setlength{\parskip}{0pt}
				\item $\cx$ is a $\KE$-closed subcategory of $\ca$.
				\item There exists some torsion-free class $\cf$ of $\ca$ such that $\cx$ is a torsion-free class of $\cf$ in the sense of exact categories (see Definition \ref{def ex subcat}).
		\end{enumerate}
		
		\item The following are equivalent$\colon$
		\begin{enumerate}[$(a)$]
			\setlength{\parskip}{0pt}
				\item $\cx$ is a $\CE$-closed subcategory of $\ca$.
				\item There exists some torsion class $\ct$ of $\ca$ such that $\cx$ is a torsion class of $\ct$ in the sense of exact categories.
		\end{enumerate}

	\end{enumerate}
\end{theo}

Inspired by the theorem, we introduce \textit{$n$-fold torsion(-free) classes} as follows:

\begin{defi} [{Definition \ref{def nfold torf}}]
	Let $\ca$ be an abelian category, $\cx$ its subcategory and $n$ a positive integer. 
	Then $\cx$ is said to be an \textit{$n$-fold torsion-free class of $\ca$} if there exist subcategories $\cf_{0}, \cf_{1}, \cdots, \cf_{n}$ of $\ca$ satisfying the following two conditions:
	\begin{enumerate}[$(1)$]
	\setlength{\itemsep}{0pt} 
		\item $\cf_{0} = \ca$ and $\cf_{n} = \cx$,
		\item $\cf_{i+1}$ is a torsion-free class of $\cf_{i}$ for $i \in \{0, 1, \cdots, n-1\}$.
	\end{enumerate}

	Dually, $\cx$ is said to be an \textit{$n$-fold torsion class of $\ca$} if there exist subcategories $\ct_{0}, \ct_{1}, \cdots, \ct_{n}$ of $\ca$ satisfying the following two conditions:
	\begin{enumerate}[$(1)$]
		\setlength{\itemsep}{0pt} 
			\item $\ct_{0} = \ca$ and $\ct_{n} = \cx$,
			\item $\ct_{i+1}$ is a torsion class of $\ct_{i}$ for $i \in \{0, 1, \cdots, n-1\}$.
	\end{enumerate}

\end{defi}

We wonder if $n$-fold torsion(-free) classes of an abelian category are characterized by being closed under certain operations in the abelian category. 
We consider \textit{$\K^{n} \E$-closed subcategories} ($\resp$ \textit{$\Cc^{n} \E$-closed subcategories}) (Definition \ref{def abelian subcat} (6)) of $\ca$ which are a generalization of $\KE$-closed subcategories ($\resp$ $\CE$-closed subcategories) of $\ca$. 
The following implication is easier:

\begin{prop} [Proposition {\ref{prop nfold vs Kn-1Eclosed}} and the dual of it]
	Let $\ca$ be an abelian category, $\cx$ its subcategory and $n$ a positive integer. 
	If $\cx$ is an $n$-fold torsion-free class ($\resp$ $n$-fold torsion class) of $\ca$, then  $\cx$ is a $\K^{(n-1)}\E$-closed ($\resp$ $\Cc^{(n-1)}\E$-closed) subcategory of $\ca$.
\end{prop}

We do not know whether the converse of the proposition holds in general, but it holds under specific conditions$\colon$

\begin{theo}[Theorem \ref{thm resolving}]
	Let $\ca$ be an abelian category with enough projectives and injectives, $(\cx, \cy)$ a hereditary cotorsion pair in $\ca$ and $n$ a positive integer. Then the following are equivalent$\colon$
\begin{enumerate}[$(a)$]
\setlength{\parskip}{0pt}
	\item $\cx$ is an $n$-fold torsion-free class of $\ca$.
	\item $\cx$ is a $\K^{(n-1)}\E$-closed subcategory of $\ca$.
	\item The resolution dimension of $\ca$ with respect to $\cx$ is at most $n$.
\end{enumerate}
Dually, the following are equivalent$\colon$
\begin{enumerate}[$(a)^{\mathrm{op}}$]
\setlength{\parskip}{0pt}
	\item \textit{$\cy$ is an $n$-fold torsion class of $\ca$.}
	\item \textit{$\cy$ is a $\Cc^{(n-1)}\E$-closed subcategory of $\ca$.}
	\item The coresolution dimension of $\ca$ with respect to $\cy$ is at most $n$.
\end{enumerate}
\end{theo}

In order to understand why the case $n=2$ works well as in \cite{kobayashi2024ke}, we conduct a more detailed observation 
and obtain the following result on the kernel closure:

\begin{theo}[{Theorem \ref{thm charac of kernel closure}}]
	Let $\ca$ be an abelian category and $\cx$ its additive subcategory. Then the following three subcategories coincide$\colon$
	\begin{enumerate}[$(a)$]
	\setlength{\parskip}{0pt}
		\item the admissible-subobject closure of $\cx$ in $\F{1}{\cx}$,
		\item $\clos{\cx}{\kr_{\cx}}{}$, and
		\item $\clos{\cx}{\kr}{}$.
	\end{enumerate}
	In particular, the following are equivalent$\colon$
	\begin{enumerate}[$(a)$]
		\setlength{\parskip}{0pt}
			\item $\cx$ is an admissible-subobject-closed subcategory of $\F{1}{\cx} \mathrm{;}$
			\item $\cx$ is a $\mathrm{Ker}_{\cx}$-closed subcategory of $\ca$ (see Definition \ref{def kery closed})$\mathrm{;}$ and
			\item $\cx$ is a kernel-closed subcategory of $\ca$.
		\end{enumerate}
	\end{theo}

In the category of finitely generated modules over an finite dimensional $\mathbb{K}$-algebra $\Lambda$, where $\mathbb{K}$ is an algebraically closed field, 
there are bijections between ``nice'' subcategories and certain modules. In particular, one of the most eminent results arises in \textit{$\tau$-tilting theory} established in \cite{Adachi_Iyama_Reiten_2014}.
They gave a bijection between functorially finite torsion classes and basic support $\tau$-tilting modules as follows:

\begin{theo}[{\cite[Theorem 2.7]{Adachi_Iyama_Reiten_2014}}] \label{thm AIR}
	The following maps are mutually inverse$\colon$
	\[
	\begin{tikzpicture}[auto]
	\node at (-2.4, 1.8) {$\stautilt \Lambda$}; \node  at (2.4, 1.8) {$\ftors \Lambda$}; 
	\node[rotate=90]  at (-2.4, 1.3) {$\in$}; \node[rotate=90]  at (2.4, 1.3) {$\in$}; 
	\node  at (-2.4, 0.8) {$T$}; \node  at (2.4, 0.8) {$\Fac T$}; 
	\node  at (-2.4, 0.2) {$P(\ct)$}; \node  at (2.4, 0.2) {$\ct$}; 
	\node(a1) at (-1.8, 1.9) {}; \node(a2) at (1.8, 1.9) {};
	\node(01) at (-1.8, 1.7) {}; \node(02) at (1.8, 1.7) {};

	\node(11) at (-1.8, 0.8) {}; \node(12) at (1.8, 0.8) {};
	\node(21) at (-1.8, 0.2) {}; \node(22) at (1.8, 0.2) {};
	\node at (2.6,0.1) {,};

	\draw[->, thick] (a1) --node {$\Fac$} (a2); \draw[->, thick] (02) --node {$P(-)$} (01); \draw[|->, thick] (11) to (12); \draw[|->, thick] (22) -- (21); 
	
	\end{tikzpicture}
	\]
	where $P(\ct)$ is the basic $\Ext$-progenerator of $\ct$.
\end{theo}

We can naturally construct $n$-fold torsion classes from a certain class of modules called \textit{$\tau$-rigid modules}, which includes support $\tau$-tilting modules.

\begin{prop} [Proposition {\ref{prop coknU is n+1fold tors}}]
	Let $U$ be a $\tau$-rigid module in $\md \Lambda$. For every nonnegative integer $n$, $\cok_{n} U$ is an $(n+1)$-fold torsion class of $\md \Lambda$, 
	where $\cok_{n} U$ is the subcategory consisting of all modules which are the 
	$n$-cokernel of some $n$-term exact sequence in $\add U$.
\end{prop}

The main goal of this paper is to provide a bijection between basic $\tau$-rigid modules and $2$-fold torsion classes characterized by a certain condition, similar to Theorem \ref{thm AIR}.
More precisely, the condition for a $2$-fold torsion class $\cc$ is the following$\colon$
\begin{itemize}
	\item The smallest torsion class $\T{1}{\cc}$ of $\md \Lambda$ containing $\cc$ is functorially finite; and
	\item $\conast \colon$ For any $C \in \cc$ and the right minimal $(\add P)$-approximation $f \colon P_{C} \to C$ of $C$, we have  $P_{C} \in \cc$, 
	where $P$ is the basic $\mathrm{Ext}$-progenerator of $\T{1}{\cc}$. 
\end{itemize}

We denote by $\fltwotors^{\ast} \Lambda$ the set of $2$-fold torsion classes satisfying the condition and by $\taurigid \Lambda$ the set of isomorphism classes of basic $\tau$-rigid $\Lambda$-modules.

\begin{theo}[Theorem \ref{thm charac of 2tors}] \label{thm C}
The following maps are mutually inverse$\colon$

\[
\begin{tikzpicture}[auto]
	\node at (-2.4, 2.0) {$\taurigid \Lambda$}; \node  at (2.4, 2.0) {$\fltwotors^{\ast} \Lambda$}; 
	\node[rotate=90]  at (-2.4, 1.5) {$\in$}; \node[rotate=90]  at (2.4, 1.5) {$\in$}; 
	\node  at (-2.4, 1.0) {$U$}; \node  at (2.4, 1.0) {$\cok_{1} U$}; 
	\node  at (-2.4, 0.4) {$\Phi(\cc)$}; \node  at (2.4, 0.4) {$\cc$}; 
	\node at (2.6, 0.3) {,};
	\node(011) at (-1.7, 2.1) {}; \node(021) at (1.5, 2.1) {};
	\node(012) at (-1.7, 1.9) {}; \node(022) at (1.5, 1.9) {};
	\node(11) at (-1.7, 1.0) {}; \node(12) at (1.5, 1.0) {};
	\node(21) at (-1.7, 0.4) {}; \node(22) at (1.5, 0.4) {};

	\draw[->, thick] (011) --node{$\cok_{1}$} (021); \draw[->, thick] (022) --node{$\Phi$} (012); \draw[|->, thick] (11) to (12); \draw[|->, thick] (22) -- (21); 
	
\end{tikzpicture}
\]
where $\Phi(\cc)$ is the basic module such that $\add \Phi(\cc) = \cc \cap \add P$. 
\end{theo}

Theorem \ref{thm AIR} and Theorem \ref{thm C} have the following relationship:

\begin{prop} [{Proposition \ref{prop relationship of AIR and mine}}]
	The following diagram commutes$\colon$
\[
\begin{tikzpicture}[auto]
\node (0) at (-2.0, 2.0) {$\stautilt \Lambda$}; \node (00) at (2.0, 2.0) {$\ftors \Lambda$}; 
\node (1) at (-2.0, 0.0) {$\taurigid \Lambda$}; \node (01) at (2.0, 0.0) {$\fltwotors^{\ast} \Lambda$}; 
\node (2) at (-2.0, -2.0) {$\stautilt \Lambda$}; \node (02) at (2.0, -2.0) {$\ftors \Lambda$}; 
\node at (2.7, -2.1) {,};
\node(011) at (-1.2, 2.0) {}; \node(021) at (1.3, 2.0) {};
\node(111) at (-1.2, 0.0) {}; \node(121) at (1.1, 0.0) {};
\node(211) at (-1.2, -2.0) {}; \node(221) at (1.3, -2.0) {};

\draw[->, thick] (111) --node{$\cok_{1}$} (121);
\draw[->, thick] (011) --node{$\Fac$} (021);
\draw[->, thick] (211) --node{$\Fac$} (221);

\draw[{Hooks[right]}->, thick] (0) -- (1); \draw[->, thick] (1) --node[swap]{$\overline{(-)}$} (2);
\draw[{Hooks[right]}->, thick] (00) -- (01); \draw[->, thick] (01) --node{$\T{1}{-}$} (02);
\draw[->, thick] (0) to[bend right=70, edge label'={$\mathrm{id}$}] (2);
\draw[->, thick] (00) to[bend left=70, edge label={$\mathrm{id}$}] (02);
\draw[->, thick] (1) -- (02);
\draw[fill=white, draw=white] (0, -1.0) circle[x radius=10pt, y radius=10pt] ;
\node at (0, -1.0) {$\Fac$};

\end{tikzpicture}
\]
where $\overline{(-)}$ is the map that sends $U \in \taurigid \Lambda$ to the co-Bongartz completion of $U$. 
\end{prop}

The subcategories induced by $\tau$-rigid modules satisfy good properties.

\begin{theo}[{Theorem \ref{thm properties of 2tors induced by taurigid}}] \label{thm D}
	Let $\cc \in \fltwotors^{\ast} \Lambda$. Then $\cc$ has the following properties$\colon$
	\begin{enumerate}[$(a)$]
		\setlength{\itemsep}{0pt}
		\item $\cc$ is functorially finite in $\md \Lambda$. 
		\item There exist some subcategories $\cf_{1}, \cf_{2}$ of $\md \Lambda$ such that $(\cc, \T{1}{\cc} ; \cf_{2}, \cf_{1})$ is a $2$-fold torsion pair of $\md \Lambda$ (see Definition \ref{def nfold torspair}).
		\item $\cc$ admits an $\mathrm{Ext}$-progenerator.
	\end{enumerate} 
\end{theo}

$2$-fold torsion pairs are a generalization of torsion pairs.
Since $1$-fold torsion classes induced by $\tau$-rigid modules satisfy the properties $(a)$ and $(c)$ of Theorem \ref{thm D} (see Theorem \ref{thm charac ftors}), we can regard Theorem \ref{thm D} as a generalization.

In the hereditary case, Theorem \ref{thm C} recovers the following bijection constructed by Enomoto.
\begin{theo}[{\cite[Theorem 2.3]{enomoto2022rigid}}] \label{thm Enomoto}
	Assume that $\Lambda$ is hereditary.
  The following maps are mutually inverse$\colon$
\[
\begin{tikzpicture}[auto]
	\node at (-2.4, 2.0) {$\rigid \Lambda$}; \node  at (2.4, 2.0) {$\icep \Lambda$}; 
	\node[rotate=90]  at (-2.4, 1.5) {$\in$}; \node[rotate=90]  at (2.4, 1.5) {$\in$}; 
	\node  at (-2.4, 1.0) {$U$}; \node  at (2.4, 1.0) {$\cok_{1} U$}; 
	\node  at (-2.4, 0.4) {$P(\cc)$}; \node  at (2.4, 0.4) {$\cc$}; 
	\node at (2.6, 0.3) {,};
	\node(011) at (-1.7, 2.1) {}; \node(021) at (1.6, 2.1) {};
	\node(012) at (-1.7, 1.9) {}; \node(022) at (1.6, 1.9) {};
	\node(11) at (-1.7, 1.0) {}; \node(12) at (1.6, 1.0) {};
	\node(21) at (-1.7, 0.4) {}; \node(22) at (1.6, 0.4) {};	

	\draw[->, thick] (011) --node{$\cok_{1}$} (021); \draw[->, thick] (022) --node{$P(-)$} (012); \draw[|->, thick] (11) to (12); \draw[|->, thick] (22) -- (21); 
	
\end{tikzpicture}
\]
where $\rigid \Lambda$ is the set of isomorphism classes of basic rigid $\Lambda$-modules and $\icep \Lambda$ is the set of $\ICE$-closed subcategoies of $\md \Lambda$ with enough $\mathrm{Ext}$-projectives.
\end{theo}

This paper is organized as follows. In section \ref{sec Preliminaries}, we recall several classes of subcategoies of an abelian category and subcategories of an extension-closed subcategory of an abelian category.
In section \ref{sec nfold torsionfree classes}, we define $n$-fold torsion(-free) classes, show basic properties of them and give an equivalent condition for a subcategory of an abelian category to be an $n$-fold torsion-free class. In section \ref{sec 2fold torsionfree classes and kernel closures}, we provide a different proof of Kobayashi-Saito's result (Theorem \ref{thm A}) with other descriptions of the kernel closures. 
In section \ref{sec nfold torsionfree classes from the viewpoint of hereditary cotorsion pairs}, we give an equivalent condition for the left subcategory of a hereditary cotorsion pair of an abelian category to be an $n$-fold torsion-free class.
In section \ref{sec Preliminaries for the main theorem}, we recall notions and properties about the category of finitely generated modules over a finite dimensional algebra.
In subsection \ref{subsec approximation}, we recall the definitions and properties of left (or right) approximations and functorially finiteness.
In subsection \ref{subsec the bijection by AIR}, we review the bijection constructed by \cite{Adachi_Iyama_Reiten_2014} (Theorem \ref{thm AIR}) and several properties of functorially finite torsion classes for comparison with sections \ref{sec charac of 2tors} and \ref{sec several properties}. 
In section \ref{sec charac of 2tors}, we prove the main theorem in this paper, that is, we give a characterization of $2$-fold torsion classes induced by $\tau$-rigid modules.
In section \ref{sec several properties}, we show that $2$-fold torsion classes induced by $\tau$-rigid modules satisfy nice properties which $1$-fold torsion classes satisfy.
In section \ref{sec the hereditary case}, we recover Enomoto's result (Theorem \ref{thm Enomoto}) with our results in sections \ref{sec charac of 2tors} and \ref{sec several properties}.
In section \ref{sec example}, we give a few examples of $n$-fold torsion(-free) classes.

\subsection*{Conventions and notations}
Throughout this paper, all subcategories are assumed to be full and closed under isomorphisms.
We write $X \in \cx$ if $X$ is an object of a category $\cx$.
A ring refers to an associative ring with identity, which is not necessarily commutative. 
When clear from the context, for a ring $R$, a right $R$-module is simply referred to as a module.

For a ring $R$, we denote by $\Mod R$ the category of right $R$-modules and by $\md R$ that of finitely generated right $R$-modules.

\subsection*{Acknowledgments}
The author would like to extend sincere gratitude to his supervisor Ryo Kanda for his generous teachings, guidance and encouragement. 
He would also like to thank Shunya Saito and Arashi Sakai for helpful comments and Haruhisa Enomoto for letting the author include his proof for Proposition \ref{prop Extprogen of ftors} and for helpful comments.
The author would like to thank his seminar members Yoshihiro Tanaka and Ryosuke Hashimoto for stimulating discussions, valuable comments and encouragement. 


\section{Preliminaries} \label{sec Preliminaries}
We recall several operations in an abelian category.
\begin{defi} \label{def abelian operation}

Let $\ca$ be an abelian category, $\cx$ its subcategory and $n$ a nonnegative integer.
\begin{enumerate}[(1)]
\setlength{\itemsep}{0pt} 

	\item $\cx$ is said to be \textit{closed under extensions} if for any exact sequence $0 \to A \to B \to C \to 0$ in $\ca$ with $A, C \in \cx$, we have that $B \in \cx$.

	\item $\cx$ is said to be \textit{closed under direct summands} if for any $A, B \in \ca$ with $A \oplus B \in \cx$, we have that $A \in \cx$.
	
	\item $\cx$ is said to be \textit{closed under finite direct sums} if for any $A, B \in \cx$, we have that $A \oplus B \in \cx$.

	\item $\cx$ is said to be \textit{closed under subobjects} ($\resp$ \textit{quotients}) if for any subobject ($\resp$ quotient) $A$ of an object $X \in \cx$, we have that $A \in \cx$.

	\item $\cx$ is said to be \textit{closed under kernels} ($\resp$ \textit{cokernels}) if for any morphism $f \colon X \to Y$ in $\ca$ with $X,Y \in \cx$, we have that $\Ker f \in \cx$ ($\resp$ $\Cok f \in \cx$).

	\item $\cx$ is said to be \textit{closed under epi-kernels} ($\resp$ \textit{mono-cokernels}) if for any epimorphism $f \colon X \twoheadrightarrow Y$ ($\resp$ monomorphism $X \hookrightarrow Y$) in $\ca$ with $X,Y \in \cx$, we have that $\Ker f \in \cx$ ($\resp$ $\Cok f \in \cx$).

	\item $\cx$ is said to be \textit{closed under images} if for any morphism $f \colon X \to Y$ in $\ca$ with $X,Y \in \cx$, we have that $\Ima f \in \cx$.

	\item ({\cite[Definition 3.1]{10.1093/imrn/rnaa284}}) $\cx$ is said to be \textit{closed under $n$-kernels} ($\resp$ \textit{$n$-cokernels}) if for any exact sequence $0 \to M \to X^{0} \to X^{1} \to \cdots \to X^{n}$ ($\resp$ $X_{n} \to X_{n-1} \to \cdots \to X_{0} \to M \to 0$) in $\ca$  such that each $X^{i}$($\resp X_{i}$) belongs to $\cx$, we have that $M \in \cx$.\footnote{Our definitions of $n$-kernels and $n$-cokernels are introduced in \cite{10.1093/imrn/rnaa284} and are different from those in \cite{jasso2016n}.}
\end{enumerate}

\end{defi}

We define several classes of subcategories of an abelian category.

\begin{defi} \label{def abelian subcat}
Let $\ca$ be an abelian category, $\cx$ its subcategory and $n$ a nonnegative integer.

\begin{enumerate}[(1)]
\setlength{\itemsep}{0pt} 

	\item $\cx$ is called an \textit{additive subcategory of $\ca$}  if it is closed under finite direct sums and contains the zero object of $\ca$.

	\item $\cx$ is called a \textit{torsion-free class of $\ca$} ($\resp$ \textit{torsion class}) if it is closed under subobjects ($\resp$ quotients) and extensions.

	\item (\cite[Definition 2.1]{enomoto2022rigid})  $\cx$ is said to be \textit{$\KE$-closed} ($\resp$ \textit{$\CE$-closed}) if it is closed under kernels ($\resp$ cokernels) and extensions.
	
	\item (\cite[Definition 2.1]{enomoto2022rigid}) $\cx$ is said to be \textit{$\IKE$-closed} ($\resp$ \textit{$\ICE$-closed}) if it is closed under images, kernels ($\resp$ cokernels) and extensions.

	\item  $\cx$ is said to be \textit{$\K^{n}\E$-closed} ($\resp$ \textit{$\Cc^{n}\E$-closed}) if it is closed under $n$-kernels ($\resp$ $n$-cokernels) and extensions.

\end{enumerate}
\end{defi}

\begin{rema}
The following are obvious:
	\begin{itemize}
	\setlength{\itemsep}{0pt} 
		
		\item $\K^{0}\E$-closed subcategory $ = $ torsion-free class.
		\item $\K^{1}\E$-closed subcategory $ = $ $\KE$-closed subcategory.
		
	\end{itemize}
	Dually, 
	\begin{itemize}
	\setlength{\itemsep}{0pt} 
		\item  $\Cc^{0}\E$-closed subcategory $ = $ torsion class.
		\item  $\Cc^{1}\E$-closed subcategory $ = $ $\CE$-closed subcategory.
	\end{itemize}
\end{rema}

\begin{rema}[{\cite[Remark 3.2 (4)]{10.1093/imrn/rnaa284}}] \label{rem knE is direct summands closed}
One can easily show that any $\K^{n} \E$-closed ($\resp$ $\Cc^{n} \E$-closed) subcategory $\cx$ in an abelian category $\ca$ is closed under direct summands. 
\end{rema}

$\K^{n}\E$-closed and $\Cc^{n}\E$-closed subcategories naturally appear in the sense of Example \ref{ex KnE}.

We often use the following notations about subcategories of an abelian category.

\begin{defi} \label{def closure in abelian cat}
Let $\cx$ be a subcategory of an abelian category $\ca$. 
\begin{enumerate}[(1)]
\setlength{\itemsep}{0pt} 

	\item We denote by $\clos{\cx}{\sub}{}$ the smallest subobject-closed subcategory of $\ca$ containing $\cx$.
	 
Note that
\[
\clos{\cx}{\sub}{} = \lbrace S \in \ca \mid S \text{ is a subobject of an object of }\cx \rbrace
\]holds.

	\item We denote by $\clos{\cx}{\quot}{}$ the smallest quotient-closed subcategory of $\ca$ containing $\cx$.
	 
Note that 
\[
\clos{\cx}{\quot}{} = \lbrace Q \in \ca \mid Q \text{ is a quotient of an object of }\cx \rbrace
\]holds.

	\item We denote by $\cx * \cy$ the subcategory of $\ca$ consisting of $M \in \ca$ such that there exists an exact sequence $0 \to X \to M \to Y \to 0$ in $\ca$ with $X \in \cx$ and $Y \in \cy$.

	\item We denote by $\clos{\cx}{\ext}{}$ the smallest extension-closed subcategory of $\ca$ containing $\cx$.
	
	Note that
	\[ \clos{\cx}{\ext}{} = \bigcup_{n \ge 0} \cx^{*n} \]
	holds, where $\cx^{*0} = {0} \text{ and } \cx^{*(n+1)} = \cx^{*n} * \cx$ for $n \geq 0$.

	\item We denote by  $\clos{\cx}{\kr}{}$ ($\resp$ $\clos{\cx}{\mathrm{cok}}{}$) the smallest kernel-closed ($\resp$ cokernel-closed) subcategory of $\ca$ containing $\cx$.

	\item We denote by  $\clos{\cx}{\tf}{}$ ($\resp$ $\clos{\cx}{\ts}{}$) the smallest torsion-free class ($\resp$ torsion class) of $\ca$ containing $\cx$.

	\item We denote by $\clos{\cx}{\KE}{}$ ($\resp$ $\clos{\cx}{\CE}{}$) the smallest $\KE$-closed ($\resp$ $\CE$-closed) subcategory of $\ca$ containing $\cx$.

\end{enumerate}
\end{defi}

We recall some important properties of these closures.

\begin{prop}[{see, for example, \cite[Lemmas 2.5 and 2.6]{enomoto2023ie}}] \label{prop clos}
Let $\ca$ be an abelian category and $\cx$ its subcategories.
It holds that $\clos{\cx}{\tf}{} = \clos{\clos{\cx}{\sub}{}}{\ext}{}$ and $\clos{\cx}{\ts}{} = \clos{\clos{\cx}{\quot}{}}{\ext}{}$.
\end{prop}

Next, we define some classes of subcategories of an extension-closed subcategory of an abelian category.

\begin{defi} \label{def ex subcat}
Let $\cx$ be an extension-closed subcategory of an abelian category $\ca$. 
\begin{enumerate}[(1)]
\setlength{\itemsep}{0pt}

	\item A \textit{conflation} in $\cx$ is an exact sequence  $0 \to A \to B \to C \to 0$ in $\ca$ such that $A,B \text{ and } C \in \cx$.

	\item A subcategory $\cs$ of $\cx$ is said to be \textit{closed under conflations in $\cx$} if for any conflation  $0 \to A \to B \to C \to 0$ in $\cx$ with $A,C \in \cs$, we have that $B \in \cs$.

	\item A subcategory $\cs$ of $\cx$ is said to be \textit{closed under admissible subobjects} ($\resp$ \textit{admissible quotients}) in $\cx$ if for any conflation  $0 \to A \to B \to C \to 0$ in $\cx$ with $B \in \cs$, we have that $A \in \cs$ ($\resp$ $C \in \cs$).

	\item A \textit{Serre subcategory} of $\cx$ is a subcategory of $\cx$ closed under conflations,  admissible subobjects and admissible quotients of $\cx$.

	\item A \textit{torsion-free class} of $\cx$ is a subcategory of $\cx$ closed under conflations and admissible subobjects of $\cx$.

	\item A \textit{torsion class} of $\cx$ is a subcategory of $\cx$ closed under conflations and admissible quotients of $\cx$.
\end{enumerate}
\end{defi}

Let $\ca$ be an abelian category, $\ce$ its extension-closed subcategory and $\cx$ a subcategory of $\ca$. 
We write $\cx \torf \ce$ ($\resp$ $\cx \tors \ce$) if $\cx$ is a torsion-free class ($\resp$ torsion class) of $\ce$.

Similarly to Definition \ref{def closure in abelian cat}, let us introduce the following notations:

\begin{defi}
Let $\cf$ be an extension-closed subcategory of an abelian category $\ca$ and $\cx$ a subcategory of $\ca$ such that $\cx \subset \cf$.
\begin{enumerate}[(1)]
\setlength{\itemsep}{0pt} 

	\item We denote by $\clos{\cx}{\ads}{\cf}$ the smallest admissible-subobject-closed subcategory of $\cf$ containing $\cx$. 
	It is easy to show that
	\[
	\clos{\cx}{\ads}{\cf} = \lbrace S \in \cf \mid S \text{ is an admissible subobject of an object of }\cx \text{ in }\cf \rbrace.
	\]

	\item We denote by $\clos{\cx}{\adq}{\cf}$ the smallest admissible-quotient-closed subcategory of $\cf$ containing $\cx$. It is easy to show that
	\[
	\clos{\cx}{\adq}{\cf} = \lbrace S \in \cf \mid S \text{ is an admissible quotient of an object of }\cx \text{ in }\cf \rbrace.
	\]

	\item We denote by $\clos{\cx}{\tf}{\cf}$ ($\resp$ $\clos{\cx}{\ts}{\cf}$) the smallest torsion-free class ($\resp$ torsion class) of $\cf$ containing $\cx$.\\

\end{enumerate}
\end{defi}

\begin{rema}
Let $\ca$ be an abelian category and $\cx$ its subcategory. We can regard the abelian category $\ca$ as an extension-closed category of itself. Then 
\[
\clos{\cx}{\tf}{\ca} = \clos{\cx}{\tf}{}
\]
holds.
\end{rema}

We often use the next lemma as naturally as breathing in this paper.
\begin{lemm} \label{lem extension closed}
Let $\cf$ be an extension-closed subcategory of an abelian category $\ca$ and $\cx$ a subcategory of $\ca$ such that $\cx \subset \cf$. 
Then $\cx$ is closed under extensions in $\ca$ if and only if it is closed under conflations in $\cf$.
\end{lemm}
\begin{proof}
It is straightforward.
\end{proof}


\section{$n$-fold torsion-free classes} \label{sec nfold torsionfree classes}
In this section, we define $n$-fold torsion(-free) classes and show basic properties of them. 

\begin{defi} \label{def nfold torf}
Let $\ca$ be an abelian category, $\cx$ its subcategory and $n$ a nonnegative integer. 
Then $\cx$ is said to be an \textit{n-fold torsion-free class of $\ca$} if there exist subcategories $\cf_{0}, \cf_{1}, \cdots, \cf_{n}$ of $\ca$ satisfying the following two conditions:
\begin{enumerate}[(1)]
\setlength{\itemsep}{0pt} 
	\item $\cf_{0} = \ca$ and $\cf_{n} = \cx$,
	\item $\cf_{i+1}$ is a torsion-free class of $\cf_{i}$ for $i \in \{0, 1, \cdots, n-1\}$.
\end{enumerate}
In summary,
\[
\cx = \cf_{n} \torf \cdots \torf \cf_{1} \torf \cf_{0} = \ca.
\]

Dually, $\cx$ is said to be an \textit{n-fold torsion class of $\ca$} if there exist subcategories $\ct_{0}, \ct_{1}, \cdots, \ct_{n}$ of $\ca$ satisfying the following two conditions:
\begin{enumerate}[(1)]
	\setlength{\itemsep}{0pt} 
		\item $\ct_{0} = \ca$ and $\ct_{n} = \cx$,
		\item $\ct_{i+1}$ is a torsion class of $\ct_{i}$ for $i \in \{0, 1, \cdots, n-1\}$.
\end{enumerate}
In summary,
\[
\cx = \ct_{n} \tors \cdots \tors \ct_{1} \tors \ct_{0} = \ca.
\]
\end{defi}

\begin{rema} \label{rem nfold implies n+1fold}
\begin{enumerate}[(1)]
	\item $1$-fold torsion-free classes of an abelian category $\ca$ are nothing but torsion-free classes of $\ca$. 
	\item Let $\ca$ be an abelian category and $i$ a nonnegative integer.
	Every $i$-fold torsion-free class of $\ca$ is $(i+1)$-fold torsion-free class of $\ca$.

Indeed, let $\cx$ a subcategory of $\ca$. Assume that $\cx$ is an $i$-fold torsion-free class of $\ca$. 
Then there exist subcategories $\cf_{0}, \cf_{1}, \cdots, \cf_{i}$ of $\ca$ such that 
\[
\cx = \cf_{i} \torf \cdots \torf \cf_{1} \torf \cf_{0} = \ca.
\] 
By putting $\cf_{i+1} := \cx$, it follows that $\cx$ is an $(i+1)$-fold torsion-free class of $\ca$.
\end{enumerate}
\end{rema}

Let $\ca$ be an abelian category. For all positive integers $i$, we can define the functors $\mathrm{Ext}_{\ca}^{i}(-,-)$ without projectives and injectives; see, for example, \cite[Chapter VII]{mitchell1965theory}.
Let $\cc$ be a subcategory of $\ca$.
For any nonnegative integer $j$, we let 
\[
\cc^{\perp_{j}} := \{M \in \ca \mid \Ext_{\ca}^{j}(\cc ,M) = 0\},
\] 
\[
{}^{\perp_{j}}\cc := \{M \in \ca \mid \Ext_{\ca}^{j}(M ,\cc) = 0\}.
\] 
We can construct a lot of examples of $n$-fold torsion(-free) classes using the following proposition:

\begin{prop} \label{prop how to make nfold tors}
  Let $\ca$ be an abelian category and $\cx_{n} \subset \cx_{n-1} \subset \cdots \subset \cx_{1} \subset \ca$ a chain of subcategories of $\ca$, where $n$ is a positive integer.
  The following hold$\colon$
  \begin{enumerate}[$(a)$]
    \setlength{\itemsep}{0pt}
    \item For each $i \in \{ 1, \cdots , n\}$, put $\ct_{i}$ as follows:
    \[
		\ct_{i} := {}^{\perp_{0}}\cx_{1} \cap {}^{\perp_{1}}\cx_{2} \cap \cdots \cap {}^{\perp_{(i-1)}}\cx_{i}.
    \]
    Then we have 
    \[
    \ct_{n} \tors \ct_{n-1} \tors \cdots \tors \ct_{1} \tors \ca.
    \]
    Thus $\ct_{i}$ is an $i$-fold torsion class of $\ca$ for every $i = 1, \cdots , n$.
    \item For each $i \in \{ 1, \cdots , n\}$, put $\cf_{i}$ as follows:
    \[
			\cf_{i} := \cx_{1}^{\perp_{0}} \cap \cx_{2}^{\perp_{1}} \cap \cdots \cap \cx_{i}^{\perp_{(i-1)}}.
		\]
    Then we have 
    \[
    \cf_{n} \torf \cf_{n-1} \torf \cdots \torf \cf_{1} \torf \ca.
    \]
    Thus $\cf_{i}$ is an $i$-fold torsion-free class of $\ca$ for every $i = 1, \cdots , n$.
  \end{enumerate}
\end{prop}

\begin{proof}
  $(a)$. Put $\ct_{0} := \ca$ and show that $\ct_{i+1}$ is a torsion class of $\ct_{i}$ for every $i \in \{ 0, 1, \cdots , n-1\}$.
  Take any conflation $0 \to T' \to E \to T'' \to 0$ in $\ct_{i}$ such that both $T'$ and $T''$ belong to $\ct_{i+1}$ and any object $X$ of $\cx_{i+1}$.
  It follows from the definition of $\ct_{i+1}$ that $\Ext_{\ca}^{i}(T', X) = \Ext_{\ca}^{i}(T'', X) = 0$ and hence $\Ext_{\ca}^{i}(E, X) = 0$ holds by considering the exact sequence induced by the functor $\Ext_{\ca}^{i}(-, X)$.
  Because $E \in \ct_{i}$, we have $E \in \ct_{i+1}$, which means that $\ct_{i+1}$ is closed under conflations in $\ct_{i}$. 

  It remains to show that $\ct_{i+1}$ is closed under admissible quotients in $\ct_{i}$. Let $0 \to K \to T \to M \to 0$ be a conflation in $\ct_{i}$ with $T \in \ct_{i+1}$ and $X$ an object of $\cx_{i+1}$.
  Since $X \in \cx_{i+1} \subset \cx_{i}$ and $K \in \ct_{i}$, it holds that $\Ext^{i-1}_{\ca}(K, X) = 0$. By the definition of $\ct_{i+1}$, $\Ext^{i}_{\ca}(T, X) = 0$ holds. Thus we have $\Ext^{i}_{\ca}(M, X) = 0$ and hence we have $M \in \ct_{i+1}$. 
	This completes the proof. 
  One can show $(b)$, dually.
\end{proof}

For any subcategory of an abelian category, we can construct a canonical $n$-fold torsion(-free) class determined by the subcategory as follows:
\begin{defi}
Let $\ca$ be an abelian category and $\cx$ its subcategory. We set $\F{0}{\cx} := \ca$ and $\F{n+1}{\cx} := \clos{\cx}{\tf}{\F{n}{\cx}}$ for $n \geq 0$. Then 
\[
\cdots \torf \F{n+1}{\cx} \torf \F{n}{\cx} \torf \cdots \torf \F{1}{\cx} \torf \ca (=\F{0}{\cx})
\]holds. 
Each $\F{i}{\cx}$ is, by definition, an $i$-fold torsion-free class of $\ca$ and is called \textit{the $i$-fold torsion-free closure} of $\cx$.
Dually, we can define \textit{the $i$-fold torsion closure} $\T{i}{\cx}$ of $\cx$.
\end{defi}

One may wonder whether calling it the $i$-fold torsion-free ``closure'' is appropriate or not. 
This will be justified in Proposition \ref{prop nfold torf closure}. 

\begin{rema} \label{rem torf clos}
Let $\ca$ be an abelian category and $\cf$ its subcategory. By Proposition \ref{prop clos}, we have
\[
\F{1}{\cx} = \clos{\clos{\cx}{\sub}{}}{\ext}{} \text{ and } \T{1}{\cx} = \clos{\clos{\cx}{\quot}{}}{\ext}{}.
\]
\end{rema}

In the proof of Proposition \ref{prop charc nfold torf}, the following lemma plays an important role.

\begin{lemm}\label{lem torf trans}
Let $\ca$ be an abelian category, $\ce$ its extension-closed subcategory, $\cf$ a subcategory of $\ce$ which is closed under conflations in $\ce$ 
and $\cx$ a subcategory of $\cf$.
If $\cx$ is a torsion-free class of $\ce$, then  $\cx$ is a torsion-free class of $\cf$.
\end{lemm}
\begin{proof}
It is straightforward.
\end{proof}

\begin{prop}\label{prop nfold torf closure}
Let $\ca$ be an abelian category and $\cx$ its subcategory. For a nonnegative integer $n$, 
the $n$-fold torsion-free closure $\F{n}{\cx}$ of $\cx$ is the smallest $n$-fold torsion-free class containing $\cx$.
\end{prop}

\begin{proof}
	We show it by the induction on $n$. When $n=0$, it is clear by the definition of $\F{0}{\cx}$.
	Suppose that $n \geq 1$. Take any $n$-fold torsion-free class $\cf$ of $\ca$ such that $\cx \subset \cf$.
	By the definition of $n$-fold torsion-free classes, there exists a chain of subcategories of $\ca$
	\[
	\cf = \cf_{n} \torf \cdots \torf \cf_{1} \torf \cf_{0} = \ca.
	\]
	By the induction hypothesis, we have that $\F{n-1}{\cx} \subset \cf_{n-1}$.
	Hence it holds that $\clos{\cx}{\tf}{\F{n-1}{\cx}} \subset \clos{\cx}{\tf}{\cf_{n-1}}$.
	Since $\cf_{n}$ contains $\cx$ and is a torsion class of $\cf_{n-1}$, we obtain $\clos{\cx}{\tf}{\cf_{n-1}} \subset \cf_{n}$.
	In conclusion, we have $\F{n}{\cx} = \clos{\cx}{\tf}{\F{n-1}{\cx}} \subset \cf_{n} = \cf$.
\end{proof}

Now, we state and prove an equivalent condition for a subcategory to be an $n$-fold torsion-free class.

\begin{prop} \label{prop charc nfold torf}
Let $\ca$ be an abelian category, $\cx$ its subcategory and $n$ a positive integer. Then the following are equivalent:
\begin{enumerate}[$(a)$]
\setlength{\itemsep}{0pt} 

	\item $\cx$ is a torsion-free class of $\F{n-1}{\cx}$.

	\item $\cx$ is an $n$-fold torsion-free class of $\ca$.

\end{enumerate}
\end{prop} 

\begin{proof}
\underline{$(a) \Longrightarrow (b)$.} It is obvious by the definition of $n$-fold torsion-free classes.

\underline{$(b) \Longrightarrow (a)$.} Assume that $\cx$ is an $n$-fold torsion-free class. Then, by the definition of $n$-fold torsion-freeness, there exists a chain of subcategories 
\[
\cx = \cf_{n} \torf \cdots \torf \cf_{1} \torf \cf_{0} = \ca.
\]
By Proposition \ref{prop nfold torf closure}, we have $\cx \subset \F{n-1}{\cx} \subset \cf_{n-1}$.
Since $\cx \torf \cf_{n-1}$ holds, by Lemma \ref{lem torf trans}, $\cx$ is a torsion-free class of $\F{n-1}{\cx}$.
\end{proof}

In the end of this section, we show that being an $n$-fold torsion-free class implies being a $\K^{(n-1)}\E$-closed subcategory. 
The proof of the following proposition is inspired by the proof given in \cite[Proposition 3.1, (iii) $\Longrightarrow$ (i)]{kobayashi2024ke}.

\begin{prop} \label{prop nfold vs Kn-1Eclosed}
Let $\ca$ be an abelian category and $n$ a positive integer. 
Every $n$-fold torsion-free class of $\ca$ is a $\K^{(n-1)}\E$-closed subcategory of $\ca$.
\end{prop}

\begin{proof}
Let $\cx$ be an $n$-fold torsion-free class of $\ca$. 
Since $\cx$ is closed under extensions in $\ca$ by Lemma \ref{lem extension closed}, it suffices to show that $\cx$ is closed under $(n-1)$-kernels in $\ca$. Take  any exact sequence in $\ca$
\[
0 \to M \xrightarrow{d^{-1}} X^{0} \xrightarrow{d^{0}} \cdots \xrightarrow{d^{n-3}} X^{n-2} \xrightarrow{d^{n-2}}  X^{n-1},
\]
where each $X^{i}$ belongs to $\cx$. We obtain the following $n$ short exact sequences in $\ca$:

\[
\begin{tikzpicture}[auto]
  \node (51) at (-3.6, 1.7) {$0$}; \node (52) at (-1.8, 1.7) {$M$};  \node (53) at (0.0, 1.7) {$X^{0}$};  \node (54) at (1.8, 1.7) {$\Ima d^{0}$};  \node (55) at (3.6, 1.7) {$0$}; \node at (3.75, 1.6) {,};
  \node (41) at (-3.6, 0.7) {$0$}; \node (42) at (-1.8, 0.7) {$\Ima d^{0}$};  \node (43) at (0.0, 0.7) {$X^{1}$};  \node (44) at (1.8, 0.7) {$\Ima d^{1}$};  \node (45) at (3.6, 0.7) {$0$}; \node at (3.75, 0.6) {,};
  \node (33) at (0,0) {$\vdots$};
  \node (21) at (-3.6, -0.7) {$0$}; \node (22) at (-1.6, -0.7) {$\Ima d^{n-3}$};  \node (23) at (0.2, -0.7) {$X^{n-2}$};  \node (24) at (2.0, -0.7) {$\Ima d^{n-2}$};  \node (25) at (3.6, -0.7) {$0$}; \node at (3.75, -0.8) {,};
  \node (11) at (-3.6, -1.7) {$0$}; \node (12) at (-1.6, -1.7) {$\Ima d^{n-2}$};  \node (13) at (0.2, -1.7) {$X^{n-1}$};  \node (14) at (2.0, -1.7) {$\Cok d^{n-2}$};  \node (15) at (3.6, -1.7) {$0$}; \node at (3.75, -1.8) {.};

  \draw[->, thick] (51) -- (52); \draw[->, thick] (52) -- (53); \draw[->, thick] (53) -- (54); \draw[->, thick] (54) -- (55);
  \draw[->, thick] (41) -- (42); \draw[->, thick] (42) -- (43); \draw[->, thick] (43) -- (44); \draw[->, thick] (44) -- (45);
  \draw[->, thick] (21) -- (22); \draw[->, thick] (22) -- (23); \draw[->, thick] (23) -- (24); \draw[->, thick] (24) -- (25);
  \draw[->, thick] (11) -- (12); \draw[->, thick] (12) -- (13); \draw[->, thick] (13) -- (14); \draw[->, thick] (14) -- (15);

\end{tikzpicture}
\]
By the definition of $n$-fold torsion-free classes, there exists a chain of subcategories
\[
\cx = \cf_{n} \torf \cdots \torf \cf_{1} \torf \cf_{0} = \ca.
\]
Then we have $M, \Ima d^{j} \in \cf_{1}$ for each $j \in \{0, 1, \cdots, n-2 \}$ because each $X^{i}$ belongs to $\cx \subset \cf_{1}$ and $\cf_{1}$ is closed under subobjects. 
Thus the first $n-1$ short exact sequences are conflations in $\cf_{1}$. 
Next, we have $M, \Ima d^{l} \in \cf_{2}$ for each $l \in \{0, 1, \cdots, n-3 \}$ because each $X^{i}$ belongs to $\cx \subset \cf_{2}$ and $\cf_{2}$ is closed under admissible subobjects in $\cf_{1}$. 
Thus the first $n-2$ short exact sequences are conflations in $\cf_{2}$. 
Inductively, we can show that the first $n-m$ short exact sequences are conflations in $\cf_{m}$ for $1 \leq m \leq n-1$. 
This implies $M \in \cf_{n} = \cx$ since $\cx$ is closed under admissible subobjects in $\cf_{n-1}$.
\end{proof}

It is natural to ask whether the converse holds or not.
\begin{ques} \label{ques the converse}
Let $\ca$ be an abelian category. Suppose $\cx$ is a $\K^{(n-1)}\E$-closed subcategory of $\ca$.  Is $\cx$ an $n$-fold torsion-free class of $\ca$?
\end{ques}

Kobayashi-Saito proved that it is true when $n=2$ in \cite[Proposition 3.1]{kobayashi2024ke}. We do not know this holds in general. However, it holds for a certain case; see Theorem \ref{thm resolving}.


\section{2-fold torsion-free classes and kernel closures} \label{sec 2fold torsionfree classes and kernel closures}

The aim of this section is to prove that Question \ref{ques the converse} is true when $n=2$.
It has already been proved in \cite{kobayashi2024ke}, but we give a different proof. First, we introduce several notations in order to describe the kernel closure of an additive subcategory of an abelian category.

\begin{defi} \label{def ker}
Let $\ca$ be an abelian category and $\cx$ and $\cy$ its additive subcategories.
\begin{enumerate}[(1)]
\setlength{\itemsep}{0pt} 
	\item We set 

$\Ker(\cx, \cy) := \big\{ K \in \ca \ \mid \text{there exists an exact sequence }0 \to K \to X \to Y\text{ in }\ca \\
\hfill \text{with } X \in \cx \text{ and } Y \in \cy \big\}$.

This is an additive subcategory of $\ca$.

	\item We set $\Ker^{0}(\cx) := \cx$ and $\Ker^{n+1}(\cx) := \Ker(\Ker^{n}(\cx), \Ker^{n}(\cx))$, inductively. Since each $\Ker^{n}(\cx)$ contains the zero object, we have $\Ker ^{n}(\cx) \subset \Ker (\Ker ^{n}(\cx), 0) \subset \Ker ^{n+1}(\cx)$, that is, 
\[
\cx =  \Ker^{0}(\cx) \subset \Ker^{1}(\cx) \subset \cdots \subset \Ker^{n}(\cx) \subset \cdots .
\]

	\item We put $\Ker^{0}_{\cy}(\cx) := \cx$ and $\Ker^{n+1}_{\cy}(\cx) := \Ker(\Ker^{n}_{\cy}(\cx), \cy)$, inductively. Since $\cy$ contains the zero object, it holds that $\Ker ^{n}_{\cy}(\cx) \subset \Ker (\Ker ^{n}_{\cy}(\cx), 0) \subset \Ker ^{n+1}_{\cy}(\cx)$, that is, 
\[ 
\cx =  \Ker^{0}_{\cy}(\cx) \subset \Ker_{\cy}^{1}(\cx) \subset \cdots \subset \Ker_{\cy}^{n}(\cx) \subset \cdots .
\]

\end{enumerate}
\end{defi}

\begin{rema} \label{rem inc}
Let $\cx$ be an additive subcategory of an abelian category $\ca$. 
By the definition of $\Ker_{\cx}^{n}(\cx)$ and $\Ker^{n}(\cx)$, we have the following inclusions:

\[
\begin{tikzpicture}[auto]
\node  at (0.5, 0) {$\cx$}; \node at (1.0, 0) {$=$}; \node  at (2.0, 0) {$\Ker^{0}_{\cx}(\cx)$}; \node at (3.0, 0) {$\subset$}; \node at (4.0, 0) {$\Ker^{1}_{\cx}(\cx)$}; \node at (5.0, 0) {$\subset$}; \node at (6.0, 0) {$\Ker^{2}_{\cx}(\cx)$}; \node at (7.0, 0) {$\subset$}; \node at (7.5, 0) {$\cdots$}; \node at (8.0, 0) {$\subset$}; \node at (9.0, 0) {$\Ker^{n}_{\cx}(\cx)$}; \node at (10.0, 0) {$\subset$}; \node at (10.5, 0) {$\cdots$};

\node[rotate=270]  at (0.5, -0.7) {\large $=$}; \node[rotate=270] at (2.0, -0.7) {\large $=$}; \node[rotate=270] at (4.0, -0.7) {\large $=$}; \node[rotate=270] at (6.0, -0.7) {$\subset$}; \node[rotate=270] at (6.0, -0.7) {$\subset$}; \node[rotate=270] at (9.0, -0.7) {$\subset$};

\node  at (0.5, -1.4) {$\cx$}; \node at (1.0, -1.4) {$=$}; \node  at (2.0, -1.4) {$\Ker^{0}(\cx)$}; \node at (3.0, -1.4) {$\subset$}; \node at (4.0,-1.4) {$\Ker^{1}(\cx)$}; \node at (5.0, -1.4) {$\subset$}; \node at (6.0,-1.4) {$\Ker^{2}(\cx)$}; \node at (7.0, -1.4) {$\subset$}; \node at (7.5, -1.4) {$\cdots$}; \node at (8.0, -1.4) {$\subset$}; \node at (9.0, -1.4) {$\Ker^{n}(\cx)$}; \node at (10.0,-1.4) {$\subset$}; \node at (10.5, -1.4) {$\cdots$};
\node at (10.8, -1.5) {.};
\end{tikzpicture}
\]
\end{rema}

Next, we introduce some class of certain subcategories which are closed under these new operation.
\begin{defi} \label{def kery closed}
Let $\ca$ be an abelian category and $\cx$ and $\cy$ its additive subcategories. 
Then $\cx$ is said to be \textit{closed under kernels of morphisms to $\cy$} (or \textit{$\Ker_{\cy}$-closed}) if for any morphism $f \colon X \to Y$ in $\ca$ with $X \in \cx$ and $Y \in \cy$, we obtain $\Ker f \in \cx$ .
 We denote by $\clos{\cx}{\kr_{\cy}}{}$ the smallest $\mathrm{Ker}_{\cy}$-closed subcategory of $\ca$ containing $\cx$.
\end{defi}

\begin{prop} \label{prop ker closure}
Let $\ca$ be an abelian category and $\cx$ and $\cy$ its additive subcategories.
\begin{enumerate}[$(a)$]
\setlength{\itemsep}{0pt} 
	\item we have \[ \clos{\cx}{\kr}{} = \bigcup_{n \ge 0} \Ker^{n}(\cx).\]
	\item we obtain \[ \clos{\cx}{\kr_{\cy}}{} = \bigcup_{n \ge 0} \Ker^{n}_{\cy}(\cx).\]
\end{enumerate}
\end{prop}

\begin{proof}
\underline{$(a)$.}  Since $\cx \subset \clos{\cx}{\kr}{}$ and $\clos{\cx}{\kr}{}$ is closed under kernels, $\clos{\cx}{\kr}{}$ contains $\Ker^{1}(\cx)$. 
By the definition of $\Ker^{n}(\cx)$,  for each nonnegative integer $n$, we have $\Ker^{n}(\cx) \subset \clos{\cx}{\kr}{}$ inductively. 

It remains to show that $\cm := \bigcup_{n \ge 0} \Ker^{n}(\cx)$ is closed under kernels. 
Take any morphism $f \colon M \to N$ with $M, N \in \cm$. There exist nonnegative integers $m$ and $n$ such that $M \in \Ker^{m}(\cx)$ and $N \in \Ker^{n}(\cx)$. We put $l := \max\{m, n\}$. Both $M$ and $N$ belong to $\Ker^{l}(\cx)$ and hence we obtain $\Ker f \in \Ker^{l+1}(\cx) \subset \cm$. We have the desired result. 

\underline{$(b)$.} One can show it in the same way as above.
\end{proof}

\begin{lemm} \label{cor ker vs kerx}
Let $\cx$ be an additive subcategory of an abelian category $\ca$. Then the following holds$\colon$
\[
\clos{\cx}{\kr_{\cx}}{} \subset \clos{\cx}{\kr}{}.
\]  
\end{lemm}

\begin{proof}
It follows from Remark \ref{rem inc} and Proposition \ref{prop ker closure}.
\end{proof}

In fact, the two closures in Lemma \ref{cor ker vs kerx} coincide. 
In order to prove it,  we show the following key proposition.

\begin{prop} \label{prop kerx = adsub}
Let $\ca$ be an abelian category, $\cx$ its additive subcategory and $M \in \ca$. 
For a nonnegative integer $n$, an object $M$ of $\ca$ belongs to $\Ker^{n}_{\cx}(\cx)$ if and only if there exists a short exact sequence $0 \to M \to X \to F \to 0$ in $\ca$ with $X \in \cx$ and $F \in (\clos{\cx}{\sub}{})^{*n}$.
\end{prop}

\begin{proof}
We show it by the induction on $n$.
Suppose that $n = 0$. Then since $\Ker^{0}_{\cx}(\cx) = \cx$ and  $(\clos{\cx}{\sub}{})^{*0} = 0$, the statement is clear. Next, suppose that $n \geq 1$.

(``if''part): Assume that there exists a short exact sequence $0 \to M \to X \to F \to 0$ in $\ca$ with $X \in \cx$ 
and $F \in (\clos{\cx}{\sub}{})^{*n}$. By the definition of $*$, there exists a short exact sequences $0 \to F' \stackrel{f}{\to} F \to F'' \to 0$ in $\ca$ such that $F' \in \clos{\cx}{\sub}{}$ and $F'' \in (\clos{\cx}{\sub}{})^{*(n-1)}$. 
Therefore, by pulling back the exact sequence $0 \to M \to X \to F \to 0$ by $f$, we have the following commutative diagram with exact rows and exact columns:

\[
\begin{tikzpicture}[auto]
\node (01) at (-2.4, 0) {$0$}; \node (a1) at (-1.2, 0) {$M$}; \node (x) at (0, 0) {$X$}; \node (y) at (1.2, 0) {$F$}; \node (02) at (2.4, 0) {$0$};
\node (021) at (-2.4, 1.2) {$0$}; \node (a2) at (-1.2, 1.2) {$M$}; \node (x1) at (0, 1.2) {$M'$}; \node (y1) at (1.2, 1.2) {$F'$}; \node (022) at (2.4, 1.2) {$0$};
\node (031) at (0, 2.4) {$0$}; \node (032) at (1.2, 2.4) {$0$};
\node (y2d) at (0, -1.2) {$F''$}; \node (y2) at (1.2, -1.2) {$F''$};
\node (001) at (0, -2.4) {$0$}; \node (002) at (1.2, -2.4) {$0$};
\node at (0.6, 0.6) {PB};
\node at (1.35, -2.5) {,};

\draw[->, thick] (01) to (a1); \draw[->, thick] (a1) to (x); \draw[->, thick] (x) to (y); \draw[->, thick] (y) to (02);
\draw[->, thick] (021) to (a2); \draw[->, thick] (a2) to (x1); \draw[->, thick] (x1) to (y1); \draw[->, thick] (y1) to (022);
\draw[double distance = 2pt, thick] (y2d) -- (y2);

\draw[->, thick] (031) to (x1); \draw[->, thick] (x1) to (x); \draw[->, thick] (x) to (y2d); \draw[->, thick]  (y2d) to (001);
\draw[->, thick] (032) to (y1); \draw[->, thick] (y1) -- node {$f$} (y); \draw[->, thick] (y) to (y2); \draw[->, thick]  (y2) to (002);
\draw[double distance = 2pt, thick] (a1) -- (a2);

\end{tikzpicture}
\]
where $M'$ is an object of $\ca$. By the exact sequence $0 \to M' \to X \to F'' \to 0$ and the induction hypothesis, we obtain $M' \in \Ker^{n-1}_{\cx}(\cx)$. By the exact sequence  $0 \to M \to M' \to F' \to 0$, we have $M \in \Ker^{n}_{\cx}(\cx)$.

(``only if''part): Assume that $M \in \Ker^{n}_{\cx}(\cx)$. 
By the definition of $\Ker^{n}_{\cx}(\cx)$, there exists a short exact sequence $0 \to M \to M' \to F' \to 0$ in $\ca$ with $M' \in \Ker^{n-1}_{\cx}(\cx)$ and $F' \in \clos{\cx}{\sub}{}$. 
By the induction hypothesis, there exists some short exact sequence $0 \to M' \stackrel{g}{\to} X \to F'' \to 0$ in $\ca$ such that $X \in \cx$ and $F'' \in  (\clos{\cx}{\sub}{})^{*(n-1)}$. 
By the pushing out the exact sequence  $0 \to M \to M' \to F' \to 0$ by $g$, we obtain the following commutative diagram with exact rows and exact columns:

\[
\begin{tikzpicture}[auto]
\node (01) at (-2.4, 0) {$0$}; \node (a1) at (-1.2, 0) {$M$}; \node (x) at (0, 0) {$X$}; \node (y) at (1.2, 0) {$F$}; \node (02) at (2.4, 0) {$0$};
\node (021) at (-2.4, 1.2) {$0$}; \node (a2) at (-1.2, 1.2) {$M$}; \node (x1) at (0, 1.2) {$M'$}; \node (y1) at (1.2, 1.2) {$F'$}; \node (022) at (2.4, 1.2) {$0$};
\node (031) at (0, 2.4) {$0$}; \node (032) at (1.2, 2.4) {$0$};
\node (y2d) at (0, -1.2) {$F''$}; \node (y2) at (1.2, -1.2) {$F''$};
\node (001) at (0, -2.4) {$0$}; \node (002) at (1.2, -2.4) {$0$};
\node at (0.6, 0.6) {PO};
\node at (1.35, -2.5) {,};

\draw[->, thick] (01) to (a1); \draw[->, thick] (a1) to (x); \draw[->, thick] (x) to (y); \draw[->, thick] (y) to (02);
\draw[->, thick] (021) to (a2); \draw[->, thick] (a2) to (x1); \draw[->, thick] (x1) to (y1); \draw[->, thick] (y1) to (022);
\draw[double distance = 2pt, thick] (y2d) -- (y2);

\draw[->, thick] (031) to (x1); \draw[->, thick] (x1) --node[swap] {$g$} (x); \draw[->, thick] (x) to (y2d); \draw[->, thick]  (y2d) to (001);
\draw[->, thick] (032) to (y1); \draw[->, thick] (y1) -- (y); \draw[->, thick] (y) to (y2); \draw[->, thick]  (y2) to (002);
\draw[double distance = 2pt, thick] (a1) -- (a2);

\end{tikzpicture}
\]
where $F$ is an object of $\ca$. By the exact sequence $0 \to F' \to F \to F'' \to 0$, we have $F \in (\clos{\cx}{\sub}{})^{*n}$. This completes the proof.
\end{proof}

\begin{corr} \label{cor kerx equal adsub}
Let $\ca$ be an abelian category and $\cx$ its additive subcategory. Then the following holds$\colon$
\[
\clos{\cx}{\kr_{\cx}}{} = \clos{\cx}{\ads}{\F{1}{\cx}}.
\]
\end{corr}

\begin{proof}
	By Proposition \ref{prop ker closure}, 
	it is enough to show that $M \in \Ker ^{n}_{\cx}(\cx)$ for some $n$ if and only if there exists a exact sequence $0 \to M \to X \to F \to 0$ in $\ca$ 
	with $X \in \cx$ and $F \in \F{1}{\cx}$. 
	By Remark \ref{rem torf clos}, it has already been proved in Proposition \ref{prop kerx = adsub}. 
\end{proof}

Finally, we show that the kernel closure of an additive subcategory of an abelian category coincides with the closures in the above corollary.

\begin{theo} \label{thm charac of kernel closure}
Let $\ca$ be an abelian category and $\cx$ its additive subcategory. Then the following three subcategories coincide$\colon$
\begin{enumerate}[$(a)$]
\setlength{\parskip}{0pt}
	\item $\clos{\cx}{\ads}{\F{1}{\cx}}$,
	\item $\clos{\cx}{\kr_{\cx}}{}$ and
	\item $\clos{\cx}{\kr}{}$.
\end{enumerate}
In particular, the following are equivalent$\colon$
\begin{enumerate}[$(a)$]
	\setlength{\parskip}{0pt}
		\item $\cx$ is an admissible-subobject-closed subcategory of $\F{1}{\cx} \mathrm{;}$
		\item $\cx$ is a $\mathrm{Ker}_{\cx}$-closed subcategory of $\ca \mathrm{;}$ and
		\item $\cx$ is a kernel-closed subcategory of $\ca$.
	\end{enumerate}
\end{theo}

\begin{proof}
By Lemma \ref{cor ker vs kerx} and Corollary \ref{cor kerx equal adsub}, it remains to show that $\clos{\cx}{\kr}{} \subset \clos{\cx}{\ads}{\F{1}{\cx}}$ holds. It suffices to show that $\clos{\cx}{\ads}{\F{1}{\cx}}$ is closed under kernels. Take any morphism $f \colon M \to M'$ in $\ca$ with $M,M' \in \clos{\cx}{\ads}{\F{1}{\cx}}$. We have the following two short exact sequences:
\[
0 \to \Ker f \to M \to \Ima f \to 0 \text{ and } 0 \to \Ima f \to M' \to \Cok f \to 0.
\]
Since both $M$ and $M'$ belong to $\clos{\cx}{\ads}{\F{1}{\cx}}$, there exist some short exact sequences in $\ca$
\[
0 \to M \stackrel{i}{\to} X \to F \to 0 \text{ and } 0 \to M' \stackrel{i'}{\to} X' \to F' \to 0,
\]
where $X, X' \in \cx$ and $F, F' \in \F{1}{\cx}$. 
By pushing out the exact sequence $0 \to \Ker f \to M \to \Ima f \to 0$ by $i$ and the exact sequence $0 \to \Ima f \to M' \to \Cok f \to 0$ by $i'$, we obtain the following commutative diagrams with exact rows and exact columns:

\[
\begin{tikzpicture}[auto]
\node (01) at (-3.6, 0) {$0$}; \node (a1) at (-1.8, 0) {$\Ker f$}; \node (x) at (0, 0) {$X$}; \node (y) at (1.8, 0) {$N$}; \node (02) at (3.6, 0) {$0$};
\node (021) at (-3.6, 1.8) {$0$}; \node (a2) at (-1.8, 1.8) {$\Ker f$}; \node (x1) at (0, 1.8) {$M$}; \node (y1) at (1.8, 1.8) {$\Ima f$}; \node (022) at (3.6, 1.8) {$0$};
\node (031) at (0, 3.6) {$0$}; \node (032) at (1.8, 3.6) {$0$};
\node (y2d) at (0, -1.8) {$F$}; \node (y2) at (1.8, -1.8) {$F$};
\node (001) at (0, -3.6) {$0$}; \node (002) at (1.8, -3.6) {$0$};
\node at (0.9, 0.9) {\large PO};
\node at (4.05, 0.9) {and};

\node (01') at (4.5, 0) {$0$}; \node (a1') at (6.3, 0) {$\Ima f$}; \node (x') at (8.1, 0) {$X'$}; \node (y') at (9.9, 0) {$N'$}; \node (02') at (11.7, 0) {$0$};
\node (021') at (4.5, 1.8) {$0$}; \node (a2') at (6.3, 1.8) {$\Ima f$}; \node (x1') at (8.1, 1.8) {$M'$}; \node (y1') at (9.9, 1.8) {$\Cok f$}; \node (022') at (11.7, 1.8) {$0$};
\node (031') at (8.1, 3.6) {$0$}; \node (032') at (9.9, 3.6) {$0$};
\node (y2d') at (8.1, -1.8) {$F'$}; \node (y2') at (9.9, -1.8) {$F'$};
\node (001') at (8.1, -3.6) {$0$}; \node (002') at (9.9, -3.6) {$0$};
\node at (9.0, 0.9) {\large PO};
\node at (10.1, -3.7) {.};

\draw[->, thick] (01) to (a1); \draw[->, thick] (a1) to (x); \draw[->, thick] (x) to (y); \draw[->, thick] (y) to (02);
\draw[->, thick] (021) to (a2); \draw[->, thick] (a2) to (x1); \draw[->, thick] (x1) to (y1); \draw[->, thick] (y1) to (022);
\draw[double distance = 2pt, thick] (y2d) -- (y2);

\draw[->, thick] (01') to (a1'); \draw[->, thick] (a1') to (x'); \draw[->, thick] (x') to (y'); \draw[->, thick] (y') to (02');
\draw[->, thick] (021') to (a2'); \draw[->, thick] (a2') to (x1'); \draw[->, thick] (x1') to (y1'); \draw[->, thick] (y1') to (022');
\draw[double distance = 2pt, thick] (y2d') -- (y2');

\draw[->, thick] (031) to (x1); \draw[->, thick] (x1) --node[swap] {$i$} (x); \draw[->, thick] (x) to (y2d); \draw[->, thick]  (y2d) to (001);
\draw[->, thick] (032) to (y1); \draw[->, thick] (y1) --node{$j$} (y); \draw[->, thick] (y) to (y2); \draw[->, thick]  (y2) to (002);
\draw[double distance = 2pt, thick] (a1) -- (a2);

\draw[->, thick] (031') to (x1'); \draw[->, thick] (x1') --node[swap] {$i'$} (x'); \draw[->, thick] (x') to (y2d'); \draw[->, thick]  (y2d') to (001');
\draw[->, thick] (032') to (y1'); \draw[->, thick] (y1') -- (y'); \draw[->, thick] (y') to (y2'); \draw[->, thick]  (y2') to (002');
\draw[double distance = 2pt, thick] (a1') -- (a2');
\end{tikzpicture}
\]
 Moreover, by pushing out the exact sequence $0 \to \Ima f \to X' \to N' \to 0$ by $j$, we have the following commutative diagram with exact rows and exact columns:

\[
\begin{tikzpicture}[auto]
\node (01) at (-3.6, 0) {$0$}; \node (a1) at (-1.8, 0) {$N$}; \node (x) at (0, 0) {$L$}; \node (y) at (1.8, 0) {$N'$}; \node (02) at (3.6, 0) {$0$};
\node (021) at (-3.6, 1.8) {$0$}; \node (a2) at (-1.8, 1.8) {$\Ima f$}; \node (x1) at (0, 1.8) {$X'$}; \node (y1) at (1.8, 1.8) {$N'$}; \node (022) at (3.6, 1.8) {$0$};
\node (031) at (-1.8, 3.6) {$0$}; \node (032) at (0, 3.6) {$0$};
\node (y2d) at (-1.8, -1.8) {$F$}; \node (y2) at (0, -1.8) {$F$};
\node (001) at (-1.8, -3.6) {$0$}; \node (002) at (0, -3.6) {$0$};
\node at (-0.9, 0.9) {\large PO};
\node at (0.2, -3.7) {.};

\draw[->, thick] (01) to (a1); \draw[->, thick] (a1) to (x); \draw[->, thick] (x) to (y); \draw[->, thick] (y) to (02);
\draw[->, thick] (021) to (a2); \draw[->, thick] (a2) to (x1); \draw[->, thick] (x1) to (y1); \draw[->, thick] (y1) to (022);
\draw[double distance = 2pt, thick] (y2d) -- (y2);

\draw[->, thick] (031) to (a2); \draw[->, thick] (x1) to (x); \draw[->, thick] (a1) to (y2d); \draw[->, thick]  (y2d) to (001);
\draw[->, thick] (032) to (x1); \draw[->, thick] (a2) --node[swap] {$j$} (a1); \draw[->, thick] (x) to (y2); \draw[->, thick]  (y2) to (002);
\draw[double distance = 2pt, thick] (y1) -- (y);

\end{tikzpicture}
\]
Since $X' \in \cx \subset \F{1}{\cx}$, $F \in \F{1}{\cx}$ and $\F{1}{\cx}$ is a torsion-free class of $\ca$, by the exact sequence $0 \to X' \to L \to F \to 0$, $L$ belongs to $\F{1}{\cx}$ and so does its subobject $N$. 
Hence, by the exact sequence $0 \to \Ker f \to X \to N \to 0$, we have $\Ker f \in \clos{\cx}{\ads}{\F{1}{\cx}}$.
\end{proof}

With Theorem \ref{thm charac of kernel closure}, we can give another proof to Kobayashi-Saito's result.

\begin{theo}[{\cite[Proposition 3.1]{kobayashi2024ke}}] \label{thm KE=2torf}
The following are equivalent for a subcategory $\cx$ of an abelian category $\ca$.
\begin{enumerate}[$(a)$]
\setlength{\parskip}{0pt}
	\item $\cx$ is a $\KE$-closed subcategory of $\ca$.
	\item $\cx$ is a torsion-free class of $\F{1}{\cx}$.
	\item $\cx$ is a $2$-fold torsion-free class of $\ca$.
\end{enumerate}
\end{theo}

\begin{proof}
The implication $(b) \Longrightarrow (c)$ follows from Proposition \ref{prop charc nfold torf}. The implication $(c) \Longrightarrow (a)$ follows from Proposition \ref{prop nfold vs Kn-1Eclosed}. Hence 
it remains to show $(a) \Longrightarrow (b)$. Assume that $\cx$ is a $\KE$-closed subcategory of $\ca$. 
By Lemma \ref{lem extension closed}, $\cx$ is closed under conflations in $\F{1}{\cx}$. 
Since $\cx$ is closed under kernels, by Theorem \ref{thm charac of kernel closure}, $\cx = \clos{\cx}{\kr}{} = \clos{\cx}{\ads}{\F{1}{\cx}}$ holds, which means that $\cx$ is closed under admissible subobjects of $\F{1}{\cx}$. 
Therefore, $\cx$ is a torsion-free class of $\F{1}{\cx}$.
\end{proof}

The following corollary follows from Theorem \ref{thm KE=2torf}.
\begin{corr}
Let $\ca$ be an abelian category and $\cx$ its subcategory. Then the $\KE$ closure of $\cx$ coincides with the $2$-fold torsion-free closure of $\cx$, that is, $\clos{\cx}{\KE}{} = \F{2}{\cx}$.
\end{corr}

\begin{proof}
	($\subset$): By Theorem \ref{thm KE=2torf}, $\F{2}{\cx}$ is a $\KE$-closed subcategory of $\ca$ containing $\cx$ 
	and hence we have $\clos{\cx}{\KE}{} \subset \F{2}{\cx}$.

	($\supset$): By Theorem \ref{thm KE=2torf}, we obtain $\clos{\cx}{\KE}{} \torf \F{1}{\clos{\cx}{\KE}{}}$. Since $\F{1}{\cx}$ is a torsion-free class of $\ca$, it is also a $\KE$-closed subcategory of $\ca$.
	Hence it holds that $\clos{\cx}{\KE}{} \subset \F{1}{\cx} \subset \F{1}{\clos{\cx}{\KE}{}}$. By Lemma \ref{lem torf trans}, we have $\clos{\cx}{\KE}{} \torf \F{1}{\cx}$.
	Therefore, $\clos{\cx}{\KE}{} \supset \F{2}{\cx}$ holds.
\end{proof}

Since the opposite category of an abelian category is again abelian, we have the following results.

\begin{prop}[{\cite[Proposition 3.2]{kobayashi2024ke}}] \label{prop CE = 2tors}
The following are equivalent for a subcategory $\cx$ of an abelian category $\ca$.
\begin{enumerate}[$(a)$]
\setlength{\parskip}{0pt}
	\item $\cx$ is a $\CE$-closed subcategory of $\ca$.
	\item $\cx$ is a torsion class of $\T{1}{\cx}$.
	\item $\cx$ is a $2$-fold torsion class of $\ca$.
\end{enumerate}
\end{prop}

\begin{corr}
Let $\ca$ be an abelian category and $\cx$ its subcategory. Then the $\CE$ closure of $\cx$ coincides with the $2$-fold torsion closure of $\cx$, that is, $\clos{\cx}{\CE}{} = \T{2}{\cx}$.
\end{corr}

For $\IKE$-closed ($\ICE$-closed) subcategories, the following results are known. Propositions \ref{prop IKE = serre in torf} and \ref{prop ICE = serre in tors} will be used in Proposition \ref{prop 2fold = nfold}.

\begin{prop}[{\cite[Corollary 3.4]{kobayashi2024ke}}] \label{prop IKE = serre in torf}
The following are equivalent for a subcategory $\cx$ of an abelian category $\ca$.
\begin{enumerate}[$(a)$]
\setlength{\parskip}{0pt}
	\item $\cx$ is an $\IKE$-closed subcategory of $\ca$.
	\item $\cx$ is a Serre subcategory of $\F{1}{\cx}$.
	\item There exists a torsion-free class $\cf$ of $\ca$ such that $\cx$ is a Serre subcategory of $\cf$.
\end{enumerate}
\end{prop}

\begin{prop}[{\cite[Corollary 3.5]{kobayashi2024ke}}] \label{prop ICE = serre in tors}
The following are equivalent for a subcategory $\cx$ of an abelian category $\ca$.
\begin{enumerate}[$(a)$]
\setlength{\parskip}{0pt}
	\item $\cx$ is an $\ICE$-closed subcategory of $\ca$.
	\item $\cx$ is a Serre subcategory of $\T{1}{\cx}$.
	\item There exists a torsion class $\ct$ of $\ca$ such that $\cx$ is a Serre subcategory of $\ct$.
\end{enumerate}
\end{prop}


\section{$n$-fold torsion(-free) classes from the viewpoint of hereditary cotorsion pairs} \label{sec nfold torsionfree classes from the viewpoint of hereditary cotorsion pairs}
The goal of this section is to prove that Question \ref{ques the converse} is true for the left subcategory of a hereditary cotorsion pair.
First, we recall hereditary cotorsion pairs.
\begin{defi}
Let $\ca$ be an abelian category. A \textit{cotorsion pair} in $\ca$ is a pair $(\cx, \cy)$ of subcategories of $\ca$ such that 
\[
\cx = \{ M \in \ca \mid \Ext^{1}_{\ca}(M, Y) = 0 \text{ for any } Y \in \cy \},
\]
\[
\cy = \{ M \in \ca \mid \Ext^{1}_{\ca}(\cx, M) = 0 \text{ for any } X \in \cx \}.
\]
A cotorsion pair $(\cx, \cy)$ in $\ca$ is \textit{hereditary} if $\Ext^{>0}_{\ca}(X, Y) = 0$ for all $X \in \cx$ and $Y \in \cy$.
\end{defi}

There are a lot of examples of hereditary cotorsion pairs.
\begin{exam} \label{ex cotorsion pair}
Let $R$ be a ring.
\begin{enumerate}[(1)]
	\item We denote by $\Proj R$ ($\resp$ $\Inj R$) the subcategory of all projective ($\resp$ injective) modules. 
	Then $(\Proj R, \Mod R)$ and $(\Mod R, \Inj R)$ are hereditary cotorsion pairs in $\Mod R$.
	\item A right $R$-module $M$ is called \textit{cotorsion} if $\Ext^{1}_{R}(\Flat R, M) = 0$, where $\Flat R$ is the subcategory of flat modules. We denote by $\Cot R$ the subcategory of cotorsion right $R$-modules. 
	It is shown that $(\Flat R, \Cot R)$ is a hereditary cotorsion pair in $\Mod R$; see \cite[the proof of Proposition 3.1.2, Lemma 3.4.1]{xu2006flat}. 
	\item Assume that $R$ is \textit{Iwanaga-Gorenstein}, that is, $R$ is a two-sided noetherian ring with $\id R_{R} = \id {}_{R}R < \infty$. 
	A finitely generated right $R$-module $M$ is called \textit{Gorenstein-projective} (or \textit{Cohen-Macaulay}) if $\Ext^{>0}_{R}(M_{R}, R_{R}) = 0$. We denote by $\gproj R$ the subcategory of finitely generated Gorenstein-projective right $R$-modules. 
	It is known that $(\gproj R, \fpd R)$ is a hereditary cotorsion pair in $\md R$, where $\fpd R$ is the subcategory of finite projective-dimensional right $R$-modules of $\md R$; 
	see \cite[Theorem 6.2.4]{krause2021homological}.
\end{enumerate}
\end{exam}
 
A resolving subcategory is a generalization of the subcategory of projective modules and we can take a resolution of a module with respect to the subcategory.
The notion of resolving subcategories is closely related to that of hereditary cotorsion pairs in the sense of Proposition \ref{prop chrac of hereditary}.

\begin{defi} [{\cite[Chapter 3]{auslander1991applications}}] \label{def resolving}
Let $\ca$ be an abelian category with enough projectives and injectives. 
A subcategory $\cx$ of $\ca$ is called \textit{resolving} ($\resp$ \textit{coresolving}) if it satisfies the following conditions:
\begin{enumerate}[(1)]
\setlength{\parskip}{0pt}
	\item $\cx$ contains all projective ($\resp$ injective) objects of $\ca$.
	\item $\cx$ is closed under extensions.
	\item $\cx$ is closed under direct summands.
	\item $\cx$ is closed under epi-kernels ($\resp$ mono-cokernels).
\end{enumerate}
\end{defi}

Let $\ca$ be an abelian category with enough projectives and injectives, $\cx$ its subcategory and $M$ an object of $\ca$. 
If $\cx$ is resolving, then an exact sequence $\cdots \to X_{n} \to \cdots \to X_{1} \to X_{0} \to M \to 0$ in $\ca$ is called 
\textit{$\cx$-resolution} of $M$ if all $X_{i}$ belong to $\cx$.
Since $\cx$ contains all projectives of $\ca$, there exists some $\cx$-resolution of $M$.
We inductively define the \textit{$\cx$-resolution dimension of $M$}, denoted by $\rdim_{\cx} M$, as follows. 
If $M \in \cx$, then  $\rdim_{\cx} M = 0$. If $d \geq 1$, $\rdim_{\cx} M \leq d$ 
if there exists an exact sequence $0 \to X_{d} \to X_{d-1} \to \cdots \to X_{0} \to M \to 0$ with each $X_{j} \in \cx$ for $j \in \{0, 1, \cdots, d\}$. 
The \textit{$\cx$-resolution dimension of $\ca$}, denoted by $\rdim_{\cx} \ca$, is defined by $\rdim_{\cx} \ca := \sup \{ \rdim_{\cx}M \mid M \in \ca \}$.
Dually, if $\cx$ is coresolving, we can define the \textit{$\cx$-coresolution dimension of an object $M \in \ca$ or the category $\ca$}, denoted by $\crdim_{\cx} M$ or $\crdim_{\cx} \ca$ respectively.

The following proposition would be well-known.

\begin{prop}[{\cite[Proposition 6.17]{stovicek2014}}] \label{prop chrac of hereditary}
Let $\ca$ be an abelian category with enough projectives and injectives and $(\cx, \cy)$ a cotorsion pair in $\ca$. The cotorsion pair $(\cx, \cy)$ is hereditary if and only if $\cx$ is resolving if and only if $\cy$ is coresolving.
\end{prop}
\begin{proof}
Assume that a pair $(\cx, \cy)$ is hereditary cotorsion pair. We prove the first equivalent condition. We show that $\cx$ satisfy the conditions (1), (2), (3) and (4) of Definition \ref{def resolving}.

$\underline{(1).}$ This follows from the $\mathrm{Ext}^{1}$-vanishing property of projective objects.

$\underline{(2).}$ Take any exact sequence $0 \to X_{1} \to E \to X_{2} \to 0$ with $X_{i} \in \cx$ for $i \in \{1, \  2\}$. It is enough to show that $\Ext_{\ca}^{1}(E, Y) =0$ for any $Y \in \cy$. 

Applying the above exact sequence to the functor $\Ext_{\ca}^{1} (-, Y)$, we have the following exact sequence:
\[
\Ext_{\ca}^{1}(X_{2}, Y) \to \Ext_{\ca}^{1}(E, Y)  \to \Ext_{\ca}^{1}(X_{1}, Y). 
\]
Now, since $(\cx, \cy)$ is a cotorsion pair, the right term and left term of the exact sequence are zero. Thus it holds that $\Ext_{\ca}^{1}(E, Y) = 0$.

$\underline{(3).}$ Let $X_{1} \oplus X_{2} \in \cx$. We have

\[
\Ext_{\ca}^{1} (X_{1}, Y) \oplus \Ext_{\ca}^{1}(X_{2}, Y) \cong \Ext_{\ca}^{1} (X_{1} \oplus X_{2}, Y) = 0
\]
for any $Y \in \cy$.
Therefore, both $X_{1}$ and $X_{2}$ belong to $\cx$ by the definition of cotorsion pairs.

$\underline{(4).}$ Take any exact sequence $0 \to K \to X_{1} \to X_{0} \to 0$ with $X_{i} \in \cx$ for $i \in \{0, 1\}$. The following are exact:
\[
\Ext_{\ca}^{1}(X_{1}, Y) \to \Ext_{\ca}^{1}(K, Y)  \to \Ext_{\ca}^{2}(X_{0}, Y) 
\]
for all $Y \in \cy$.
Since $(\cx, \cy)$ is a hereditary cotorsion pair, the right term and left term of the exact sequence vanish and so does the middle term.

Conversely, assume that $\cx$ is resolving. Let $X \in \cx$ and $Y \in \cy$. Take a projective resolution 
\[
\cdots \xrightarrow{d_{n+1}} P_{n} \xrightarrow{d_{n}} \cdots \xrightarrow{d_{2}} P_{1} \xrightarrow{d_{1}} P_{0} \xrightarrow{d_{0}} X \to 0
\]
of $X$. Then we have the exact sequences
\[
\Ext_{\ca}^{j}(X, Y) \cong \Ext_{\ca}^{j-1}(\Omega^{1} X, Y) \cong \cdots \cong \Ext_{\ca}^{1}(\Omega^{j-1}X, Y)
\]
for $j > 0$, where $\Omega^{k}X := \Ima d_{k}$ for $k \geq 0$. Since $\cx$ contains all projective objects and it is closed under epi-kernels, 
each $\Omega^{k}X$ belongs to $\cx$ and hence $\Ext_{\ca}^{j}(X, Y) \cong  \Ext_{\ca}^{1}(\Omega^{j-1}X, Y) = 0$ holds. 
This completes the proof.

One can show the other equivalence similarly.
\end{proof}

\begin{exam} \label{ex resolving}
Let $R$ be a ring. By Example \ref{ex cotorsion pair} and Proposition \ref{prop chrac of hereditary}, the following hold:
\begin{enumerate}[(1)]
	\item $\Proj R$ ($\resp$ $\Inj R$) is a resolving ($\resp$ coresolving) subcategory of $\Mod R$. For $M \in \Mod R$, $\rdim_{\Proj R} M = \pd M$. ($\resp$ $\crdim_{\Inj R} M = \id M$.) Therefore, the $(\Proj R)$-resolving dimension of $\Mod R$ ($\resp$ $(\Inj R)$-coresolving dimension) is nothing but the right global dimension of $R$.
	\item $\Flat R$ is a resolving subcategory of $\Mod R$. It is clear that $\rdim_{\Flat R} M = \fd M$ for $M \in \Mod R$ and hence the $(\Flat R)$-resolving dimension of $\Mod R$ coincides with the weak global dimension of $R$.
	\item Assume that $R$ is Iwanaga-Gorenstein. Then $\gproj R$ is a resolving subcategory of $\md R$. Then it holds that $\rdim_{\gproj R} \md R = \id R_{R} = \id {}_{R}R$; see, for example, \cite[Corollary 12.3.2]{enochs2011relative}.
\end{enumerate}
\end{exam}

\begin{lemm} \label{lem hereditary}
Let $\ca$ be an abelian category with enough projectives and injectives and $(\cx, \cy)$ a hereditary cotorsion pair in $\ca$. For all $j \geq 0$, put 
\[
\cf_{j} := \{M \in \ca \mid \Ext_{\ca}^{> j}(M, \cy) = 0 \} \subset \cx.
\]
Then the following hold$\colon$
\begin{enumerate}[$(a)$]
\setlength{\itemsep}{0pt} 
	\item $\cf_{k}$ is a torsion-free class of $\cf_{k+1}$ for $k  \geq 0$.
	\item $\cf_{j} = \{M \in \ca \mid \rdim_{\cx}M \leq j \}$ for $j \geq 0$.
\end{enumerate}
\end{lemm}

\begin{proof}
$(a)$ It is easy to show $\cf_{k}$ is closed under extensions in $\ca$ for each $k \geq 0$. 
Thus it follows from Lemma \ref{lem extension closed} that $\cf_{k}$ is closed under conflations in $\cf_{k+1}$. 
In the rest, we show that $\cf_{k}$ is closed under admissible subobjects in $\cf_{k + 1}$. Let $0 \to S \to F_{k} \to F_{k+1} \to 0$ be a conflation in $\cf_{k+1}$ such that $F_{k} \in \cf_{k}$. 
By the definition of $\cf_{k}$, it is enough to show that $\Ext_{\ca}^{k+1}(S, \cy) = 0$. This follows from the long exact sequence of $\mathrm{Ext}$.

$(b)$ $(\subset)$: Take any $F \in \cf_{j}$ and fix a $\cx$-resolution ($\cdots \stackrel{d_{n+1}}{\to} X_{n} \stackrel{d_{n}}{\to} \cdots \stackrel{d_{2}}{\to} X_{1} \stackrel{d_{1}}{\to} X_{0} \stackrel{d_{0}}{\to} F \to 0 $) of $F$. Then we have the exact sequences
\[
\Ext_{\ca}^{k}(\Omega^{j}F, Y) \cong \Ext_{\ca}^{k+1}(\Omega^{j-1} F, Y) \cong \cdots \cong \Ext_{\ca}^{j+k}(F, Y) = 0
\]
for any $k > 0$ and $Y \in \cy$, where $\Omega^{l}X := \Ima d_{l}$ for $l \geq 0$. This means that $\Omega^{j}F$ belongs to $\cx$ and hence $\rdim_{\cx} F \leq j$. 

$(\supset)$: Assume that $M$ is an object of $\ca$ such that $\rdim_{\cx} M \leq j$. By the assumption, there exists an  $\cx$-resolution $0 \to  X_{j} \to \cdots \to X_{1} \to X_{0} \to M \to 0 $. Similarly, we have
\[
\Ext_{\ca}^{k}(M, Y) \cong \Ext_{\ca}^{k-j}(X_{j}, Y) = 0
\]
for all $k > j$ and $Y \in \cy$. Therefore, we obtain $M \in \cf_{j}$.
\end{proof}

Now, we prove the main theorem of this section.
\begin{theo} \label{thm resolving}
Let $\ca$ be an abelian category with enough projectives and injectives, $(\cx, \cy)$ a hereditary cotorsion pair in $\ca$ and $n$ a positive integer. Then the following are equivalent$\colon$
\begin{enumerate}[$(a)$]
\setlength{\parskip}{0pt}
	\item $\cx$ is an $n$-fold torsion-free class of $\ca$.
	\item $\cx$ is a $\K^{(n-1)}\E$-closed subcategory of $\ca$.
	\item $\rdim_{\cx} \ca \leq n$.
\end{enumerate}
Dually, the following are equivalent$\colon$
\begin{enumerate}[$(a)^{\mathrm{op}}$]
\setlength{\parskip}{0pt}
	\item $\cy$ is an $n$-fold torsion class of $\ca$.
	\item $\cy$ is a $\Cc^{(n-1)}\E$-closed subcategory of $\ca$.
	\item $\crdim_{\cy} \ca \leq n$.
\end{enumerate}
\end{theo}

\begin{proof}

\underline{$(a) \Longrightarrow (b)$.} It follows from Proposition \ref{prop nfold vs Kn-1Eclosed}. 

\underline{$(b) \Longrightarrow (c)$.} Assume that $\cx$ is a $\K^{(n-1)}\E$-closed subcategory of $\ca$. 
Let $M \in \ca$. By Proposition \ref{prop chrac of hereditary}, $\cx$ is a resolving subcategory of $\ca$ and hence we can take an $\cx$-resolution ($\cdots \stackrel{d_{n}}{\to} X_{n-1} \to \cdots \to X_{1} \to X_{0} \to M \to 0 $) of $M$. Now, since $\cx$ is closed under $(n-1)$-kernels, $\Ima d_{n} \in \cx$. Therefore, $\rdim_{\cx} M \leq n$.

\underline{$(c) \Longrightarrow (a)$.} Assume that $\rdim_{\cx} \ca \leq n$. By Lemma \ref{lem hereditary}, we have
\[
\cx = \cf_{0} \torf \cf_{1} \torf \cdots \torf \cf_{n-1} \torf \cf_{n} = \ca,
\]
where $\cf_{j} := \{M \in \ca \mid \Ext_{\ca}^{> j}(M, \cy) = 0 \} = \{M \in \ca \mid \rdim_{\cx}M \leq j \}$ for any nonnegative integer $j$.
Thus $\cx$ is an $n$-fold torsion-free class of $\ca$.

By the dual argument, equivalence between $(a)^{\mathrm{op}}$, $(b)^{\mathrm{op}}$ and $(c)^{\mathrm{op}}$ follows.
\end{proof}

\begin{exam} \label{ex KnE}
Let $R$ be a ring and $n$ a positive integer. By using Theorem \ref{thm resolving}, we have the following:

\begin{enumerate}[(1)]
\setlength{\parskip}{0pt}
	\item \textit{The following are equivalent}$\colon$
		\begin{enumerate}[$(a)$]
			\item \textit{$\Proj R$ is an $n$-fold torsion-free class of $\Mod R$.}
			\item \textit{$\Proj R$ is a $\K^{(n-1)}\E$-closed subcategory of $\Mod R$.}
			\item \textit{$\Inj R$ is an $n$-fold torsion class of $\Mod R$.}
			\item \textit{$\Inj R$ is a $\Cc^{(n-1)}\E$-closed subcategory of $\Mod R$.}
			\item \textit{$\rgldim R \leq n$.}

		\end{enumerate}
	\item \textit{The following are equivalent}$\colon$
		\begin{enumerate}[$(a)$]
			\item \textit{$\Flat R$ is an $n$-fold torsion-free class of $\Mod R$.}
			\item \textit{$\Flat R$ is a $\K^{(n-1)}\E$-closed subcategory of $\Mod R$.}
			\item \textit{$\wgldim R \leq n$.}

		\end{enumerate}
	\item \textit{Assume $R$ is Iwanaga-Gorenstein. Then the following are equivalent}$\colon$
		\begin{enumerate}[$(a)$]
			\item \textit{$\gproj R$ is an $n$-fold torsion-free class of $\md R$.}
			\item \textit{$\gproj R$ is a $\K^{(n-1)}\E$-closed subcategory of $\md R$.}
			\item \textit{$\id R_{R} = \id {}_{R}R \leq n$.}
		\end{enumerate}

\end{enumerate}
\end{exam}


\section{Preliminaries for the main theorem} \label{sec Preliminaries for the main theorem}
In this section, we recall notions and properties about the category of all finitely generated modules over a finite dimensional algebra.

In the sequel, assume that $\Lambda$ is a finite dimensional algebra over an algebraically closed field $\mathbb{K}$. 
We denote by $\tau$ the Auslander-Reiten translation. 
For $M \in \md \Lambda$, we denote by $|M|$ the number of pairwise-nonisomorphic indecomposable direct summands of $M$, by $\add M$ the subcategory 
consisting of all modules which are direct sums of direct summands of $M$ and by $\ind M$ the subcategory consisting of 
all indecomposable direct summands of $M$. 

We recall the notion of torsion pairs of $\md \Lambda$.
Let $\ct$ and $\cf$ be subcategoies of $\md \Lambda$.
We call a pair $(\ct, \cf)$ \textit{a torsion pair of $\md \Lambda$} if it hold that $\ct = {}^{\perp_{0}}\cf$ and $\cf = \ct^{\perp_{0}}$.
It is easy to show that if $(\ct, \cf)$ is a torsion pair of $\md \Lambda$, then $\ct$ is a torsion class of $\md \Lambda$ ($\resp$ $\cf$ is a torsion-free class of $\md \Lambda$).
In this case, the converse holds, that is, if $\ct$ ($\resp$ $\cf$) is a torsion class ($\resp$ a torsion-free class) of $\md \Lambda$, then $(\ct, \ct^{\perp_{0}})$ ($\resp$ $({}^{\perp_{0}}\cf, \cf)$) is a torsion pair of $\md \Lambda$.
It is well-known that $(\ct, \cf)$ is a torsion pair of $\md \Lambda$ if and only if $\ct$ and $\cf$ satisfy the following conditions$\colon$
\begin{enumerate}[$(1)$]
\setlength{\itemsep}{0pt}
	\item $\Hom_{\Lambda}(\ct, \cf) = 0$.
	\item $\ct \ast \cf = \md \Lambda$.
\end{enumerate}

\subsection{Approximations and functorially finite subcategories} \label{subsec approximation}
In this subsection, we recall the definition of functorially finite subcategories and properties of left (or right) minimal morphisms and left (or right) approximations in $\md \Lambda$, 
which will be used later.
First, we review functorially finite subcategories.
\begin{defi}
	Let $\cb$ be an additive category and $\cx$ its subcategory.
	\begin{enumerate}[$(1)$]
		\setlength{\itemsep}{0pt}
		\item A morphism $f \colon M \to N$ in $\cb$  is called \textit{left minimal} if every morphism $g \in \End_{\cb} N$ such that $gf = f$ is an isomorphism.
		\item A morphism $f \colon M \to X$ in $\cb$ is called \textit{a left $\cx$-approximation of $M$} if $X \in \cx$ and any morphism from $M$ to an object in $\cx$ factors through $f$. 
		The latter condition means that for any $X' \in \cx$, the induced map $\Hom_{\cb}(f, X') \colon \Hom_{\cb}(X, X') \to \Hom_{\cb}(M, X')$ is surjective. 
		\item A morphism in $\cb$ is called a \textit{left minimal $\cx$-approximation} if it is left minimal and a left $\cx$-approximation.
		\item $\cx$ is called \textit{covariantly finite in $\cb$} if for all objects $M$ of $\cb$, there exists a left $\cx$-approximation of $M$. 
	\end{enumerate}
	Dually, we can define \textit{right minimal morphisms}, \textit{right $\cx$-approximations} and \textit{contravariantly finite subcategories in $\cb$}.
	\begin{enumerate}[$(4)$]
		\item $\cx$ is called \textit{functorially finite in $\cb$} if it is covariantly finite and contravariantly finite.
	\end{enumerate}
\end{defi}

\begin{rema} \label{rem left approx}
Let $\cx$ be a subcategory of an additive category $\cb$ and $M$ an object of $\cb$.
	\begin{enumerate}[(1)]
		\setlength{\itemsep}{0pt}
		\item A left $\cx$-approximation $f \colon M \to X$ of $M$ is not unique but $X$ is often written as $X^{M} := X$.
		\item A left minimal $\cx$-approximation $f \colon M \to X^{M}$ of $M$ is unique up to isomorphisms in the sense that, 
		if $f' \colon M \to X'$ is also left minimal $\cx$-approximation, then  there is an isomorphism (which is not unique) $g \colon X^{M} \to X'$ such that $gf = f'$.
		Indeed, since both $f$ and $f'$ are left $\cx$-approximations of $M$, there exist morphisms $g \colon X^{M} \to X'$ and $g' \colon X' \to X^{M}$ such that $gf = f'$ and $g'f' = f$.
		It holds that $f = g'f'= (g'g)f$ and $f' = gf = (gg')f'$. It follows from the left minimality of $f$ and $f'$ that both $gg'$ and $g'g$ are isomorphisms and hence so is $g$.
	\end{enumerate}
\end{rema}

\begin{exam} \label{ex covariantly fin}
Let $(\ct, \cf)$ be a torsion pair of $\md \Lambda$ and $M \in \md \Lambda$. We can take the unique short exact sequence $0 \to T_{M} \stackrel{\varphi}{\to} M \stackrel{\psi}{\to} F^{M} \to 0$ in $\md \Lambda$ with $T_{M} \in \ct$ and $F^{M} \in \cf$.
Then $\varphi$ is the right  minimal $\ct$-approximation of $M$ and $\psi$ is the left minimal $\cf$-approximation of $M$.
Therefore, $\ct$ is contravariantly finite and $\cf$ is covariantly finite.
\end{exam}

The next proposition is well-known.

\begin{prop}[{see, for example, \cite[Example 2.9 (1)]{takahashi2021resolving}}] \label{prop add is functorially finite}
	Let $M \in \md \Lambda$. Then $\add M$ is functorially finite in $\md \Lambda$.
\end{prop}

The following necessary and sufficient condition for a monomorphism 
to be left minimal is known. It will be used in Proposition \ref{prop Extprogen of ftors} and \ref{prop existence of Extprogen of 2tors}.

\begin{lemm}[{\cite[the dual of Proposition 1.1]{jasso2019introduction}}] \label{lem charac of left minimal}
  Let $0 \to L \stackrel{f}{\to} M \stackrel{g}{\to} N \to 0$ be an exact sequence in $\md \Lambda$.
  Then  $f$ is left minimal if and only if $g$ belongs to $\rad (M, N)$.
\end{lemm}

The following proposition seems to be folklore. 

\begin{prop} \label{prop right minimal existence}
 Let $\cx$ be an additive subcategory of $\md \Lambda$ closed under direct summands and $M \in \md \Lambda$.
 If there exists a left $\cx$-approximation of $M$, then the left minimal $\cx$-approximation of $M$ exists.
\end{prop}

\begin{proof}
	It is well-known.
\end{proof}

The following lemma is called \textit{Wakamatsu's lemma}.

\begin{lemm}[Wakamatsu's lemma] \label{lem wakamatsu}
	Let $\cx$ be an extension-closed subcategory of $\md \Lambda$.
	If $\varphi: X_{M} \to M$ is the right minimal $\cx$-approximation of $M$, then  $\Ext_{\Lambda}^{1}(X, \Ker \varphi) = 0$ for all $X \in \cx$. 
\end{lemm}

\begin{proof}
See, for example, \cite[Lemma 2.1.1]{xu2006flat}.
\end{proof}

In the end of this subsection, we prove the following proposition, which is used in Propositions \ref{prop covariantly finiteness} and \ref{prop contravariantly finiteness}.
\begin{prop} \label{prop cov trans}
	Let $\ca$ be an abelian category, $\cy$ a covariantly finite additive subcategory of $\ca$ and $\cx$ a subcategory of $\ca$ such that $\cx \subset \cy$.
	Then $\cx$ is covariantly finite in $\ca$ if and only if it is covariantly finite in $\cy$.
\end{prop}

\begin{proof}
	(``if'' part)$\colon$ Assume that $\cx$ is covariantly finite in $\cy$. Let $M \in \ca$. Since $\cy$ is a covariantly finite in $\ca$, there exists a left $\cy$-approximation $f \colon M \to Y^{M}$ of $M$.
	Because $\cx$ is covariantly finite in $\cy$, there exists a left $\cx$-approximation $g \colon Y^{M} \to X^{M}$ of $Y^{M}$.
	We show that $ gf \colon M \to X^{M}$ is a left $\cx$-approximation of $M$.
	Take any morphism $\varphi \colon M \to X$ with $X \in \cx$. Since $X \in \cx \subset \cy$ and $f$ is a left $\cy$-approximation of $M$, there exists some morphism $h \colon Y^{M} \to X$ such that the following diagram commutes$\colon$

	\[
	\begin{tikzpicture}[auto]
	\node (11) at (0,0) {$M$}; \node (12) at (1.6,0) {$Y^{M}$}; \node (13) at (3.2,0) {$X^{M}$};
	\node (22) at (1.6, -1.6) {$X$};
	\node at (1.8, -1.7) {.};

	\draw[->, thick] (11) --node{$f$} (12); \draw[->, thick] (12) --node{$g$} (13);
	\draw[->, thick] (11) --node[swap]{$\varphi$} (22); \draw[->, dotted, thick] (12) --node{$h$} (22);
	\end{tikzpicture}
	\]

	Since $g$ is a left $\cx$-approximation of $Y^{M}$, there exists some morphism $i \colon X^{M} \to X$ such that the following diagram commutes$\colon$

	\[
		\begin{tikzpicture}[auto]
		\node (11) at (0,0) {$M$}; \node (12) at (1.6,0) {$Y^{M}$}; \node (13) at (3.2,0) {$X^{M}$};
		\node (22) at (1.6, -1.6) {$X$};
		\node at (1.8, -1.7) {.};
	
		\draw[->, thick] (11) --node{$f$} (12); \draw[->, thick] (12) --node{$g$} (13);
		\draw[->, thick] (11) --node[swap]{$\varphi$} (22); \draw[->, thick] (12) --node{$h$} (22);
		\draw[->, dotted, thick] (13) --node{$i$} (22);
		\end{tikzpicture}
	\]
	Therefore, we have $igf = hf = \varphi$, which means that $gf$ is a left $\cx$-approximation of $M$. This completes the proof.
	
	(``only if'' part)$\colon$ Since $\ca$ contains $\cy$, this implication is clear.
\end{proof}

\subsection{The bijection between functorially finite torsion classes and basic support $\tau$-tilting modules} \label{subsec the bijection by AIR}
The aim of this section is to recall a part of $\tau$-tilting theory, in particular, the bijection between functorially finite torsion classes and basic support $\tau$-tilting modules which is shown in \cite{Adachi_Iyama_Reiten_2014}.
We define several classes of modules introduced in \cite{Adachi_Iyama_Reiten_2014}, which play an important role in $\tau$-tilting theory.

\begin{defi}
	Let $M \in \md \Lambda$.
	\begin{enumerate}[$(1)$]
		\setlength{\itemsep}{0pt}
		\item $M$ is \textit{basic} if it does not have a direct summand of the form $N \oplus N$ for any indecomposable module $N$ in $\md \Lambda$.
		\item $M$ is \textit{rigid} if $\Ext_{\Lambda}^{1}(M, M) = 0$.
		\item $M$ is \textit{$\tau$-rigid} if $\Hom_{\Lambda}(M, \tau M) = 0$.
	  \item $M$ is \textit{$\tau$-tilting} if $M$ is $\tau$-rigid and $|M| =  |\Lambda|$.
		\item $M$ is \textit{support $\tau$-tilting} if there exists an idempotent $e$ of $\Lambda$ such that $M$ is a $\tau$-tilting right $(\Lambda / \langle e \rangle)$-module, where 
		$\langle e \rangle$ is the (two-sided) ideal generated by $e$. 
	\end{enumerate}
We denote by $\taurigid \Lambda$ ($\resp$ $\stautilt \Lambda$) the set of all isomorphism classes of basic $\tau$-rigid ($\resp$ basic support $\tau$-tilting) $\Lambda$-modules.
\end{defi}

\begin{defi}
 Let $M \in \md \Lambda$.
	We denote by $\Fac M$ the subcategory of $\md \Lambda$ that consists of all modules $F$ such that there exists an epimorphism from a module in $\add M$ to $F$, that is, 
	\[
	\Fac M := \big\{F \in \md \Lambda \mid \text{there exists an epimorphism } M' \twoheadrightarrow F, \text{ for some } M' \in \add M \big\}.
	\]
	Dually, we can define $\Sub M$.
\end{defi}

\begin{prop}[{\cite[Proposition 5.8]{auslander1981almost}}] \label{prop taurigid}
	Let $M, N \in \md \Lambda$. 
  $\Hom(M, \tau N) = 0$ if and only if $\Ext_{\Lambda}^{1}(N, \Fac M) = 0$. 
	In particular, $\tau$-rigid modules are rigid.
\end{prop}

We can construct a $1$-fold torsion class of $\md \Lambda$ from a $\tau$-rigid module as follows:
\begin{prop} [see, for example, {\cite[Chapter VI, 1.9. Lemma]{assem2006elements}}] \label{prop FacU is 1fold tors}
	Let $U$ be a $\tau$-rigid module in $\md \Lambda$. $\Fac U$ is the $1$-fold torsion closure of $\add U$, namely, 
	\[\T{1}{\add U} = \Fac U \] holds.
\end{prop}

\begin{proof}
	A more general result is shown in Proposition \ref{prop coknU is n+1fold tors}.
\end{proof}

We recall $\mathrm{Ext}$-projective modules and $\mathrm{Ext}$-progenerators of an extension-closed subcategory of $\md \Lambda$. 

\begin{defi}
	Let $\cx$ be an extension-closed subcategory of $\md \Lambda$ and $M$ and $P$ its objects.
 \begin{enumerate}[(1)]
 \setlength{\itemsep}{0pt}
	 \item $M$ is \textit{$\mathrm{Ext}$-projective in $\cx$} if for all $X \in \cx$, $\Ext_{\Lambda}^{1}(M, X) = 0$.
	 \item $M$ is \textit{split $\mathrm{Ext}$-projective in $\cx$} if all epimorphisms $X \twoheadrightarrow M$ with $X \in \cx$  are split.
	 \item $\cx$ \textit{has enough $\mathrm{Ext}$-projectives} if for any $X \in \cx$, there exists some conflation \[0 \to K \to P \to X \to 0\] in $\cx$ 
	 such that $P$ is $\mathrm{Ext}$-projective in $\cx$. 
	 \item $P$ is an \textit{$\mathrm{Ext}$-progenerator of $\cx$} if $P$ is $\mathrm{Ext}$-projective in $\cx$ and for every $X \in \cx$, there exists a conflation $0 \to T \to P' \to X \to 0$ in $\cx$ with $P' \in \add P$.
 \end{enumerate}
 We denote by $\cp (\cx)$ ($\resp$ $\cp_{0} (\cx)$) the subcategory of $\mathrm{Ext}$-projective ($\resp$ split $\mathrm{Ext}$-projective) modules of $\cx$.
\end{defi}

\begin{rema} \label{rem Ext proj}
	Let $\cx$ be an extension-closed subcategory of $\md \Lambda$.
	\begin{enumerate}[$(1)$]
	\setlength{\itemsep}{0pt}
		\item It is obvious that all split $\mathrm{Ext}$-projective modules of an extension-closed subcategory $\cx$ are $\mathrm{Ext}$-projective in $\cx$, that is, $\cp_{0}(\cx) \subset \cp(\cx)$ holds.
		However, in general, it does $\mathbf{NOT}$ hold that $\cp_{0}(\cx) = \cp(\cx)$. Indeed, assume that $\cx$ has enough $\mathrm{Ext}$-projectives and is closed under direct summands. Then, by \cite[Proposition 4.10]{enomoto2022rigid}, $\cp_{0}(\cx) = \cp(\cx)$ holds if and only if $\cx$ is closed under epi-kernels.
		\item If an extension-closed subcategory $\cx$ of $\md \Lambda$ has an $\mathrm{Ext}$-progenerator $P$ of $\cx$, any $\mathrm{Ext}$-projective module $M$ of $\cx$ belongs to $\add P$. 
		Indeed,  by the definition of $P$, there exists some conflation $0 \to X \to P' \to M \to 0$ in $\cx$ such that $P' \in \add P$. 
		Since $\Ext_{\Lambda}^{1}(M, X) = 0$, we have $M \oplus X = P'$. Hence we have $M \in \add P$ and $\add P = \cp(\cx)$.
		\item If $\cx$ has an $\mathrm{Ext}$-progenerator, by definition, it has enough $\mathrm{Ext}$-projectives.
	\end{enumerate}
\end{rema}

In fact, every functorially finite torsion class has an $\mathrm{Ext}$-progenerator.
In order to prove it, we prepare the following two lemmas.

\begin{lemm}[{\cite[Proposition 3.7]{auslander1980preprojective}}] \label{lem cover}
	Let $\cc$ be an additive subcategory of $\md \Lambda$ which is closed under direct summands in $\md \Lambda$. The following are equivalent$\colon$
	\begin{enumerate}[$(a)$]
	 \setlength{\itemsep}{0pt}
	 \item There exists some module $M \in \cc$ such that $\cc \subset \Fac M$.
	 \item There exists a left $\cc$-approximation of $\Lambda$. 
	\end{enumerate}
\end{lemm}

\begin{lemm}[{\cite[Proposition 3.7]{auslander1980preprojective}}] \label{lem split projective}
	Let $\cx$ be an extension-closed and direct-summand-closed subcategory of $\md \Lambda$. If there is a left minimal $\cx$-approximation $f \colon \Lambda \to X^{\Lambda}$ of $\Lambda$, then 
	it holds that $\add X^{\Lambda} = \cp_{0}(\cx)$.
\end{lemm}

\begin{proof}
	Using the terminology of \cite{auslander1980preprojective}, we show that $\ind X^{\Lambda}$ is \textit{a minimal cover} for $\cx$, that is, for any direct summand $Y$ of $X^{\Lambda}$, $\cx \subset \Fac Y$ implies that $\add X^{\Lambda} = \add Y$.
	By the definition of $Y$, it holds that $\add X^{\Lambda} \supset \add Y$. It remains to show that $\add X^{\Lambda} \subset \add Y$. 
	Since $\cx \subset \Fac Y$, it follows from Lemma \ref{lem cover} and Proposition \ref{prop right minimal existence} that there exists a left minimal $\cx$-approximation $\Lambda \to Y^{\Lambda}$ of $\Lambda$ such that $Y^{\Lambda} \in \add Y$.
	By Remark \ref{rem left approx} (2), we have $X^{\Lambda} \cong Y^{\Lambda}$ and hence $X^{\Lambda} \in \add Y$.
	This implies that $\add X^{\Lambda} \subset \add Y$.
	Since $\ind X^{\Lambda}$ is a minimal cover for $\cx$, it follows from \cite[Proposition 2.3 (a)]{auslander1980preprojective} that \[\ind X^{\Lambda} = \{ P \in \cp_{0}(\cx) \mid P \text{ is indecomposable} \}.\]
	Therefore, we obtain $\add X^{\Lambda} = \cp_{0}(\cx)$ because $\cp_{0}(\cx)$ is an additive subcategory of $\md \Lambda$ which is closed under direct summands.
\end{proof}
	 
We can construct an $\mathrm{Ext}$-progenerator of a functorially finite torsion class of $\md \Lambda$ as follows.
The proof of the following proposition is due to Haruhisa Enomoto in his lecture on June 13, 2023. 
\begin{prop} [{\cite[Theorem 4.3]{auslander1981almost}}, {\cite[Lemma 3.7]{marks2017torsion}}] \label{prop Extprogen of ftors}
	 Let $\ct$ be a functorially finite torsion class of $\md \Lambda$ and $f \colon \Lambda \to T^{\Lambda}_{0}$ the left minimal $\ct$-approximation of $\Lambda$. 
	 We put $T^{\Lambda}_{1} := \Cok f$. Then  the following hold$\colon$
	 \begin{enumerate}[$(a)$]
	 \setlength{\itemsep}{0pt}
		 \item $\ct = \Fac T_{0}^{\Lambda}$.
		 \item $T^{\Lambda}_{0}$ is split $\mathrm{Ext}$-projective and $T^{\Lambda}_{1}$ is $\mathrm{Ext}$-projective in $\ct$.
		 \item $\ind T_{0}^{\Lambda} \cap \ind T^{\Lambda}_{1} = \emptyset .$
		 \item $T^{\Lambda}_{0} \oplus T^{\Lambda}_{1}$ is an $\mathrm{Ext}$-progenerator of $\ct$.
	 \end{enumerate}
\end{prop}
	 
\begin{proof}
	 $(a)$. It holds from Lemma \ref{lem cover} that $\ct \subset \Fac T_{0}^{\Lambda}$. Thus we obtain $\ct = \Fac T_{0}^{\Lambda}$ because $\ct$ is a torsion class of $\md \Lambda$.
	 	 
	 $(b)$. It follows from Lemma \ref{lem split projective} that $T_{0}^{\Lambda}$ is split $\mathrm{\Ext}$-projective in $\ct$. 
	 Since $T^{\Lambda}_{0} \in \ct$ and $\ct$ is a torsion class, we have $T_{1}^{\Lambda} \in \ct$. 
	 It follows from the dual of Wakamatsu's lemma (Lemma \ref{lem wakamatsu}) that $T_{1}^{\Lambda}$ is $\mathrm{Ext}$-projective in $\ct$.
	 
	 $(c)$. Suppose that there exists a module $M \in \ind T_{0}^{\Lambda} \cap \ind T^{\Lambda}_{1}$. 
	 Let $\pi \colon T_{1}^{\Lambda} \twoheadrightarrow M$ be the projection. The composition $(T_{0}^{\Lambda} \stackrel{\mathrm{cok} f}{\twoheadrightarrow} T_{1}^{\Lambda} \stackrel{\pi}{\twoheadrightarrow} M)$ is, by $(b)$ and Lemma \ref{lem split projective}, split because $M \in \ind T_{0}^{\Lambda}$. 
	 On the other hand, since the morphism $\Ima f \to T^{\Lambda}_{0}$ is left minimal, it follows from Lemma \ref{lem charac of left minimal} that $\mathrm{cok} f$ belongs to $\rad(T^{\Lambda}_{0}, T^{\Lambda}_{1})$ and hence so does the composition $(T_{0}^{\Lambda} \stackrel{\mathrm{cok} f}{\twoheadrightarrow} T_{1}^{\Lambda} \stackrel{\pi}{\twoheadrightarrow} M)$, which contradicts that it is split. 
	 
	 $(d)$. Since both $T^{\Lambda}_{0}$ and $T^{\Lambda}_{1}$ are $\Ext$-projective in $\ct$ by $(b)$, $T^{\Lambda}_{0} \oplus T^{\Lambda}_{1}$ is $\Ext$-projective in $\ct$.
	 Let $T \in \ct$. By $(a)$, there exists an exact sequence $0 \to K \to (T^{\Lambda}_{0})^{n} \to T \to 0$, where $n$ is a positive integer.
	 Take an epimorphism $\pi \colon \Lambda^{m} \twoheadrightarrow K$. Since $f^{m}$ is a left $\ct$-approximation of $\Lambda$, we obtain the following commutative diagram with exact rows:

	 \[
	 \begin{tikzpicture}[auto]
	 \node (01) at (-4.0, 0) {$0$}; \node (11) at (-2.0, 0) {$K$}; 
	 \node (a1) at (0, 0) {$(T^{\Lambda}_{0})^{n}$}; \node (b1) at (2.0, 0) {$T$}; \node(c1) at (4.0, 0) {$0$};
	 \node (12) at (-2.0, 2.0) {$\Lambda^{m}$}; \node (a2) at (0.0, 2.0) {$(T_{0}^{\Lambda})^{m}$}; 
	 \node (b2) at (2.0, 2.0) {$(T_{1}^{\Lambda})^{m}$}; \node (c2) at (4.0, 2.0) {$0$};
	 \node at (1.0, 1.0) {\large $(\star)$}; \node at (4.15, -0.1) {.};
	 
	 \draw[->, thick] (01) -- (11); \draw[->, thick] (11) -- (a1); 
	 \draw[->, thick] (a1) -- (b1); \draw[->, thick] (b1) -- (c1); 
	 \draw[->, thick] (12) --node{$f^{m}$} (a2); \draw[->, thick] (a2) -- (b2);  \draw[->, thick] (b2) -- (c2);

	 \draw[->>, thick] (12) --node[swap]{$\pi$} (11);
	 \draw[->, thick] (a2) -- (a1);
	 \draw[->, thick] (b2) --  (b1);
	 
	 \end{tikzpicture}
	 \]
	 Since $\pi$ is an epimorphism, it follows that the square $(\star)$ is pushout. 
	 Thus we have the following exact sequence:
	 \[
	 (T^{\Lambda}_{0})^{m} \stackrel{h}{\to} (T^{\Lambda}_{1})^{m} \oplus (T^{\Lambda}_{0})^{n} \to T \to 0.
	 \]
	 Since $(T^{\Lambda}_{0})^{m} \in \ct$, its quotient $\Ima h$ is also in $\ct$. Hence $T^{\Lambda}_{0} \oplus T^{\Lambda}_{1}$ is an $\mathrm{Ext}$-progenerator of $\ct$.
\end{proof}
	 
For every functorially finite torsion class $\ct$, we denote by $P(\ct)$ the direct sum of pairwise nonisomorphic indecomposable direct summands of 
the $\mathrm{Ext}$-progenerator constructed by Proposition \ref{prop Extprogen of ftors}.

We have already seen that torsion classes of an abelian category are contravariantly finite (see Example \ref{ex covariantly fin}). It is natural to ask when they are functorially finite.
The following theorem is well-known$\colon$
	 
\begin{theo} [see, for example, {\cite[Proposition 1.1]{Adachi_Iyama_Reiten_2014}, \cite[Theorem 4.8]{treffinger2022tautiltingtheoryintroduction}}] \label{thm charac ftors}
	Let $\ct$ be a torsion class of $\md \Lambda$. Then  the following are equivalent$\colon$
 \begin{enumerate}[$(a)$]
 \setlength{\itemsep}{0pt}
	 \item $\ct$ is functorially finite in $\md \Lambda$.
	 \item $\ct$ has an $\mathrm{Ext}$-progenerator.
	 \item There exists some $M \in \md \Lambda$ such that $\ct = \Fac M$.
 \end{enumerate}
\end{theo}

We denote by $\ftors \Lambda$ the set of functorially finite torsion classes of $\md \Lambda$. 
By Proposition \ref{prop Extprogen of ftors}, every functorially finite torsion class $\ct$ of $\md \Lambda$ has the basic $\mathrm{Ext}$-progenerator $P(\ct)$. 
In \cite{Adachi_Iyama_Reiten_2014}, the authors constructed a bijection between $\stautilt \Lambda$ and $\ftors \Lambda$. 
That is precisely the next theorem:

\begin{theo}[{\cite[Theorem 2.7]{Adachi_Iyama_Reiten_2014}}] \label{thm stautilt ftors}
	The following maps are mutually inverse$\colon$
	\[
	\begin{tikzpicture}[auto]
	\node at (-2.4, 1.8) {$\stautilt \Lambda$}; \node  at (2.4, 1.8) {$\ftors \Lambda$}; 
	\node[rotate=90]  at (-2.4, 1.3) {$\in$}; \node[rotate=90]  at (2.4, 1.3) {$\in$}; 
	\node  at (-2.4, 0.8) {$T$}; \node  at (2.4, 0.8) {$\Fac T$}; 
	\node  at (-2.4, 0.2) {$P(\ct)$}; \node  at (2.4, 0.2) {$\ct$}; 
	\node(a1) at (-1.8, 1.9) {}; \node(a2) at (1.8, 1.9) {};
	\node(01) at (-1.8, 1.7) {}; \node(02) at (1.8, 1.7) {};

	\node(11) at (-1.8, 0.8) {}; \node(12) at (1.8, 0.8) {};
	\node(21) at (-1.8, 0.2) {}; \node(22) at (1.8, 0.2) {};
	\node at (2.6,0.1) {.};

	\draw[->, thick] (a1) --node {$\Fac$} (a2); \draw[->, thick] (02) --node {$P(-)$} (01); \draw[|->, thick] (11) to (12); \draw[|->, thick] (22) -- (21); 
		
	\end{tikzpicture}
	\]
\end{theo}
$\stautilt \Lambda$ is contained in $\taurigid \Lambda$ and we can naturally construct $2$-fold torsion classes from $\tau$-rigid modules in the sense of Proposition \ref{prop coknU is n+1fold tors}. 
Theorem \ref{thm stautilt ftors} can be extended and that is exactly the main theme of the next section.

\section{A characterization of $2$-fold torsion classes induced by $\tau$-rigid modules} \label{sec charac of 2tors}
As in the previous section, assume that $\Lambda$ is a finite dimensional algebra over an algebraically closed field $\mathbb{K}$. 
The goal in this section is to characterize $2$-fold torsion classes induced by $\tau$-rigid modules.
First, we naturally construct an $n$-fold torsion class of $\md \Lambda$ from a $\tau$-rigid module.

\begin{defi}
	Let $M \in \md \Lambda$ and $n$ a nonnegative integer.
	 We denote by $\cok_{n} M$ the subcategory of $\md \Lambda$ that consists of all modules $C$ such that $C$ is the $n$-cokernel of some $n$-term exact sequence in $\add M$, that is,
	 
	 $\cok_{n} M := \big\{C \in \md \Lambda \mid \text{there is an exact sequence } M_{n} \to \cdots \to M_{0} \to C \to 0 \\ 
	 \hfill \text{ with } M_{i} \in \add M \big\}$.
 
	 We define $\cok_{-1} M$ as $\cok_{-1} M := \md \Lambda$.
	 Note that  we have the following inclusions:
	 \[
		 \cdots \subset \cok_{n} M \subset \cdots \subset \cok_{1} M \subset \cok_{0} M \subset \cok_{-1} M = \md \Lambda.
	 \]
	 Dually, we can define $\sker_{n} M$ and $\sker_{-1} M$.
 \end{defi}
 
 \begin{rema} \label{rem cor and ker}
	 Let $M \in \md \Lambda$.
	 \begin{enumerate}[(1)]
		 \setlength{\itemsep}{0pt}
		 \item By the above definition, $\Fac M = \cok_{0} M$ and $\Sub M = \sker_{0} M$ obviously hold.
		 \item By definition, it is obvious that $\clos{\add M}{\adq}{\cok_{i-1} M} = \cok_{i} M$ and $\clos{\add M}{\ads}{\sker_{i-1} M} = \sker_{i} M$ hold for $i \geq 0$.
	 \end{enumerate}
 \end{rema}

\begin{prop}  \label{prop coknU is n+1fold tors}
	Let $U$ be a $\tau$-rigid module in $\md \Lambda$. For any nonnegative integer $n$,  $\cok_{n-1} U$ is the $n$-fold torsion closure of $\add U$, namely, 
	\[\T{n}{\add U} = \cok_{n-1} U \] holds.
\end{prop}

\begin{proof}
	We prove it by the induction on $n$. The case $n = 0$ follows from $\T{0}{\add U} = \md \Lambda = \cok_{-1} U$. 
	Next, suppose that $n > 1$. By induction hypothesis and Remark \ref{rem cor and ker} (2), the following holds:
	\[\clos{\add U}{\adq}{\T{n-1}{U}} = \cok_{n-1} U, \]
	where $\T{n-1}{U} := \T{n-1}{\add U}$. If $\cok_{n-1} U$ is closed under extensions in $\md \Lambda$, then by Lemma \ref{lem extension closed},
	$\cok_{n-1} U$ is closed under conflations in $\T{n-1}{U}$ and hence we have $\cok_{n-1} U = \clos{\add U}{\adq}{\T{n-1} U} = \T{n}{U}$.
	
	Therefore, it is enough to show that $\cok_{n-1} U$ is closed under extensions in $\md \Lambda$. Take any short exact sequence $0 \to T_{1} \to E \to T_{2} \to 0$ with each $T_{i} \in \cok_{n-1} U$.
	By the definition of $\cok_{n-1} U$, for each $i=1,2$, there exists short exact sequence $0 \to K_{i} \to U_{i} \stackrel{q_{i}}{\to} T_{i} \to 0$ such that $U_{i} \in \add U$ and $K_{i} \in \cok_{n-2} U$.
  By pulling back the exact sequence $0 \to T_{1} \to E \to T_{2} \to 0$ by $q_{2}$, we obtain the following commutative diagram with exact rows and columns:

  \[
    \begin{tikzpicture}[auto]
    \node (001) at (0, -1.2) {$0$}; \node (002) at (1.2, -1.2) {$0$};
    \node (01) at (-2.4, 0) {$0$}; \node (a1) at (-1.2, 0) {$T_{1}$}; \node (x) at (0, 0) {$E$}; \node (y) at (1.2, 0) {$T_{2}$}; \node (02) at (2.4, 0) {$0$};
    \node (021) at (-2.4, 1.2) {$0$}; \node (a2) at (-1.2, 1.2) {$T_{1}$}; \node (x1) at (0, 1.2) {$M$}; \node (y1) at (1.2, 1.2) {$U_{2}$}; \node (022) at (2.4, 1.2) {$0$};
    \node (x2) at (0, 2.4) {$K_{2}$}; \node (y2) at (1.2, 2.4) {$K_{2}$};
    \node (x3) at (0, 3.6) {$0$}; \node (y3) at (1.2, 3.6) {$0$};
    \node at (0.6, 0.6) {PB};
    \node at (1.35, -1.3) {.};

    \draw[->, thick] (01) to (a1); \draw[->, thick] (a1) to (x); \draw[->, thick] (x) to (y); \draw[->, thick] (y) to (02);
    \draw[->, thick] (021) to (a2); \draw[->, thick] (a2) to (x1); \draw[->, thick] (x1) to (y1); \draw[->, thick] (y1) to (022);
    \draw[double distance = 2pt, thick] (x2) -- (y2);

    \draw[->, thick] (x1) -- (x); \draw[->, thick] (x3) -- (x2); \draw[->, thick] (x2) -- (x1); 
    \draw[->, thick] (y3) -- (y2); \draw[->, thick] (y2) -- (y1); \draw[->, thick] (x) -- (001);  \draw[->, thick] (y) -- (002);
    \draw[->, thick] (y1) --node{$q_{2}$} (y); \draw[double distance = 2pt, thick] (a2) -- (a1);

    \end{tikzpicture}
\]
Since $T_{1} \in \cok_{n-1} U \subset \Fac U$ and $U$ is $\tau$-rigid, by Proposition \ref{prop taurigid}, the exact sequence $0 \to T_{1} \to M \to U_{2} \to 0$ is split and hence $M = T_{1} \oplus U_{2}$.
By pulling back the exact sequence $0 \to K_{2} \to T_{1} \oplus U_{2} \to E \to 0$ by $f:= q_{1} \oplus 1_{U_{2}}$, we have the following commutative diagram with exact rows and columns:

\[
  \begin{tikzpicture}[auto]
		\node (001) at (-1.6, -1.6) {$0$}; \node (002) at (0, -1.6) {$0$};
		\node (01) at (-3.2, 0) {$0$}; \node (a1) at (-1.6, 0) {$K_{2}$}; \node (x) at (0, 0) {$T_{1} \oplus U_{2}$}; \node (y) at (1.6, 0) {$E$}; \node (02) at (3.2, 0) {$0$};
		\node (021) at (-3.2, 1.6) {$0$}; \node (a2) at (-1.6, 1.6) {$K$}; \node (x1) at (0, 1.6) {$U_{1} \oplus U_{2}$}; \node (y1) at (1.6, 1.6) {$E$}; \node (022) at (3.2, 1.6) {$0$};
		\node (x2) at (-1.6, 3.2) {$K_{1}$}; \node (y2) at (0, 3.2) {$K_{1}$};
		\node (x3) at (-1.6, 4.8) {$0$}; \node (y3) at (0, 4.8) {$0$};
		\node at (-0.8, 0.8) {\large PB};
		\node at (0.15, -1.7) {.};

  \draw[->, thick] (01) to (a1); \draw[->, thick] (a1) to (x); \draw[->, thick] (x) to (y); \draw[->, thick] (y) to (02);
  \draw[->, thick] (021) to (a2); \draw[->, thick] (a2) to (x1); \draw[->, thick] (x1) to (y1); \draw[->, thick] (y1) to (022);
  \draw[double distance = 2pt, thick] (x2) -- (y2);

  \draw[->, thick] (x1) --node{$f$} (x); \draw[->, thick] (x3) -- (x2); \draw[->, thick] (x2) -- (a2); 
  \draw[->, thick] (y3) -- (y2); \draw[->, thick] (y2) -- (x1); \draw[->, thick] (x) -- (002);  \draw[->, thick] (a1) -- (001);
  \draw[double distance = 2pt, thick] (y1) -- (y); \draw[->, thick] (a2) -- (a1);

  \end{tikzpicture}
\] 
Since $K_{1}, K_{2} \in \cok_{n-2} U$ and $\cok_{n-2} U$ is closed under extensions in $\md \Lambda$ by the induction hypothesis, it holds that $K \in \cok_{n-2} U$ 
and hence by the exact sequence $0 \to K \to U_{1} \oplus U_{2} \to E \to 0$, we have $E \in \cok_{n-1} U$. 
This completes the proof.
\end{proof}

Let $U$ be a $\tau$-rigid module. By Proposition \ref{prop coknU is n+1fold tors}, $\cok_{1} U$ is a $2$-fold torsion class of $\md \Lambda$ and it holds $\T{1}{\cok_{1} U} = \T{1}{\T{2}{\add U}} = \T{1}{\add U} = \Fac U$. In addition, $\Fac U$ is functorially finite in $\md \Lambda$ by Theorem \ref{thm charac ftors}. This allows us to construct a map from $\taurigid \Lambda$ to 
$\fltwotors \Lambda := \{2$-fold torsion classes $\cc$ of $\md \Lambda$ with $\T{1}{\cc} \in \ftors \Lambda \}$. We use the notation $\fltwotors \Lambda$, inspired by the notation in \cite[Definition 1.2]{asai2020semibricks}. 
Next, let us construct a map in the opposite direction. 

\begin{prop} \label{prop prepare charac of 2tors}
  There are maps 
  \[
  \begin{tikzpicture}[auto]
		\node at (-2.4, 2.0) {$\taurigid \Lambda$}; \node  at (2.4, 2.0) {$\fltwotors \Lambda$}; 
		\node[rotate=90]  at (-2.4, 1.5) {$\in$}; \node[rotate=90]  at (2.4, 1.5) {$\in$}; 
		\node  at (-2.4, 1.0) {$U$}; \node  at (2.4, 1.0) {$\cok_{1} U$}; 
		\node  at (-2.4, 0.4) {$\Phi(\cc)$}; \node  at (2.4, 0.4) {$\cc$}; 
    \node at (2.6, 0.3) {,};
		\node(011) at (-1.7, 2.1) {}; \node(021) at (1.6, 2.1) {};
    \node(012) at (-1.7, 1.9) {}; \node(022) at (1.6, 1.9) {};
		\node(11) at (-1.7, 1.0) {}; \node(12) at (1.6, 1.0) {};
		\node(21) at (-1.7, 0.4) {}; \node(22) at (1.6, 0.4) {};

		\draw[->, thick] (011) --node{$\cok_{1}$} (021); \draw[->, thick] (022) --node{$\Phi$} (012); \draw[|->, thick] (11) to (12); \draw[|->, thick] (22) -- (21); 
		
	\end{tikzpicture}
	\]
  where $\Phi(\cc)$ is the basic module such that it holds that $\add \Phi(\cc) = \cc \cap \add P(\T{1}{\cc})$. 
  We also have $\Phi \circ \cok_{1} = \mathrm{id}$. 
\end{prop}

\begin{proof}
  Firstly, we show that $\Phi$ is well-defined, that is, $\Phi(\cc)$ is $\tau$-rigid for any $\cc \in \fltwotors \Lambda$. 
  It follows from Theorem \ref{thm charac ftors} that $P(\T{1}{\cc})$ is a support $\tau$-tilting module in $\md \Lambda$ and it is, in particular, $\tau$-rigid. 
  Since $\Phi(\cc)$ is basic, by the definition of $\Phi(\cc)$, it is a direct summand of $P(\T{1}{\cc})$. Because every direct summands of a $\tau$-rigid modules are $\tau$-rigid modules, $\Phi (\cc)$ belongs to $\taurigid \Lambda$.

  Secondly, we prove that $U = \Phi(\cok_{1} U)$ for every $U \in \taurigid \Lambda$. Let $V := \Phi (\cok_{1} U)$. 
  Since $U \in \cok_{1} U \subset \Fac U$ and $P(\Fac U)$ is an $\mathrm{Ext}$-progenerator of $\Fac U$, there exists a conflation $0 \to F \to P' \to U \to 0$ in $\Fac U$ such that $P' \in \add P(\Fac U)$. 
  It follows from Proposition \ref{prop taurigid} that the conflation is split and hence we obtain $U \in \cok_{1} U \cap \add P(\Fac U)$. 
  Since $\add V = \cok_{1} U \cap \add P(\T{1}{\cok_{1} U})$ by the definition of $V$, we have $\add U \subset \add V$. There exists a conflation $0 \to K \to U' \to V \to 0$ in $\Fac U$ with $U' \in \add U$ because $V \in \cok_{1} U$. 
  Since $V$ is $\tau$-rigid and $\Fac U \subset \Fac V$ holds, by Proposition \ref{prop taurigid}, the conflation $0 \to K \to U' \to V \to 0$ is split and hence we have $U' = K \oplus V$. 
  Thus it holds that $\add U \supset \add V$. This implies $U = V$ because both $U$ and $V$ are basic.
\end{proof}

\begin{exam} \label{exam 2tors vs taurigid}
	In general, the map $\cok_{1}$ is $\mathbf{NOT}$ surjective. Indeed, let $\Lambda := \mathbb{K} Q / \langle ab \rangle$, where $Q$ is the following quiver and $\langle ab \rangle$ is the (two-sided) ideal generated by the path $ab$ of $\mathbb{K} Q$:
  \[
	\begin{tikzpicture}[auto]
		\node (1) at (0, 0) {$1$}; \node (2) at (2.0, 0) {$2$}; \node (3) at (4.0, 0) {$3$};
		\node at (4.15, -0.1) {.};
		\draw[->, thick] (3) --node[swap]{$a$} (2); \draw[->, thick] (2) --node[swap]{$b$} (1); 
	\end{tikzpicture} 
	\]
	The AR quiver of $\md \Lambda$ is as follows:
	\[
	\begin{tikzpicture}[auto]
		\node (1) at (0, 0) {$P_{1}$}; \node (2) at (2.0, 0) {$S_{2}$}; \node (3) at (4.0, 0) {$S_{3}$};
		\node (12) at (1.0, 1.0) {$P_{2}$}; \node (23) at (3.0, 1.0) {$P_{3}$};
		\node at (4.3, -0.1) {.};
		\draw[->, thick, dotted] (3) -- (2); \draw[->, thick, dotted] (2) -- (1);
		\draw[->, thick] (1) -- (12); \draw[->, thick] (12) -- (2); 
		\draw[->, thick] (2) -- (23); \draw[->, thick] (23) -- (3); 
	\end{tikzpicture} 
	\]
	Since $\Lambda$ is representation-finite, for every $2$-fold torsion class of $\md \Lambda$, its $1$-fold torsion closure is functorially finite.  
	The number of basic $\tau$-rigid modules in $\md \Lambda$ is 16 but that of $2$-fold torsion classes belonging to $\fltwotors \Lambda$ is 17 as described in Table \ref{table the number of besic taurigid and 2tors}.
	In fact, we show that $\add(P_{2} \oplus S_{3})$ is not included in the image of $\cok_{1}$ in Example \ref{ex not included in the image of cok1}.

	\begin{table}[hbtp]
		\caption{the list of basic $\tau$-rigid $\Lambda$-modules and $2$-fold torsion classes of $\md \Lambda$ in Example \ref{exam 2tors vs taurigid}}
		\label{table the number of besic taurigid and 2tors}
		\centering
		\begin{tabular}{|c|c|}
			\hline
			basic $\tau$-rigid $\Lambda$-modules & $2$-fold torsion classes of $\md \Lambda$ \\
			\hline
			$0$ & $\{ 0\}$ \\
			$P_{1}$ & $\add P_{1}$ \\
			$S_{2}$ & $\add S_{2}$ \\
			$S_{3}$ & $\add S_{3}$ \\
			$P_{2}$ & $\add P_{2}$ \\
			$P_{3}$ & $\add P_{3}$ \\
	
			$P_{1} \oplus P_{2}$ & $\add (P_{1} \oplus P_{2} \oplus S_{2})$ \\
			$P_{1} \oplus P_{3}$ & $\add (P_{1} \oplus P_{3})$ \\
			$P_{1} \oplus S_{3}$ & $\add (P_{1} \oplus S_{3})$ \\
			$S_{2} \oplus P_{2}$ & $\add (S_{2} \oplus P_{2})$ \\
			$S_{2} \oplus P_{3}$ & $\add (S_{2} \oplus P_{3} \oplus S_{3})$ \\
			$S_{3} \oplus P_{3}$ & $\add (S_{3} \oplus P_{3})$ \\
			$P_{2} \oplus P_{3}$ & $\add (P_{2} \oplus P_{3} \oplus S_{3})$ \\
			$P_{1} \oplus P_{3} \oplus S_{3}$ & $\add (P_{1} \oplus P_{3} \oplus S_{3})$ \\
			$P_{1} \oplus P_{2} \oplus P_{3}$ & $\md \Lambda$ \\
			$P_{2} \oplus S_{2} \oplus P_{3}$ & $\add (P_{2} \oplus S_{2} \oplus P_{3} \oplus S_{3})$ \\
							& $\add (P_{2} \oplus S_{3})$ \\
		\hline
	\end{tabular}
	\end{table}
\end{exam}

Let us determine the image of $\cok_{1}$. We consider the following condition for $\cc \in \fltwotors \Lambda$:

\begin{defi} \label{def condition star}
	Let $\cc \in \fltwotors \Lambda$.
	$\cc$ is said to satisfy condition $\conast$ if for any $C \in \cc$ and the right minimal $(\add P(\T{1}{\cc}))$-approximation $f \colon P_{C} \to C$ of $C$, 
  $P_{C}$ belongs to $\cc$. 
\end{defi}

\begin{rema}
	Let $\cc \in \fltwotors \Lambda$. Note that it holds that $\add P(\T{1}{\cc}) = \cp(\T{1}{\cc})$ by Remark \ref{rem Ext proj} (2). 
\end{rema}

Before we prove that $\cok_{1} U$ for a $\tau$-rigid module $U$ satisfies condition $\conast$, we show the following lemma$\colon$

\begin{lemm} \label{lem often used lemma for a torsion class}
	Let $\ct$ be a torsion class of $\md \Lambda$.
	Assume that the following diagram with exact rows commutes in $\ca \colon$

	\[
	\begin{tikzpicture}[auto]
		\node (11) at (-3.2, 1.6) {$0$}; \node (12) at (-1.6, 1.6) {$T$}; \node (13) at (0, 1.6) {$N$}; \node (14) at (1.6, 1.6) {$C$}; \node (15) at (3.2, 1.6) {$0$};
		\node (21) at (-3.2, 0.0) {$0$}; \node (22) at (-1.6, 0.0) {$M$}; \node (23) at (0, 0.0) {$T'$}; \node (24) at (1.6, 0) {$C$}; \node (25) at (3.2, 0.0) {$0$};
		\node at (3.35, -0.1) {.};

		\draw[->, thick] (11) -- (12); \draw[->, thick] (12) -- (13); \draw[->, thick] (13) -- (14); \draw[->, thick] (14) -- (15);
		\draw[->, thick] (21) -- (22); \draw[->, thick] (22) -- (23); \draw[->, thick] (23) -- (24); \draw[->, thick] (24) -- (25);   
   
		\draw[->, thick] (12) -- (22); \draw[->, thick] (13) -- (23); \draw[double distance = 2pt, thick] (14) --(24); 		
	\end{tikzpicture}
	\]
	If both $T$ and $T'$ belong to $\ct$, then so do $M$ and $N$.
\end{lemm}

\begin{proof}
	Assume that both $T$ and $T'$ belong to $\ct$. The square 
	\[
		\begin{tikzpicture}[auto]
			\node (12) at (-1.6, 1.6) {$T$}; \node (13) at (0, 1.6) {$N$};
			\node (22) at (-1.6, 0.0) {$M$}; \node (23) at (0, 0.0) {$T'$};
			\node at (0.15, -0.1) {.};
	
			\draw[->, thick] (12) -- (13);
			\draw[->, thick] (22) -- (23);
		 
			\draw[->, thick] (12) -- (22); \draw[->, thick] (13) -- (23);
		\end{tikzpicture}
	\]
	is pushout and pullback. Hence, we obtain the following exact sequence$\colon$
	\[
	0 \to T \to N \oplus M \to T' \to 0.
	\]
	Since $T$, $T' \in \ct$ and $\ct$ is a torsion class of $\md \Lambda$, $N \oplus M$ belongs to $\ct$ and hence so do its quotients $M$ and $N$. 
\end{proof}

We verify that $\cok_{1} U$ with $U \in \taurigid \Lambda$ satisfies condition $\conast$.
Note once again that it holds that $\T{1}{\cok_{1} U} = \Fac U$.

\begin{prop} \label{prop star condition}
	Let $U$ be a $\tau$-rigid module in $\md \Lambda$, $C$ a module in $\cok_{1} U$ and $P_{C} \to C$ the right minimal $(\add P(\Fac U))$-approximation of $C$.
	Then $P_{C}$ belongs to $\add U$. In particular, $\cok_{1} U$ satisfy condition $\conast$.
\end{prop}

\begin{proof}
	Take the right minimal $(\add U)$-approximation $f \colon U_{C} \to C$ of $C$. It suffices to show that $f$ is also right minimal $(\add P(\Fac U))$-approximation of $C$ 
	because a right minimal approximation is unique in the sense of the dual of Remark \ref{rem left approx} (2).

	To begin with, we prove that $U_{C} \in \add P(\Fac U)$. Since $U_{C} \in \Fac U$ and $P(\Fac U)$ is an $\mathrm{Ext}$-progenerator of $\Fac U$, there exists some conflation $0 \to F \to P' \to U_{C} \to 0$ in $\Fac U$ with $P' \in \add P(\Fac U)$. 
	It follows from Proposition \ref{prop taurigid} that the conflation is split and hence $U_{C} \in \add P(\Fac U)$.

	It remains to show that $f$ is a right $(\add P(\Fac U))$-approximation of $C$. 
	Let $g \colon P'' \to C$ be a  morphism such that $P'' \in \add P(\Fac U)$. 
	Since $C$ belongs to $\cok_{1} U$, we can take a conflation $0 \to F \to U_{0} \xrightarrow{h} C \to 0$ in $\Fac U$ with $U_{0} \in \add U$. 
	As $f$ is a right $(\add U)$-approximation of $C$, $h$ passes through $f$. Hence $f$ is an epimorphism because $h$ is an epimorphism.  
	Then we have the following commutative diagram with exact rows$\colon$
	\[
		\begin{tikzpicture}[auto]
			\node (11) at (-3.4, 1.7) {$0$}; \node (12) at (-1.7, 1.7) {$F$}; \node (13) at (0.0, 1.7) {$U_{0}$}; \node (14) at (1.7, 1.7) {$C$}; \node (15) at (3.4, 1.7) {$0$};
			\node (21) at (-3.4, 0.0) {$0$}; \node (22) at (-1.7, 0.0) {$K$}; \node (23) at (0.0, 0.0) {$U_{C}$}; \node (24) at (1.7, 0.0) {$C$}; \node (25) at (3.4, 0.0) {$0$};
			\node at (3.55, -0.1) {.};
		
			\draw[->, thick] (11) -- (12); \draw[->, thick] (12) -- (13); \draw[->, thick] (13) --node{$h$} (14); \draw[->, thick] (14) -- (15);
			\draw[->, thick] (21) -- (22); \draw[->, thick] (22) -- (23); \draw[->, thick] (23) --node[swap] {$f$} (24); \draw[->, thick] (24) -- (25);
		
			\draw[->, thick] (12) -- (22);
			\draw[->, thick] (13) -- (23);
			\draw[double distance = 2pt, thick] (14) -- (24);		
		\end{tikzpicture}	
  \]
	Since $F, U_{C} \in \Fac U$ and $\Fac U$ is a torsion class of $\md \Lambda$, by Lemma \ref{lem often used lemma for a torsion class}, we have $K \in \Fac U$.
	Therefore, the exact sequence $0 \to K \to U_{C} \to C \to 0$ is a conflation in $\Fac U$. 
	By the $\mathrm{Ext}$-projectivity of $P''$, $g$ factors through $f$ as follows$\colon$
	
	\[
	\begin{tikzpicture}[auto]
		\node (24) at (1.7, 1.7) {$P''$};
		\node (11) at (-3.4,0) {$0$}; \node (12) at (-1.7,0) {$K$}; \node (13) at (0,0) {$U_{C}$}; \node (14) at (1.7,0) {$C$}; \node (15) at (3.4,0) {$0$};
		\node at (3.55, -0.1) {.};

		\draw[->, thick] (11) -- (12); \draw[->, thick] (12) -- (13); \draw[->, thick] (13) --node[swap]{$f$} (14); \draw[->, thick] (14) -- (15);
		\draw[->, thick] (24) --node{$g$} (14); \draw[->, dotted, thick] (24) -- (13);
	\end{tikzpicture}
	\]

	This implies that $f$ is a right $(\add P(\Fac U))$-approximation of $C$. 
	The proof is complete.
\end{proof}

We denote by $\fltwotors^{\ast} \Lambda$ the set of $2$-fold torsion classes in $\fltwotors \Lambda$ satisfying condition $\conast$. 
Finally, we prove that $2$-fold torsion classes induced by basic $\tau$-rigid modules are exactly the subcategories in $\fltwotors^{\ast} \Lambda$. 
\footnote{In \cite[Corollary 6.12]{hafezi2024tau} submitted to arXiv on October 23, 2024, the authors claim that $2$-fold torsion classes induced by $\tau$-rigid modules are exactly the ``$\CE$-closed subcategories which have enough projectives and admit an $\Ext$-progenerator with no self-right-extensions'' but the author of this paper is unable to verify the statement.}

\begin{theo} \label{thm charac of 2tors}
The following maps are mutually inverse$\colon$
\[
\begin{tikzpicture}[auto]
	\node at (-2.4, 2.0) {$\taurigid \Lambda$}; \node  at (2.4, 2.0) {$\fltwotors^{\ast} \Lambda$}; 
	\node[rotate=90]  at (-2.4, 1.5) {$\in$}; \node[rotate=90]  at (2.4, 1.5) {$\in$}; 
	\node  at (-2.4, 1.0) {$U$}; \node  at (2.4, 1.0) {$\cok_{1} U$}; 
	\node  at (-2.4, 0.4) {$\Phi(\cc)$}; \node  at (2.4, 0.4) {$\cc$}; 
	\node at (2.6, 0.3) {,};
	\node(011) at (-1.7, 2.1) {}; \node(021) at (1.5, 2.1) {};
	\node(012) at (-1.7, 1.9) {}; \node(022) at (1.5, 1.9) {};
	\node(11) at (-1.7, 1.0) {}; \node(12) at (1.5, 1.0) {};
	\node(21) at (-1.7, 0.4) {}; \node(22) at (1.5, 0.4) {};

	\draw[->, thick] (011) --node{$\cok_{1}$} (021); \draw[->, thick] (022) --node{$\Phi$} (012); \draw[|->, thick] (11) to (12); \draw[|->, thick] (22) -- (21); 
	
\end{tikzpicture}
\]
where $\Phi(\cc)$ is the basic module such that it holds that $\add \Phi(\cc) = \cc \cap \add P(\T{1}{\cc})$. 

These maps restrict to the maps in Theorem \ref{thm stautilt ftors}.
\end{theo}

\begin{proof}
	By Propositions \ref{prop prepare charac of 2tors} and \ref{prop star condition}, in the former part, it remains to show that $\mathsf{cok}_{1} \circ \Phi = \mathrm{id}$. 
	Let $\cc \in \fltwotors^{\ast} \Lambda$. Since $\Phi(\cc) \in \cc$ and $\cc$ is closed under cokernels by Proposition \ref{prop CE = 2tors}, $\cok_{1} \Phi(\cc) \subset \cc$ holds. We show that $\cok_{1} \Phi(\cc) \supset \cc$. 
	Take any $C \in \cc$ and the right minimal $(\add P(\T{1}{\cc}))$-approximation $f \colon P_{C} \to C$ of $C$. It follows from condition $\conast$ that $P_{C} \in \cc$ and hence $P_{C} \in \cc \cap \add P(\T{1}{\cc}) = \add \Phi(\cc)$. 
	Since $C \in \cc \subset \T{1}{\cc}$ and $P(\T{1}{\cc})$ is an $\mathrm{Ext}$-progenerator of $\T{1}{\cc}$, there exists some conflation $0 \to F \to P' \xrightarrow{g} C \to 0$ in $\T{1}{\cc}$ with $P' \in \add P(\T{1}{\cc})$. 
	Because $f$ is a right $(\add P(\T{1}{\cc}))$-approximation of $C$, the epimorphism $g$ factors through $f$. Hence $f$ is an epimorphism.
	Then we have the following commutative diagram with exact rows$\colon$
	\[
		\begin{tikzpicture}[auto]
			\node (11) at (-3.4, 1.7) {$0$}; \node (12) at (-1.7, 1.7) {$F$}; \node (13) at (0.0, 1.7) {$P'$}; \node (14) at (1.7, 1.7) {$C$}; \node (15) at (3.4, 1.7) {$0$};
			\node (21) at (-3.4, 0.0) {$0$}; \node (22) at (-1.7, 0.0) {$K$}; \node (23) at (0.0, 0.0) {$P_{C}$}; \node (24) at (1.7, 0.0) {$C$}; \node (25) at (3.4, 0.0) {$0$};
			\node at (3.55, -0.1) {.};
		
			\draw[->, thick] (11) -- (12); \draw[->, thick] (12) -- (13); \draw[->, thick] (13) --node{$g$} (14); \draw[->, thick] (14) -- (15);
			\draw[->, thick] (21) -- (22); \draw[->, thick] (22) -- (23); \draw[->, thick] (23) --node[swap]{$f$} (24); \draw[->, thick] (24) -- (25);
		
			\draw[->, thick] (12) -- (22); 
			\draw[->, thick] (13) -- (23); 
			\draw[double distance = 2pt, thick] (14) -- (24);
		\end{tikzpicture}	
		\]	
	Thus we have $C \in \Fac \Phi(\cc)$ because $P_{C} \in \add \Phi(\cc)$. This implies that $\cc \subset \Fac \Phi(\cc)$. It follows from the minimality of $\T{1}{\cc}$ that $\T{1}{\cc} = \Fac \Phi(\cc)$	because $\Fac \Phi(\cc)$ is a torsion class of $\md \Lambda$ by Proposition \ref{prop FacU is 1fold tors} and $\cc \subset \Fac \Phi(\cc) \subset \T{1}{\cc}$. 

	It is enough to show that $K := \Ker f$ belongs to $\T{1}{\cc} = \Fac \Phi(\cc)$. 
	Since both $F$ and $P_{C}$ belong to the torsion class $\Fac \Phi(\cc)$, by Lemma \ref{lem often used lemma for a torsion class}, it follows that $K \in \T{1}{\cc} = \Fac \Phi(\cc)$. 
	The former part of the proof is complete.

	(The latter part): Let $T \in \stautilt \Lambda$. Then $T$ is an $\mathrm{Ext}$-progenerator of $\Fac T$ and hence we obtain $\Fac T = \cok_{1} T$. 
	For the opposite direction, let $\ct \in \ftors \Lambda$. Since we have $\T{1}{\ct} = \ct$, it holds $\Phi(\ct) = P(\ct)$ by the definition of $\Phi(\ct)$. This completes the proof. 
\end{proof}

If $U$ is a module in $\taurigid \Lambda$, then we write $\overline{U} := P(\Fac U)$. We call $\overline{U}$ the \textit{co-Bongartz completion} of $U$. 
Then the following holds$\colon$

\begin{prop} \label{prop relationship of AIR and mine}
	The following diagram commutes$\colon$
\[
\begin{tikzpicture}[auto]
\node (0) at (-2.0, 2.0) {$\stautilt \Lambda$}; \node (00) at (2.0, 2.0) {$\ftors \Lambda$}; 
\node (1) at (-2.0, 0.0) {$\taurigid \Lambda$}; \node (01) at (2.0, 0.0) {$\fltwotors^{\ast} \Lambda$}; 
\node (2) at (-2.0, -2.0) {$\stautilt \Lambda$}; \node (02) at (2.0, -2.0) {$\ftors \Lambda$}; 
\node at (2.7, -2.1) {.};
\node(011) at (-1.2, 2.0) {}; \node(021) at (1.3, 2.0) {};
\node(111) at (-1.2, 0.0) {}; \node(121) at (1.1, 0.0) {};
\node(211) at (-1.2, -2.0) {}; \node(221) at (1.3, -2.0) {};

\draw[->, thick] (111) --node{$\cok_{1}$} (121);
\draw[->, thick] (011) --node{$\Fac$} (021); 
\draw[->, thick] (211) --node{$\Fac$} (221);

\draw[{Hooks[right]}->, thick] (0) -- (1); \draw[->, thick] (1) --node[swap]{$\overline{(-)}$} (2);
\draw[{Hooks[right]}->, thick] (00) -- (01); \draw[->, thick] (01) --node{$\T{1}{-}$} (02);
\draw[->, thick] (0) to[bend right=70, edge label'={$\mathrm{id}$}] (2);
\draw[->, thick] (00) to[bend left=70, edge label={$\mathrm{id}$}] (02);
\draw[->, thick] (1) -- (02);
\draw[fill=white, draw=white] (0, -1.0) circle[x radius=10pt, y radius=10pt] ;
\node at (0, -1.0) {$\Fac$};

\end{tikzpicture}
\]
\end{prop}

\begin{proof}
 For any $T \in \stautilt \Lambda$, $\overline{T} = P(\Fac T) \cong T$ by Theorem \ref{thm stautilt ftors}. 
 Hence we have \[ (\stautilt \Lambda \hookrightarrow \taurigid \Lambda \xrightarrow{\overline{(-)}} \stautilt \Lambda) = \mathrm{id}. \]

 Let $U \in \taurigid \Lambda$. Since $\overline{U}$ is an $\mathrm{Ext}$-progenerator of $\Fac U$ by the definition of $\overline{U}$, we have $\Fac \overline{U} = \Fac U$.
 Thus we obtain \[ (\taurigid \Lambda \xrightarrow{\overline{(-)}} \stautilt \Lambda \xrightarrow{\Fac} \ftors \Lambda) = (\taurigid \Lambda \xrightarrow{\Fac} \ftors \Lambda). \]
 Since it holds that $\T{1}{\cok_{1} U} = \T{1}{\T{2}{\add U}} = \T{1}{\add U} = \Fac U$, we have 
 \[
 (\taurigid \Lambda \xrightarrow{\cok_{1}} \fltwotors^{\ast} \Lambda \xrightarrow{\T{1}{-}} \ftors \Lambda) = (\taurigid \Lambda \xrightarrow{\Fac} \ftors \Lambda).
 \]

 For any $\ct \in \ftors \Lambda$, it follows from $\T{1}{\ct} = \ct$ by the definition of $\T{1}{\ct}$.
 Therefore,
 \[
 (\ftors \Lambda \hookrightarrow \fltwotors^{\ast} \Lambda \xrightarrow{\T{1}{-}} \ftors \Lambda) = \mathrm{id}
 \]
 follows.

 The remaining commutativity has already been proved in Theorem \ref{thm charac of 2tors}.
\end{proof}

\begin{exam} \label{ex not included in the image of cok1}
	With notations in Example \ref{exam 2tors vs taurigid}, $\cc := \add (P_{2} \oplus S_{3})$ is a $2$-fold torsion class of $\md \Lambda$ (the {\color{red}red} part in the figure below). 
	Then $\T{1}{\cc} = \add (P_{2} \oplus S_{2} \oplus P_{3} \oplus S_{3})$ (the part shaded in {\color{blue} blue} in the figure below) and $P(\T{1}{\cc}) = P_{2} \oplus S_{2} \oplus P_{3}$. 
	The morphism $f \colon P_{3} \twoheadrightarrow S_{3}$ is the right minimal $(\add P(\T{1}{\cc}))$-approximation of $S_{3} \in \cc$ but $P_{3}$ does $\mathbf{NOT}$ belong to $\cc$. Therefore, by Theorem \ref{thm charac of 2tors}, there exist no $\tau$-rigid modules $U$ such that $\cc = \cok_{1} U$.

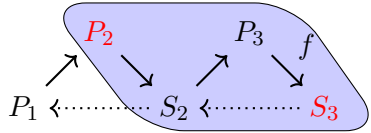
	\[
	\begin{tikzpicture}[auto]
		\draw[rounded corners=15pt, fill=blue!20] (1.6, -0.3) -- (4.8, -0.3) --(3.6, 1.4) -- (0.3, 1.4) -- cycle;
		\node (1) at (0, 0) {$P_{1}$}; \node (2) at (2.0, 0) {$S_{2}$}; \node (3) at (4.0, 0) {\color{red} $S_{3}$};
		\node (12) at (1.0, 1.0) {\color{red} $P_{2}$}; \node (23) at (3.0, 1.0) {$P_{3}$};
		\draw[->, thick, dotted] (3) -- (2); \draw[->, thick, dotted] (2) -- (1);
		\draw[->, thick] (1) -- (12); \draw[->, thick] (12) -- (2); 
		\draw[->, thick] (2) -- (23); \draw[->, thick] (23) --node{$f$} (3); 
	\end{tikzpicture} 
	\]		 
\end{exam}

\section{Several properties of $2$-fold torsion classes induced by $\tau$-rigid modules} \label{sec several properties}
Continuing from the previous section, assume that $\Lambda$ is a finite dimensional algebra over an algebraically closed field $\mathbb{K}$. 
In this section, we verify that $2$-fold torsion classes induced by $\tau$-rigid modules satisfy several nice properties which $1$-fold torsion classes induced by $\tau$-rigid modules satisfy. 
First, let us show that they are covariantly finite in $\md \Lambda$. 
Recall that for any $\tau$-rigid module $U$, $\Fac U$ is a functorially finite torsion class of $\md \Lambda$ which admits an $\mathrm{Ext}$-progenerator by Proposition \ref{prop FacU is 1fold tors} and Theorem \ref{thm charac ftors}. 

\begin{lemm} \label{lem approx}
	Let $U$ be a $\tau$-rigid module in $\md \Lambda$ and $P$  an $\mathrm{Ext}$-progenerator of $\Fac U$.
	There is a left $(\cok_{1} U)$-approximation of $P$.
\end{lemm}

\begin{proof}
Let $\cc := \cok_{1} U$ and $f \colon P \to U^{P}$ a left $(\add U)$-approximation of $P$. 
We show that it is a left $\cc$-approximation of $P$. 
Take any morphism $g \colon P \to C$ with $C \in \cc$. 
Since $C \in \cc$, there is a conflation $0 \to F \to U' \stackrel{c}{\to} C \to 0$ in $\Fac U$ with $F \in \Fac U$ and $U' \in \add U$. Thus, by the $\mathrm{Ext}$-projectivity of $P$, we have the following commutative diagram with an exact row:

\[
\begin{tikzpicture}[auto]
\node (01) at (-2.4, 0) {$0$}; \node (a1) at (-1.2, 0) {$F$}; \node (x) at (0, 0) {$U'$}; \node (y) at (1.2, 0) {$C$}; \node (02) at (2.4, 0) {$0$}; \node at (2.55, -0.1) {.};
\node (y1) at (1.2, 1.2) {$P$};

\draw[->, thick] (01) to (a1); \draw[->, thick] (a1) to (x); \draw[->, thick] (x) --node[swap]{$c$} (y); \draw[->, thick] (y) to (02);

\draw[->, thick] (y1) -- node {$g$} (y);

\draw[->, thick] (y1) -- node[swap] {$h$} (x);

\end{tikzpicture}
\]
By the definition of $f$, $h$ factors through $f$ and hence $g$ passes through $f$. This completes the proof.
\end{proof}

\begin{prop} \label{prop covariantly finiteness}
	Let $U$ be a $\tau$-rigid module in $\md \Lambda$. $\cok_{1} U$ is covariantly finite in $\md \Lambda$.
\end{prop}

\begin{proof}
	Since $\Fac U$ is covariantly finite and $\cok_{1} U \subset \Fac U$, by Proposition \ref{prop cov trans}, it suffices to show that $\cok_{1} U$ is covariantly finite in $\Fac U$. Let $T \in \Fac U$ and $P$ an $\mathrm{Ext}$-progenerator of $\Fac U$.
	As $P$ is an $\mathrm{Ext}$-progenerator of $\Fac U$, there is a conflation $0 \to F \to P^{m} \to T \to 0$ in $\Fac U$, where $m$ is a positive integer.
	By Lemma \ref{lem approx}, there exists a left $(\cok_{1} U)$-approximation $P \to C^{P}$ of $P$. Then $f \colon P^{m} \to {(C^{P})}^{m}$ is also a left $(\cok_{1} U)$-approximation of $(C^{P})^{m}$.
	By pushing out the conflation $0 \to F \to P^{m} \to T \to 0$ in $\Fac U$ by $f$, we obtain the following commutative diagram with exact rows:

\[
\begin{tikzpicture}[auto]
\node (01) at (-2.8, 0) {$0$}; \node (a1) at (-1.4, 0) {$F'$}; \node (x) at (0, 0) {${(C^{P})}^{m}$}; \node (y) at (1.4, 0) {$C$}; \node (02) at (2.8, 0) {$0$};
\node (021) at (-2.8, 1.4) {$0$}; \node (a2) at (-1.4, 1.4) {$F$}; \node (x1) at (0, 1.4) {$P^{m}$}; \node (y1) at (1.4, 1.4) {$T$}; \node (022) at (2.8, 1.4) {$0$};

\node at (0.7, 0.7) {PO};
\node at (2.95, -0.1) {.};

\draw[->, thick] (01) to (a1); \draw[->, thick] (a1) to (x); \draw[->, thick] (x) to (y); \draw[->, thick] (y) to (02);
\draw[->, thick] (021) to (a2); \draw[->, thick] (a2) to (x1); \draw[->, thick] (x1) to (y1); \draw[->, thick] (y1) to (022);

\draw[->, thick] (x1) --node[swap] {$f$} (x);
\draw[->, thick] (y1) --node{$g$} (y); \draw[->>, thick] (a2) --(a1);

\end{tikzpicture}
\]
Since $F \in \Fac U$ and $\Fac U$ is a torsion class of $\md \Lambda$, $F'$ belongs to $\Fac U$. $(C^{P})^{m} \in \cok_{1} U$ implies that $C$ belongs to $\Fac U$. Hence the exact sequence $0 \to F' \to {(C^{P})}^{m} \to C \to 0$ is a conflation in $\Fac U$ with ${(C^{P})}^{m} \in \cok_{1}U$.
This implies $C \in \cok_{1}U$ because $\cok_{1} U$ is a torsion class of $\Fac U$. 
We prove that the $g \colon T \to C$ is a left $(\cok_{1} U)$-approximation of $T$. Take any morphisms $T \to C'$ with $C' \in \cok_{1} U$. Since $f \colon P^{m} \to {(C^{P})}^{m}$ is a left $(\cok_{1} U)$-approximation of $P^{m}$, the composition ($P^{m} \to T \to C'$) factors through $f$. By the universality of pushouts, $T \to C'$ factors through $g$.
We have the desired result.
\end{proof}

Next, we prove that they are also contravariantly finite in $\md \Lambda$ and hence functorially finite. 

\begin{prop} \label{prop contravariantly finiteness}
	Let $U$ be a $\tau$-rigid module in $\md \Lambda$. $\cok_{1} U$ is contravariantly finite in $\md \Lambda$.
\end{prop}

\begin{proof}
	Since $\Fac U$ is contravariantly finite and $\cok_{1} U \subset \Fac U$, by the dual of Proposition \ref{prop cov trans}, it suffices to show that $\cok_{1} U$ is contravariantly finite in $\Fac U$.
  Let $T \in \Fac U$ and $f \colon U_{T} \to T$ a right $(\add U)$-approximation of $T$. 
  Since $T$ is a quotient of a module in $\add U$ and $f$ is a right $(\add U)$-approximation of $T$, $f$ is an epimorphism. Put $K := \Ker f$. Since $\Fac U$ is a torsion class of $\md \Lambda$, there exists a subcategory $\cf$ of $\md \Lambda$ such that $(\Fac U, \cf)$ is a torsion pair of $\md \Lambda$.
  Thus there exists a short exact sequence $0 \to T_{K} \to K \stackrel{g}{\to} F^{K} \to 0$ with $T_{K} \in \Fac U$ and $F^{K} \in \cf$. 
  By pushing out the exact sequence $0 \to K \to U_{T} \to T \to 0$ by $g$, we obtain the following commutative diagram with exact rows and columns:

  \[
  \begin{tikzpicture}[auto]
		\node (001) at (-1.4, -1.4) {$0$}; \node (002) at (0, -1.4) {$0$};
		\node (01) at (-2.8, 0) {$0$}; \node (a1) at (-1.4, 0) {$F^{K}$}; \node (x) at (0, 0) {$C_{T}$}; \node (y) at (1.4, 0) {$T$}; \node (02) at (2.8, 0) {$0$};
		\node (021) at (-2.8, 1.4) {$0$}; \node (a2) at (-1.4, 1.4) {$K$}; \node (x1) at (0, 1.4) {$U_{T}$}; \node (y1) at (1.4, 1.4) {$T$}; \node (022) at (2.8, 1.4) {$0$};
		\node (x2) at (-1.4, 2.8) {$T_{K}$}; \node (y2) at (0, 2.8) {$T_{K}$};
		\node (x3) at (-1.4, 4.2) {$0$}; \node (y3) at (0, 4.2) {$0$};
		\node at (-0.7, 0.7) {PO};
		\node at (0.15, -1.5) {.};
	
  \draw[->, thick] (01) to (a1); \draw[->, thick] (a1) to (x); \draw[->, thick] (x) --node[swap]{$h$} (y); \draw[->, thick] (y) to (02);
  \draw[->, thick] (021) to (a2); \draw[->, thick] (a2) to (x1); \draw[->, thick] (x1) --node{$f$} (y1); \draw[->, thick] (y1) to (022);
  \draw[double distance = 2pt, thick] (x2) -- (y2);

  \draw[->, thick] (x1) --node{$v$} (x); \draw[->, thick] (x3) -- (x2); \draw[->, thick] (x2) -- (a2); 
  \draw[->, thick] (y3) -- (y2); \draw[->, thick] (y2) -- (x1); \draw[->, thick] (x) -- (002);  \draw[->, thick] (a1) -- (001);
  \draw[double distance = 2pt, thick] (y1) -- (y); \draw[->, thick] (a2) -- node[swap] {$g$} (a1);

  \end{tikzpicture}
  \]
  By the exact sequence $0 \to T_{K} \to U_{T} \to C_{T} \to 0$, we have $C_{T} \in \cok_{1} U$. 
	
	We show that the morphism $h \colon C_{T} \to T$ is a right $(\cok_{1} U)$-approximation of $T$.
  Take any morphism $\varphi \colon C \to T$ with $C \in \cok_{1} U$. Let $(C \stackrel{\eta}{\twoheadrightarrow} \Ima \varphi \stackrel{\nu}{\rightarrowtail} T)$ be the image factorization of $\varphi$. 
	Since $C \in \cok_{1} U \subset \Fac U$ and $\Fac U$ is a torsion class of $\md \Lambda$, $I:= \Ima \varphi$ belongs to $\Fac U$.
  By the definition of $\cok_{1} U$, there exists a short exact sequence $0 \to \Ker c \stackrel{\kr{c}}{\to} U_{0} \stackrel{c}{\to} C \to 0$ such that $U_{0} \in \add U$ and $\Ker c \in \Fac U$.
  By pulling back the exact sequence $0 \to \Ker \varphi \to C \xrightarrow{\eta} I \to 0$ by $c$, we have the following commutative diagram with exact rows and columns:

  \[
    \begin{tikzpicture}[auto]
			\node (001) at (-1.8, -1.8) {$0$}; \node (002) at (0, -1.8) {$0$};
			\node (01) at (-3.6, 0) {$0$}; \node (a1) at (-1.8, 0) {$\Ker \varphi$}; \node (x) at (0, 0) {$C$}; \node (y) at (1.8, 0) {$I$}; \node (02) at (3.6, 0) {$0$};
			\node (021) at (-3.6, 1.8) {$0$}; \node (a2) at (-1.8, 1.8) {$L$}; \node (x1) at (0, 1.8) {$U_{0}$}; \node (y1) at (1.8, 1.8) {$I$}; \node (022) at (3.6, 1.8) {$0$};
			\node (x2) at (-1.8, 3.6) {$\Ker c$}; \node (y2) at (0, 3.6) {$\Ker c$};
			\node (x3) at (-1.8, 5.4) {$0$}; \node (y3) at (0, 5.4) {$0$};
			\node at (-0.9, 0.9) {\large PB};
			\node at (0.15, -1.9) {.};
		
    \draw[->, thick] (01) to (a1); \draw[->, thick] (a1) -- (x); \draw[->, thick] (x) --node[swap]{$\eta$} (y); \draw[->, thick] (y) to (02);
    \draw[->, thick] (021) to (a2); \draw[->, thick] (a2) --node{$\kr \psi$} (x1); \draw[->, thick] (x1) --node{$\psi$} (y1); \draw[->, thick] (y1) to (022);
    \draw[double distance = 2pt, thick] (x2) -- (y2);

    \draw[->, thick] (x1) --node{$c$} (x); \draw[->, thick] (x3) -- (x2); \draw[->, thick] (x2) --node[swap]{$k$} (a2); 
    \draw[->, thick] (y3) -- (y2); \draw[->, thick] (y2) --node{$\kr c$} (x1); \draw[->, thick] (x) -- (002);  \draw[->, thick] (a1) -- (001);
    \draw[double distance = 2pt, thick] (y1) -- (y); \draw[->, thick] (a2) -- (a1);
  
    \end{tikzpicture}
  \]
  Since $f \colon U_{T} \to T$ is a right $(\add U)$-approximation of $T$, there exists a morphism $u \colon U_{0} \to U_{T}$ such that the following diagram with exact rows commutes:
  \[
    \begin{tikzpicture}[auto]
			\node (01) at (-3.6, 0) {$0$}; \node (a1) at (-1.8, 0) {$K$}; \node (x) at (0, 0) {$U_{T}$}; \node (y) at (1.8, 0) {$T$};
			\node (021) at (-3.6, 1.8) {$0$}; \node (a2) at (-1.8, 1.8) {$L$}; \node (x1) at (0, 1.8) {$U_{0}$}; \node (y1) at (1.8, 1.8) {$T$};
			\node at (2.0, -0.1) {.}; \node (i) at (0.9,0.9) {$I$};
		
    \draw[->, thick] (01) to (a1); \draw[->, thick] (a1) -- (x); \draw[->, thick] (x) --node[swap]{$f$} (y);
    \draw[->, thick] (021) to (a2); \draw[->, thick] (a2) --node{\footnotesize $\kr \psi$} (x1); \draw[->, thick] (x1) --node{$\nu\psi$} (y1);

    \draw[->, thick] (x1) --node[swap]{$u$} (x);
    \draw[double distance = 2pt, thick] (y1) -- (y); \draw[->, thick] (a2) -- (a1);
  
		\draw[->>, thick] (x1) --node[swap]{$\psi$} (i); \draw[>->, thick] (i) --node[swap]{$\nu$} (y1);
    \end{tikzpicture}
  \]

  Thus 
  \begin{align}
  (\Ker c \stackrel{\kr c}{\to} U_{0} \stackrel{u}{\to} U_{T} \stackrel{v}{\to} C_{T}) &= (\Ker c \stackrel{k}{\to} L \stackrel{\kr \psi}{\to} U_{0} \stackrel{u}{\to} U_{T} \stackrel{v}{\to} C_{T}) \notag \\  
                                                                                   &= (\Ker c \stackrel{k}{\to} L \to K \to U_{T} \stackrel{v}{\to} C_{T}) \notag \\
                                                                                   &= (\Ker c \stackrel{k}{\to} L \to K \stackrel{g}{\to} F^{K} \to C_{T}), \notag 
  \end{align}
	which is zero because $\Ker c \in \Fac U$, $F^{K} \in \cf$ and $(\Fac U, \cf)$ is a torsion pair of $\md \Lambda$.
	Since $vu(\kr c) = 0$, by the universality of the cokernel of $c$, there exists a morphism $t \colon C \to C_{T}$ such that $tc = vu$.
  Since $htc = hvu = fu = \nu \psi = \nu \eta c = \varphi c$ and $c$ is an epimorphism, $ht = \varphi$ holds, which means that $\varphi$ factors through $h$.

  In the end of the proof, we prove that $h$ is right minimal because it is used in Corollary \ref{corr ker of right minimal}. Since it holds that $\Hom_{\Lambda} (C_{T}, F^{K}) = 0$, it follows from the dual of Lemma \ref{lem charac of left minimal} that $h$ is right minimal. This completes the proof.
\end{proof}

\begin{corr} \label{corr ker of right minimal}
	Let $U$ be a $\tau$-rigid module in $\md \Lambda$. For any $M \in \md \Lambda$, the kernel of the right minimal $(\cok_{1} U)$-approximation $f \colon C_{M} \to M$ of $M$ 
	belongs to $(\Fac U)^{\perp_{0}} \cap (\cok_{1} U)^{\perp_{1}}$.
\end{corr}
 
\begin{proof}
	Since $\Fac U$ is a torsion class of $\md \Lambda$, by Example \ref{ex covariantly fin}, we can take the right minimal $(\Fac U)$-approximation $f \colon T_{M} \rightarrowtail M$ of $M$.
	By the proof of Proposition \ref{prop contravariantly finiteness}, there exists a right minimal $(\cok_{1} U)$-approximation $g \colon C_{M} \to T_{M}$ of $T_{M}$ such that $\Ker g \in (\Fac U)^{\perp_{0}}$.
	Since $f$ is a monomorphism, $\Ker(fg) = \Ker g \in (\Fac U)^{\perp_{0}}$.
	If $fg$ is right minimal, then, by Wakamatsu's lemma (Lemma \ref{lem wakamatsu}), it holds $\Ker (fg) \in (\cok_{1} U)^{\perp_{1}}$ and hence we have $\Ker (fg) \in (\Fac U)^{\perp_{0}} \cap (\cok_{1} U)^{\perp_{1}}$.
	Therefore, it is enough to show that $fg$ is right minimal.
	Take any morphism $\varphi \in \End C_{M}$ with $fg \varphi = fg$. Because $f$ is a monomorphism, $g \varphi = g$.
	It follows from the right minimality of $g$ that $\varphi$ is an isomorphism. The proof is complete.
\end{proof}

\begin{prop} \label{prop functorially finiteness}
	Let $U$ be a $\tau$-rigid module in $\md \Lambda$. $\cok_{1} U$ is functorially finite.
\end{prop}

\begin{proof}
	It follows from Proposition \ref{prop covariantly finiteness} and \ref{prop contravariantly finiteness}.
\end{proof}

We introduce $n$-fold torsion pairs of $\md \Lambda$, which is a generalization of torsion pairs of $\md \Lambda$.

\begin{defi} \label{def nfold torspair}
	Let $n$ be a positive integer. We call a $2n$-tuple of subcategories 
	\[
	(\ct_{n}, \ct_{n-1}, \cdots, \ct_{1}; \cf_{n}, \cf_{n-1}, \cdots, \cf_{1})
	\] 
	of $\md \Lambda$ an \textit{$n$-fold torsion pair of $\md \Lambda$} 
	if it satisfies the following conditions$\colon$
	\begin{enumerate}[(1)]
		\setlength{\itemsep}{0pt}
		\item For each $i = 1, \cdots, n$, 
		\[
			\ct_{i} = {}^{\perp_{0}}\cf_{1} \cap {}^{\perp_{1}}\cf_{2} \cap \cdots \cap {}^{\perp_{(i-1)}}\cf_{i}.
		\]
		\item For each $i = 1, \cdots, n$, 
		\[
			\cf_{i} = \ct_{1}^{\perp_{0}} \cap \ct_{2}^{\perp_{1}} \cap \cdots \cap \ct_{i}^{\perp_{(i-1)}}.
		\]
	\end{enumerate}
\end{defi}

\begin{rema}
	\begin{enumerate}[$(1)$]
		\setlength{\itemsep}{0pt}
		\item It follows that the notion of torsion pairs is the same as that of $1$-fold torsion pairs in $\md \Lambda$. 
		\item It follows from Proposition \ref{prop how to make nfold tors} that \[\ct_{n} \tors \ct_{n-1} \tors \cdots \tors \ct_{1} \tors \ca \] and \[ \cf_{n} \torf \cf_{n-1} \torf \cdots \torf \cf_{1} \torf \ca. \]
	\end{enumerate}
\end{rema}

Given a $\tau$-rigid module, we can naturally construct a $2$-fold torsion pair of $\md \Lambda$ as follows$\colon$

\begin{prop} \label{prop 2torspair induced by taurigid}
	Let $U$ be a $\tau$-rigid in $\md \Lambda$. Then there are subcategories $\cf_{1}, \cf_{2}$ of $\md \Lambda$ such that $(\cok_{1} U, \Fac U; \cf_{2}, \cf_{1})$ is a $2$-fold torsion pair of $\md \Lambda$. 
\end{prop}
 
\begin{proof}
	We put \[\cf_{1} := (\Fac U)^{\perp_{0}} \text{ and } \cf_{2} := (\Fac U)^{\perp_{0}} \cap (\cok_{1} U)^{\perp_{1}}. \] 	
	We show that \[\Fac U = {}^{\perp_{0}}\cf_{1} \text{ and }\]
	\[\cok_{1} U = {}^{\perp_{0}}{\cf_{1}} \cap {}^{\perp_{1}}{\cf_{2}}.\] 
	Since $\Fac U$ is a torsion class of $\md \Lambda$, the first equation holds. 

	We show that the second equation holds.

	($\subset$): Let $M \in \cok_{1} U$. It holds $M \in \cok_{1} U \subset \Fac U = {}^{\perp_{0}}{\cf_{1}}$. Thus it remains to show that $M \in {}^{\perp_{1}}\cf_{2}$.
	Take any $F \in \cf_{2}$. By the definition of $\cf_{2}$, $F \in (\cok_{1} U)^{\perp_{1}}$ and hence $\Ext_{\Lambda}^{1}(M, F) = 0$. Therefore, we obtain $M \in {}^{\perp_{1}}\cf_{2}$.

	($\supset$):  Let $M \in {}^{\perp_{0}}{\cf_{1}} \cap {}^{\perp_{1}}{\cf_{2}}$ and take the right minimal $(\cok_{1} U)$-approximation $f \colon C_{M} \to M$ of $M$. 
	Since $M \in {}^{\perp_{0}}\cf_{1} = \Fac U$, by the proof of Proposition \ref{prop contravariantly finiteness}, $f$ is an epimorphism.
	It follows from Corollary \ref{corr ker of right minimal} that $\Ker f \in \cf_{2}$. Because $M \in {}^{\perp_{1}}\cf_{2}$, by the exact sequence $0 \to \Ker f \to C_{M} \xrightarrow{f} M \to 0$, $C_{M} = \Ker f \oplus M$ holds.
	Hence $M$ belongs to $\cok_{1} U$ because $\cok_{1} U$ is closed under direct summands by Remark \ref{rem knE is direct summands closed}. 
	We have the desired result.
\end{proof}

The next proposition is used in Proposition \ref{prop existence of Extprogen of 2tors}.

\begin{prop} \label{prop 4term exact sequence of pushout}
	Let $\ca$ be an abelian category.
	Assume that the following diagram with exact rows and exact columns in $\ca$ is commutative$\colon$

	\[
		\begin{tikzpicture}[auto]
		\node (42) at (-1.8, -1.8) {$0$};
		\node (31) at (-3.6, 0) {$0$}; \node (32) at (-1.8, 0) {$X''$}; 
		\node (33) at (0, 0) {$Y''$}; \node (34) at (1.8, 0) {$Z''$}; \node (35) at (3.6, 0) {$0$};
		\node (22) at (-1.8, 1.8) {$X'$}; \node (23) at (0.0, 1.8) {$Y'$}; 
		\node (24) at (1.8, 1.8) {$Z'$}; \node (25) at (3.6, 1.8) {$0$};
		\node (12) at (-1.8, 3.6) {$X$}; \node (13) at (0.0, 3.6) {$Y$}; 
		\node (14) at (1.8, 3.6) {$Z$}; \node (15) at (3.6, 3.6) {$0$};
		\node (02) at (-1.8, 5.4) {$0$}; \node (03) at (0.0, 5.4) {$0$}; 
		\node (04) at (1.8, 5.4) {$0$}; 
		\node at (3.75, -0.1) {.};	
		\draw[->, thick] (12) --node{$f$} (13); \draw[->, thick] (13) --node{$g$} (14); \draw[->, thick] (14) -- (15);
		\draw[->, thick] (22) --node[swap]{$f'$} (23); \draw[->, thick] (23) --node{$g'$} (24); \draw[->, thick] (24) -- (25);
		\draw[->, thick] (31) -- (32); \draw[->, thick] (32) --node[swap]{$f''$} (33); \draw[->, thick] (33) --node[swap]{$g''$} (34); \draw[->, thick] (34) -- (35);

		\draw[->, thick](02) -- (12); \draw[->, thick](12) --node[swap]{$x$} (22); \draw[->, thick](22) --node[swap]{$x'$} (32); \draw[->, thick](32) -- (42);  
		\draw[->, thick](03) -- (13); \draw[->, thick](13) --node{$y$} (23); \draw[->, thick](23) --node[swap]{$y'$} (33);  
		\draw[->, thick](04) -- (14); \draw[->, thick](14) --node{$z$} (24); \draw[->, thick](24) --node{$z'$} (34);  
		
		\draw[->, thick](12) --(23);
		\draw[white, line width=10pt](22) --(13);
	
		\node at(-0.9, 2.7) {$h$};
		\node at (-1.2, 2.4) {$\circlearrowleft$}; \node at (-0.45, 2.9) {$\circlearrowleft$};
		
		\end{tikzpicture}
		\]
		Then we obtain the exact sequence
		\[
		X \stackrel{h}{\longrightarrow}  Y' \xrightarrow{\begin{pmatrix} g' \\ y' \\ \end{pmatrix}} Z' \oplus Y'' \xrightarrow{\begin{pmatrix} -z' & g'' \\ \end{pmatrix}} Z'' \to 0.
		\]
\end{prop}

\begin{proof}
	Since $x'$ is an epimorphism, the square
	\[
		\begin{tikzpicture}[auto]
			\node (y') at (0.0, 2.0) {$Y'$}; \node (z') at (2.0, 2.0) {$Z'$};
			\node (y'') at (0.0, 0.0) {$Y''$}; \node (z'') at (2.0, 0.0) {$Z''$};
			\node at (2.2, -0.1) {.};

			\draw[->, thick] (y') --node[swap]{$y'$} (y''); \draw[->, thick] (z') --node{$z '$} (z'');

			\draw[->, thick] (y') --node{$g'$} (z'); \draw[->, thick] (y'') --node[swap]{$g''$} (z'');
			
		\end{tikzpicture}
	\]
	is pushout and hence we have the following exact sequence$\colon$
	\[
	Y' \xrightarrow{\begin{pmatrix} g' \\ y' \\ \end{pmatrix}} Z' \oplus Y'' \xrightarrow{\begin{pmatrix} -z' & g'' \\ \end{pmatrix}} Z'' \to 0.
	\]
	We show that $\Ima h = \Ker \begin{pmatrix} g' \\ y' \\ \end{pmatrix}$. Let $K:= \Ker \begin{pmatrix} g' \\ y' \\ \end{pmatrix}$. Since $h = yf$ holds and $y$ is a monomorphism, we have $\Ima f = \Ima h$.
	Let 
	\[
	(X \stackrel{\varphi}{\twoheadrightarrow} \Ima f \stackrel{\psi}{\rightarrowtail} Y), \ 
	(X' \stackrel{\varphi '}{\twoheadrightarrow} \Ima f' \stackrel{\psi '}{\rightarrowtail} Y') \text{ and }
	(X'' \stackrel{\varphi ''}{\twoheadrightarrow} \Ima f'' \stackrel{\psi ''}{\rightarrowtail} Y'')
	\]
	be the image factorization of $f$, $f'$ and $f''$, respectively.
	Then we have the following commutative diagram with exact rows and exact columns:

	\[
	\begin{tikzpicture}[auto]
		\node (31) at (-3.6, 0) {$0$}; \node (32) at (-1.8, 0) {$\Ima f''$}; 
		\node (33) at (0, 0) {$Y''$}; \node (34) at (1.8, 0) {$Z''$}; \node (35) at (3.6, 0) {$0$};
		\node (21) at (-3.6, 1.8) {$0$};  \node (22) at (-1.8, 1.8) {$\Ima f'$}; \node (23) at (0.0, 1.8) {$Y'$}; 
		\node (24) at (1.8, 1.8) {$Z'$}; \node (25) at (3.6, 1.8) {$0$};
		\node (11) at (-3.6, 3.6) {$0$};  \node (12) at (-1.8, 3.6) {$\Ima f$}; \node (13) at (0.0, 3.6) {$Y$}; 
		\node (14) at (1.8, 3.6) {$Z$}; \node (15) at (3.6, 3.6) {$0$};
		\node (02) at (-1.8, 5.4) {$0$}; \node (03) at (0.0, 5.4) {$0$}; 
		\node (04) at (1.8, 5.4) {$0$}; 
		\node at (3.75, -0.1) {.};	\node at (-0.9, 2.7) {\large $(\star)$};

		\draw[->, thick] (12) --node{$\psi$} (13); \draw[->, thick] (13) --node{$g$} (14); \draw[->, thick] (14) -- (15);
		\draw[->, thick] (22) --node[swap]{$\psi '$} (23); \draw[->, thick] (23) --node[swap]{$g'$} (24); \draw[->, thick] (24) -- (25);
		\draw[->, thick] (31) -- (32); \draw[->, thick] (32) --node[swap]{$\psi ''$} (33); \draw[->, thick] (33) --node[swap]{$g''$} (34); \draw[->, thick] (34) -- (35);
		\draw[->, thick] (21) -- (22); \draw[->, thick] (11) -- (12);

		\draw[->, thick](02) -- (12); \draw[->, thick](12) --node[swap]{$i$} (22); \draw[->, thick](22) --node[swap]{$i'$} (32);  
		\draw[->, thick](03) -- (13); \draw[->, thick](13) --node{$y$} (23); \draw[->, thick](23) --node[swap]{$y'$} (33);  
		\draw[->, thick](04) -- (14); \draw[->, thick](14) --node{$z$} (24); \draw[->, thick](24) --node{$z'$} (34);

		\end{tikzpicture}
		\]
		It follows that the square $(\star)$ is pullback because $z$ is a monomorphism. 

		Put $k := \ker \begin{pmatrix} g' \\ y' \\ \end{pmatrix}$. Then there exist morphisms $a \colon K \to Y$ and $b \colon K \to \Ima f'$ in $\ca$ such that 
		$k = ya$ and $k = \psi ' b$ by the universality of the kernels of $y'$ and $g'$. 
		Since $k$ is a monomorphism, so is $a$.
		As the square $(\star)$ is pullback, we obtain the following commutative diagram:

		\[
		\begin{tikzpicture}[auto]
			\node (k) at (-1.4,3.4) {$K$};
			\node (imf) at (0.0, 2.0) {$\Ima f$}; \node (y) at (2.0, 2.0) {$Y$};
			\node (imf') at (0.0, 0.0) {$\Ima f'$}; \node (y') at (2.0, 0.0) {$Y'$};
			\node at (1.0,1.0) {\large $\mathrm{PB}$}; \node at (2.2, -0.1) {.};

			\draw[->, thick] (imf) --node{$\psi$} (y); \draw[->, thick] (imf') --node[swap]{$\psi '$} (y');

			\draw[->, thick] (imf) --node[swap]{$i$} (imf'); \draw[->, thick] (y) --node{$y$} (y');
			\draw[->, thick] (k) to[bend left=50, edge label={$a$}] (y);
			\draw[->, thick] (k) to[bend right=50, edge label'={$b$}] (imf');
			\draw[->, thick] (k) -- (imf);

			\draw[fill=white, draw=white] (-0.7, 2.7) circle[radius = 5pt];
			\node at (-0.7, 2.7) {$c$};	
			
		\end{tikzpicture}
		\]

		$c$ is a monomorphism because $a$ is a monomorphism. In the rest, we prove that $c$ is an epimorphism. 
		Since $(X \xrightarrow{h} Y' \xrightarrow{\begin{pmatrix} g' \\ y' \\ \end{pmatrix}} Z' \oplus Y'') = 0$, there exists some morphism $d \colon X \to K$ such that $h = kd$ holds.
		Hence we obtain 
		
		\begin{align*}
			(X \stackrel{\varphi}{\twoheadrightarrow} \Ima f \stackrel{i}{\rightarrowtail} \Ima f' \stackrel{\psi '}{\rightarrowtail} Y') &= (X \stackrel{\varphi}{\twoheadrightarrow} \Ima f \stackrel{\psi}{\rightarrowtail} Y \stackrel{y}{\rightarrowtail} Y') \\
			&= (X \stackrel{f}{\to} Y \stackrel{y}{\rightarrowtail} Y') \\
			&= (X \stackrel{h}{\to} Y') \\
			&= (X \stackrel{d}{\to} K \stackrel{k}{\rightarrowtail} Y') \\
			&= (X \stackrel{d}{\to} K \stackrel{a}{\rightarrowtail} Y \stackrel{y}{\rightarrowtail} Y') \\
			&= (X \stackrel{d}{\to} K \stackrel{c}{\to} \Ima f \stackrel{\psi}{\rightarrowtail} Y \stackrel{y}{\rightarrowtail} Y') \\
			&= (X \stackrel{d}{\to} K \stackrel{c}{\to} \Ima f \stackrel{i}{\rightarrowtail} \Ima f' \stackrel{\psi '}{\rightarrowtail} Y'). 
		\end{align*}
		Since both $i$ and $\psi '$ are monomorphisms, it holds that $\varphi = cd$.
		This implies that $c$ is an epimorphism because $\varphi$ is an epimorphism.
\end{proof}

The following proposition is an analog of Proposition \ref{prop Extprogen of ftors}.
\begin{prop} \label{prop existence of Extprogen of 2tors}
	Let $U$ be a $\tau$-rigid module in $\md \Lambda$ and $P$ an $\mathrm{Ext}$-progenerator of $\Fac U$. 
	By Lemma \ref{lem approx} and Proposition \ref{prop right minimal existence}, we can take the left minimal $(\cok_{1} U)$-approximation $f \colon P \to U^{P}$ of $P$ such that $U^{P} \in \add U$. Put $C_{1}^{P} := \Cok f$. 
	Then  the following hold$\colon$
	\begin{enumerate}[$(a)$]
	\setlength{\itemsep}{0pt}
		\item $\Fac U = \Fac U^{P}$. 
		\item Both $U^{P}$ and $C_{1}^{P}$ are $\mathrm{Ext}$-projective in $\cok_{1} U$.
		\item $\ind U^{P} \cap \ind C_{1}^{P} = \emptyset$.
		\item $U^{P} \oplus C_{1}^{P}$ is an $\mathrm{Ext}$-progenerator of $\cok_{1} U$.
	\end{enumerate}
\end{prop}
	
\begin{proof}
  $(a).$ Since $U^{P} \in \add U \subset \Fac U$ and $\Fac U$ is a torsion class of $\md \Lambda$, we obtain $\Fac U \supset \Fac U^{P}$.
	It remains to show that $\Fac U \subset \Fac U^{P}$. Since $\T{1}{\cok_{1} U} = \Fac U$ and $\Fac U^{P}$ is a torsion class of $\md \Lambda$ by Proposition \ref{prop FacU is 1fold tors}, it suffices to show that $\cok_{1} U \subset \Fac U^{P}$. 
	Let $C \in \cok_{1} U$. Since $C \in \cok_{1} U \subset \Fac U$ and $P$ is an $\mathrm{Ext}$-progenerator of $\Fac U$, there exists a conflation $0 \to F \to P^{n} \xrightarrow{g} C \to 0$ in $\Fac U$ for some nonnegative integer $n$.
	Because $f$ is a left $(\cok_{1} U)$-approximation of $P$, $f^{n}$ is a left $(\cok_{1} U)$-approximation of $P^{n}$.
	Hence there exists some morphism $h \colon (U^P)^{n} \to C$ such that $hf^{n} = g$.
	Since $g$ is an epimorphism, so is $h$, which means that $C \in \Fac U^{P}$. 

	$(b).$ Since it follows from the $\tau$-rigidity of $U^{P}$ that $\Ext_{\Lambda}^{1}(U^{P}, \Fac U^{P}) = 0$ and $\cok_{1} U \subset \Fac U$, by $(a)$, it holds that $\Ext_{\Lambda}^{1}(U^{P}, \cok_{1} U) = 0$.
	This implies that $U^{P}$ is $\mathrm{Ext}$-projective in $\cok_{1} U$.

	We show that $C_{1}^{P}$ is $\mathrm{Ext}$-projective in $\cok_{1} U$.
	Since $P$ belongs to $\Fac U$ and $U^{P}$ belongs to $\add U$, we have $C_{1}^{P} \in \cok_{1} U$. 
	It follows from the dual of Wakamatsu's lemma (Lemma \ref{lem wakamatsu}) that $\Ext_{\Lambda}^{1}(C^{P}_{1}, \cok_{1} U) = 0$ holds, which means $C^{P}_{1}$ is $\mathrm{Ext}$-projective in $\cok_{1} U$.
	
	$(c).$ Suppose that there exists some object $M \in \ind U^{P} \cap \ind C^{P}_{1} $.
	Since $M$ belongs to $\ind C_{1}^{P}$, there exists a split exact sequence $0 \to N \stackrel{\iota}{\to} C^{P}_{1} \to M \to 0$ with $N \in \cok_{1} U$.
	By pulling back the exact sequence $0 \to \Ima f \to U^{P} \to C^{P}_{1} \to 0$ by  $\iota$, we have the following commutative diagram with exact rows and columns:

	\[
	\begin{tikzpicture}[auto]
		\node (03) at (0, 3.6) {$0$}; \node(04) at (1.8, 3.6) {$0$};
		\node (11) at (-3.6, 1.8) {$0$}; \node (12) at (-1.8, 1.8) {$\Ima f$}; \node (13) at (0.0, 1.8) {$L$}; \node (14) at (1.8, 1.8) {$N$}; \node (15) at (3.6, 1.8) {$0$};
		\node (21) at (-3.6, 0.0) {$0$}; \node (22) at (-1.8, 0.0) {$\Ima f$}; \node (23) at (0.0, 0.0) {$U^{P}$}; \node (24) at (1.8, 0.0) {$C^{P}_{1}$}; \node (25) at (3.6, 0.0) {$0$};
		\node (33) at (0, -1.8) {$M$}; \node(34) at (1.8, -1.8) {$M$};
		\node (43) at (0, -3.6) {$0$}; \node(44) at (1.8, -3.6) {$0$};
		\node at (0.9, 0.9) {\large PB}; \node at (1.95, -3.7) {.};

		\draw[->, thick] (11) -- (12); \draw[->, thick] (12) -- (13); \draw[->, thick] (13) -- (14); \draw[->, thick] (14) -- (15);
		\draw[->, thick] (21) -- (22); \draw[->, thick] (22) -- (23); \draw[->, thick] (23) -- (24); \draw[->, thick] (24) -- (25);
		\draw[double distance = 2pt, thick] (33) -- (34);

		\draw[double distance = 2pt, thick] (12) -- (22);
		\draw[->, thick] (03) -- (13); \draw[->, thick] (13) -- (23); \draw[->, thick] (23) --node{$\pi$} (33); \draw[->, thick] (33) -- (43);
		\draw[->, thick] (04) -- (14); \draw[->, thick] (14) --node{$\iota$} (24); \draw[->, thick] (24) -- (34); \draw[->, thick] (34) -- (44);	

	\end{tikzpicture}
	\]
	Since both $\Ima f$ and $N$ belong to $\Fac U$ and $\Fac U$ is closed under extensions in $\md \Lambda$, we have $L \in \Fac U$. 
	Since $M \in \ind U^{P}$ and it follows from the $\tau$-rigidity of $U^{P}$ that $\Ext_{\Lambda}^{1}(U^{P}, \Fac U^{P}) = 0$, it holds that $\Ext_{\Lambda}^{1}(M, \Fac U^{P}) = 0$. Hence, by $(a)$, $\Ext_{\Lambda}^{1}(M, \Fac U) = 0$. Thus the exact sequence $0 \to L \to U^{P} \to M \to 0$ is split. 
	It follows from the left minimality of the morphism ($\Ima f \to U^{P}$) that the morphism ($U^{P} \to C^{P}_{1}$) belongs to $\rad(U^{P}, C_{1}^{P})$ by Lemma \ref{lem charac of left minimal}. 
	Therefore, the composition $(U^{P} \to C_{1}^{P} \to M) = (U^{P} \xrightarrow{\pi} M) \in \rad(U^{P}, M)$ holds, which contradicts that it is split. 

	$(d)$. By $(b)$, both $U^{P}$ and $C_{1}^{P}$ are $\mathrm{Ext}$-projective in $\cok_{1} U$, so is the direct sum $U^{P} \oplus C_{1}^{P}$. 
	Let $C \in \cok_{1} U$. There exists a conflation $0 \to F \to P^{n} \to C \to 0$ in $\Fac U$, where $n$ is a positive integer because $P$ is an $\mathrm{Ext}$-progenerator of $\Fac U$. 
	Since $f^{n} \colon P^{n} \to (U^{P})^{n}$ is a left $(\cok_{1}U)$-approximation of $P^{n}$, we have the following commutative diagram with exact rows:

	\[
	\begin{tikzpicture}[auto]
		\node (11) at (-3.6, 1.8) {$0$}; \node (12) at (-1.8, 1.8) {$F$}; \node (13) at (0.0, 1.8) {$P^{n}$}; \node (14) at (1.8, 1.8) {$C$}; \node (15) at (3.6, 1.8) {$0$};
		\node (21) at (-3.6, 0.0) {$0$}; \node (22) at (-1.8, 0.0) {$K$}; \node (23) at (0.0, 0.0) {$(U^{P})^{n}$}; \node (24) at (1.8, 0.0) {$C$}; \node (25) at (3.6, 0.0) {$0$};
		\node at (3.75, -0.1) {.};
	
		\draw[->, thick] (11) -- (12); \draw[->, thick] (12) -- (13); \draw[->, thick] (13) -- (14); \draw[->, thick] (14) -- (15);
		\draw[->, thick] (21) -- (22); \draw[->, thick] (22) -- (23); \draw[->, thick] (23) -- (24); \draw[->, thick] (24) -- (25);
	
		\draw[->, thick] (12) -- (22);
		\draw[->, thick] (13) -- (23); 
		\draw[double distance = 2pt, thick] (14) -- (24); 
	
		\end{tikzpicture}	
	\]
	Since both $F$ and $(U^{P})^{n}$ belong to the torsion class $\Fac U$ of $\md \Lambda$, by Lemma \ref{lem often used lemma for a torsion class}, we have $K \in \Fac U$.  
	Thus there exists a conflation $0 \to X \to P^{m} \stackrel{g}{\to}  K \to 0$ in $\Fac U$. Since $f^{m}$ is a left $(\cok_{1} U)$-approximation, we obtain, by the snake lemma, the following commutative diagram with exact rows and exact columns:

	\[
	\begin{tikzpicture}[auto]
	\node (42) at (-1.8, -1.8) {$0$};
	\node (31) at (-3.6, 0) {$0$}; \node (32) at (-1.8, 0) {$K$}; 
	\node (33) at (0, 0) {$(U^{P})^{n}$}; \node (34) at (1.8, 0) {$C$}; \node(35) at (3.6, 0) {$0$};
	\node (22) at (-1.8, 1.8) {$P^{m}$}; \node (23) at (0.0, 1.8) {$(U^{P})^{m}$}; 
	\node (24) at (1.8, 1.8) {$(C_{1}^{P})^{m}$}; \node (25) at (3.6, 1.8) {$0$};
	\node (12) at (-1.8, 3.6) {$X$}; \node (13) at (0.0, 3.6) {$Y$}; 
	\node (14) at (1.8, 3.6) {$Z$}; \node (15) at (3.6, 3.6) {$0$};
	\node (02) at (-1.8, 5.4) {$0$}; \node (03) at (0.0, 5.4) {$0$}; 
	\node (04) at (1.8, 5.4) {$0$}; 
	\node at (3.75, -0.1) {.};	
	\draw[->, thick] (12) -- (13); \draw[->, thick] (13) -- (14); \draw[->, thick] (14) -- (15);
	\draw[->, thick] (22) --node[swap]{$f^{m}$} (23); \draw[->, thick] (23) -- (24); \draw[->, thick] (24) -- (25);
	\draw[->, thick] (31) -- (32); \draw[->, thick] (32) -- (33); \draw[->, thick] (33) -- (34); \draw[->, thick] (34) -- (35);

	\draw[->, thick](02) -- (12); \draw[->, thick](12) -- (22); \draw[->, thick](22) -- (32); \draw[->, thick](32) -- (42);  
	\draw[->, thick](03) -- (13); \draw[->, thick](13) -- (23); \draw[->, thick](23) -- (33);  
	\draw[->, thick](04) -- (14); \draw[->, thick](14) -- (24); \draw[->, thick](24) -- (34);  
	
	\draw[->, thick](12) --(23);
	\draw[white, line width=10pt](22) --(13);

	\node at(-0.9, 2.7) {$h$};
	\node at (-1.2, 2.4) {$\circlearrowleft$}; \node at (-0.45, 2.9) {$\circlearrowleft$};
	
	\end{tikzpicture}
	\]
	By Proposition \ref{prop 4term exact sequence of pushout}, we have the following exact sequence:
	\[
	0 \to \Ima h \to (U^{P})^{m} \to (U^{p})^{n} \oplus (C_{1}^{P})^{m} \to C \to 0.
	\]
	It holds that $\Ima h \in \Fac U$ because $X \in \Fac U$ and $\Fac U$ is a torsion class of $\md \Lambda$.
	Thus the image of the morphism $(U^{P})^{m} \to (U^{p})^{n} \oplus (C_{1}^{P})^{m}$ belongs to $\cok_{1} U$ 
	and hence it follows that $U^{P} \oplus C_{1}^{P}$ is an $\mathrm{Ext}$-progenerator of $\cok_{1} U$. 
\end{proof}

	Consequently, $2$-fold torsion classes induced by $\tau$-rigid modules have the following properties:

\begin{theo} \label{thm properties of 2tors induced by taurigid}
	Let $\cc \in \fltwotors^{\ast} \Lambda$. Then $\cc$ has the following properties$\colon$
	\begin{enumerate}[$(a)$]
		\setlength{\itemsep}{0pt}
		\item $\cc$ is functorially finite in $\md \Lambda$. 
		\item There exist some subcategories $\cf_{1}, \cf_{2}$ of $\md \Lambda$ such that $(\cc, \T{1}{\cc} ; \cf_{2}, \cf_{1})$ is a $2$-fold torsion pair of $\md \Lambda$.
		\item $\cc$ admits an $\mathrm{Ext}$-progenerator.
	\end{enumerate} 
\end{theo}

\begin{proof}
	This follows from Theorem \ref{thm charac of 2tors} and Propositions \ref{prop functorially finiteness}, \ref{prop 2torspair induced by taurigid} and \ref{prop existence of Extprogen of 2tors}.
\end{proof}

\section{The hereditary case} \label{sec the hereditary case}
Throughout this section, assume that $\Lambda$ is a hereditary finite dimensional algebra over an algebraically closed field $\mathbb{K}$.
In this case, there is prior research on $\ICE$-closed subcategories of $\md \Lambda$ in \cite{enomoto2022rigid} and \cite{enomoto2022ice}. 
In \cite{enomoto2022rigid}, Enomoto constructed a bijection between the set of isomorphism classes of basic rigid modules and the set of $\ICE$-closed subcategories of $\md \Lambda$ with enough $\mathrm{Ext}$-projectives (Theorem \ref{thm bijection hereditary case}).
The purpose of this section is to give an alternative proof to it using our results in sections \ref{sec charac of 2tors} and \ref{sec several properties}.
	
\begin{prop} \label{prop rigid=taurigid}
	Let $U \in \md \Lambda$. $U$ is rigid if and only if it is $\tau$-rigid.	
\end{prop}
	
\begin{proof}
	(``if''part): This follows from Proposition \ref{prop taurigid}.

	(``only if'' part): See, for example, \cite[Chapter IV, 2.14. Corollary]{assem2006elements}.
\end{proof}
	
The author learned the following proposition from Shunya Saito and Arashi Sakai.
\begin{prop} [{\cite[Proposition 3.2]{Stanley2012tstructuresOH}}] \label{prop Iclosed = direct summands closed}
	Let $\cx$ be an extension-closed subcategory of $\md \Lambda$. $\cx$ is closed under images if and only if it is closed under direct summands.
\end{prop}
	
\begin{proof}
	(``only if'' part): It holds in general. Indeed, let $X \oplus Y \in \cx$. Consider the morphism $\begin{pmatrix} 1 & 0 \\ 0 & 0 \\ \end{pmatrix} \colon X \oplus Y \to X \oplus Y$. Then the image of the morphism is $X$. 
	Hence it holds from the assumption that $X \in \cx$.
	
	(``if'' part): Assume that $\cx$ is closed under direct summands. Take any morphism $f \colon X \to Y$ in $\cx$. 
	Since $\Lambda$ is hereditary, the map $\Ext_{\Lambda}^{1}(\nu ,\Ker f) \colon \Ext_{\Lambda}^{1}(Y, \Ker f) \to \Ext_{\Lambda}^{1}(\Ima f, \Ker f)$ induced by the inclusion $\nu \colon \Ima f \hookrightarrow Y$ is surjective. 
	Hence we have the following commutative diagram with exact rows and exact columns:
	
\[
\begin{tikzpicture}[auto]
	\node (01) at (-3.6, 0) {$0$}; \node (a1) at (-1.8, 0) {$\Ker f$}; \node (x) at (0, 0) {$M$}; \node (y) at (1.8, 0) {$Y$}; \node (02) at (3.6, 0) {$0$};
	\node (021) at (-3.6, 1.8) {$0$}; \node (a2) at (-1.8, 1.8) {$\Ker f$}; \node (x1) at (0, 1.8) {$X$}; \node (y1) at (1.8, 1.8) {$\Ima f$}; \node (022) at (3.6, 1.8) {$0$};
	\node (031) at (0, 3.6) {$0$}; \node (032) at (1.8, 3.6) {$0$};
	\node (y2d) at (0, -1.8) {$\Cok f$}; \node (y2) at (1.8, -1.8) {$\Cok f$};
	\node (001) at (0, -3.6) {$0$}; \node (002) at (1.8, -3.6) {$0$};
	\node at (0.9, 0.9) {\large $(\star)$}; \node at (1.95, -3.7) {.};
	
	\draw[->, thick] (01) to (a1); \draw[->, thick] (a1) to (x); \draw[->, thick] (x) to (y); \draw[->, thick] (y) to (02);
	\draw[->, thick] (021) to (a2); \draw[->, thick] (a2) to (x1); \draw[->, thick] (x1) to (y1); \draw[->, thick] (y1) to (022);
	\draw[double distance = 2pt, thick] (y2d) -- (y2);
	
	\draw[->, thick] (031) to (x1); \draw[->, thick] (x1) -- (x); \draw[->, thick] (x) to (y2d); \draw[->, thick]  (y2d) to (001);
	\draw[->, thick] (032) to (y1); \draw[->, thick] (y1) --node{$\nu$} (y); \draw[->, thick] (y) to (y2); \draw[->, thick]  (y2) to (002);
	\draw[double distance = 2pt, thick] (a1) -- (a2);
	
\end{tikzpicture}
\]
Then $(\star)$ is pushout and pullback and hence we obtain an exact sequence $0 \to X \to M \oplus \Ima f \to Y \to 0$. 
This implies that $\Ima f \in \cx$ because $X, Y \in \cx$ and $\cx$ is closed under extensions and direct summands.
\end{proof}
	
\begin{corr} \label{cor ICE=CE}
	Let $\cx$ be a subcategory of $\md \Lambda$. $\cx$ is a $\CE$-closed subcategory if and only if it is an $\mathrm{ICE}$-closed subcategory.
\end{corr}

\begin{proof}
	(``if'' part)$\colon$ It is clear by definition.

	(``only if'' part)$\colon$ By Remark \ref{rem knE is direct summands closed}, every $\CE$-closed subcategory is closed under direct summands and hence, by Proposition \ref{prop Iclosed = direct summands closed}, it is closed under images. 
\end{proof}

The following proposition plays an important role in the proof of Theorem \ref{thm bijection hereditary case}.

\begin{prop}[{\cite[Proposition 3.1]{enomoto2022rigid}}] \label{prop properties of CEclosed in hereditary}
	Let $\cc$ be a $\CE$-closed subcategory of $\md \Lambda$. Then the following hold$\colon$
	\begin{enumerate}[$(a)$]
		\setlength{\itemsep}{0pt}
		\item Every $\mathrm{\Ext}$-projective module in $\cc$ is rigid.
		\item There are only finitely many indecomposable $\mathrm{Ext}$-projective modules in $\cc$ up to isomorphism.
		\item If $\cc$ has enough $\mathrm{Ext}$-projectives, then the basic $\mathrm{Ext}$-progenerator $P(\cc)$ of $\cc$ exists and $\cc = \cok_{1} P(\cc)$ holds.
	\end{enumerate}
\end{prop}

 We denote by $\rigid \Lambda$ the set of isomorphism classes of basic rigid $\Lambda$-modules and by $\icep \Lambda$ the set of $\ICE$-closed subcategories of $\md \Lambda$ with enough $\mathrm{Ext}$-projectives.
By Proposition \ref{prop properties of CEclosed in hereditary}, every subcategory $\cc$ in $\icep\Lambda$ admits the basic $\mathrm{Ext}$-progenerator $P(\cc)$.
Enomoto proved the following theorem$\colon$
\begin{theo}[{\cite[Theorem 2.3]{enomoto2022rigid}}] \label{thm bijection hereditary case}
	Assume that $\Lambda$ is hereditary.
  The following maps are mutually inverse$\colon$
\[
\begin{tikzpicture}[auto]
	\node at (-2.4, 2.0) {$\rigid \Lambda$}; \node  at (2.4, 2.0) {$\icep \Lambda$}; 
	\node[rotate=90]  at (-2.4, 1.5) {$\in$}; \node[rotate=90]  at (2.4, 1.5) {$\in$}; 
	\node  at (-2.4, 1.0) {$U$}; \node  at (2.4, 1.0) {$\cok_{1} U$}; 
	\node  at (-2.4, 0.4) {$P(\cc)$}; \node  at (2.4, 0.4) {$\cc$}; 
	\node at (2.6, 0.3) {.};
	\node(011) at (-1.7, 2.1) {}; \node(021) at (1.6, 2.1) {};
	\node(012) at (-1.7, 1.9) {}; \node(022) at (1.6, 1.9) {};
	\node(11) at (-1.7, 1.0) {}; \node(12) at (1.6, 1.0) {};
	\node(21) at (-1.7, 0.4) {}; \node(22) at (1.6, 0.4) {};	

	\draw[->, thick] (011) --node{$\cok_{1}$} (021); \draw[->, thick] (022) --node{$P(-)$} (012); \draw[|->, thick] (11) to (12); \draw[|->, thick] (22) -- (21); 
	
\end{tikzpicture}
\]
\end{theo}

In order to prove Theorem \ref{thm bijection hereditary case}, we show two lemmas.

\begin{lemm} \label{lemm fltwotors = icep}
	Let $\cc$ be a subcategory of $\md \Lambda$. $\cc$ belongs to $\icep \Lambda$ if and only if it belongs to $\fltwotors^{\ast} \Lambda$.
\end{lemm}

\begin{proof}
	(``if'' part)$\colon$ Assume that $\cc$ belongs to $\fltwotors^{\ast} \Lambda$. 
	Since $\cc$ is a $2$-fold torsion class of $\md \Lambda$, by Proposition \ref{prop CE = 2tors}, it is a $\CE$-closed subcategory of $\md \Lambda$.
	Hence, by Corollary \ref{cor ICE=CE}, $\cc$ is an $\ICE$-closed subcategory of $\md \Lambda$.
	It follows from Theorem \ref{thm properties of 2tors induced by taurigid} $(c)$ that $\cc$ has an $\mathrm{Ext}$-progenerator 
	and hence it has enough $\mathrm{Ext}$-projectives. Therefore, $\cc$ belongs to $\icep \Lambda$. 

	(``only if'' part)$\colon$ Assume that $\cc$ belongs to $\icep \Lambda$. Since $\cc$ is a $\CE$-closed subcategory of $\md \Lambda$ with enough $\mathrm{Ext}$-projectives, by Proposition \ref{prop properties of CEclosed in hereditary} $(c)$, $\cc$ has the basic $\mathrm{Ext}$-progenerator $P(\cc)$ 
	and $\cc = \cok_{1} P(\cc)$ holds. By Propositions \ref{prop properties of CEclosed in hereditary} $(a)$ and \ref{prop rigid=taurigid}, $P(\cc)$ is $\tau$-rigid.
	Thus, by Theorem \ref{thm charac of 2tors}, $\cc$ belongs to $\fltwotors^{\ast} \Lambda$.
\end{proof}

\begin{lemm} \label{lem Phi = P}
	Let $\cc \in \fltwotors^{\ast} \Lambda$. $\Phi (\cc)$ is the basic $\mathrm{Ext}$-progenerator of $\cc$, 
	where $\Phi$ is the map which appears in Theorem \ref{thm charac of 2tors}.
\end{lemm}

\begin{proof}
	Let $P:= P(\cc)$ and $Q := P(\T{1}{\cc})$.
	To begin with, we prove that $\add P \supset \add \Phi(\cc)$.
	Take any $M \in \add \Phi(\cc)$. By the definition of $\Phi(\cc)$, $M$ is $\mathrm{Ext}$-projective in $\T{1}{\cc}$. 
	Since $M \in \cc$ and $P$ is an $\mathrm{Ext}$-progenerator of $\cc$, there exists some conflation $0 \to C \to P' \to M \to 0$ in $\cc$ with $P' \in \add P$.
	Because $\cc \subset \T{1}{\cc}$ and $M$ is $\mathrm{Ext}$-projective in $\T{1}{\cc}$, the conflation is split and hence it holds $M \oplus C = P'$.
	Thus $M \in \add P$. This implies $\add P \supset \add \Phi(\cc)$.

	Next, we show that $\add P \subset \add \Phi(\cc)$. It is enough to show that $P \in \add Q$ because $P \in \cc$ and $\add \Phi(\cc) = \cc \cap \add Q$.
	By Propositions \ref{prop properties of CEclosed in hereditary} $(a)$ and \ref{prop rigid=taurigid}, $P$ is $\tau$-rigid.
	Hence it follows from Proposition \ref{prop FacU is 1fold tors} that $\Fac P$ is a torsion class of $\md \Lambda$.
	Since $\cc = \cok_{1} P \subset \Fac P$ and $\Fac P$ is a torsion class of $\md \Lambda$, by the minimality of $\T{1}{\cc}$, we obtain $\T{1}{\cc} = \Fac P$.
	There exists some conflation $0 \to F \to Q' \to P \to 0$ in $\T{1}{\cc}$ with $Q' \in \add Q$ because $P \in \cc \subset \T{1}{\cc}$ and $Q$ is an $\mathrm{Ext}$-progenerator of $\T{1}{\cc}$.
	Since $F \in \T{1}{\cc} = \Fac P$ and $P$ is $\tau$-rigid, the conflation is split by Proposition \ref{prop taurigid} and hence $P \oplus F = Q'$.
	Thus we have $P \in \cc \cap \add Q = \add \Phi (\cc)$.
	Consequently, $\add P = \add \Phi(\cc)$ and hence it holds $P = \Phi(\cc)$ because both $P$ and $\Phi(\cc)$ are basic.
\end{proof}

\begin{proof}[Proof of Theorem \ref{thm bijection hereditary case}]
	By Proposition \ref{prop rigid=taurigid}, we have $\taurigid \Lambda = \rigid \Lambda$. It follows from Lemma \ref{lemm fltwotors = icep} that $\fltwotors^{\ast} \Lambda = \icep \Lambda$.
	By Theorem \ref{thm charac of 2tors} and Lemma \ref{lem Phi = P}, the proof is complete.
\end{proof}

By combining \cite[Proposition 5.6]{enomoto2022ice} and results of this paper, we obtain the next proposition.
\begin{prop} 
 Assume that $\Lambda$ is hereditary. Let $\cc$ be a $2$-fold torsion class of $\md \Lambda$. 
 The following are equivalent$\colon$
 \begin{enumerate}[$(a)$]
	\setlength{\itemsep}{0pt}
	\item $\cc$ is functorially finite in $\md \Lambda$.
	\item $\cc$ has an $\mathrm{Ext}$-progenerator.
	\item $\cc$ has enough $\mathrm{Ext}$-projectives.
	\item $\T{1}{\cc}$ is functorially finite in $\md \Lambda$ and $\cc$ satisfies condition $\conast$.
 \end{enumerate}
\end{prop}

\begin{proof}
	By Proposition \ref{prop CE = 2tors}, every $2$-fold torsion class of $\md \Lambda$ is a $\CE$-closed subcategory of $\md \Lambda$.
	The implications $(a) \Longleftrightarrow (b) \Longleftrightarrow (c)$ follows from Corollary \ref{cor ICE=CE} and \cite[Proposition 5.6]{enomoto2022ice}.

	By Corollary \ref{cor ICE=CE} and Lemma \ref{lemm fltwotors = icep}, we have $(c) \Longleftrightarrow (d)$. This completes the proof.
\end{proof}

 In the hereditary case, there are no non-trivial $n$-fold torsion(-free) classes of $\md \Lambda$ for $n > 2$ in the following sense:

\begin{prop} \label{prop 2fold = nfold}
	Assume that $\Lambda$ is hereditary.
	Let $\cx$ be a subcategory of $\md \Lambda$ and $n$ an integer greater than $2$.
	$\cx$ is a $2$-fold torsion(-free) class of $\md \Lambda$ if and only if it is an $n$-fold torsion(-free) class of $\md \Lambda$.
\end{prop}

\begin{proof}
	(``if'' part): It suffices to show that every $3$-fold torsion ($\resp$ torsion-free) class of $\md \Lambda$ is a $2$-fold torsion ($\resp$ torsion-free) class of $\md \Lambda$. 
	Assume that $\cx$ is a $3$-fold torsion ($\resp$ torsion-free) class and then there exists a $2$-fold torsion ($\resp$ torsion-free) class $\ct_{2}$ ($\resp$ $\cf_{2}$) of $\md \Lambda$ such that $\cx \tors \ct_{2}$ ($\resp$ $\cx \torf \cf_{2}$). 
	It follows from Proposition \ref{prop CE = 2tors} ($\resp$ Theorem \ref{thm KE=2torf}), Proposition \ref{prop ICE = serre in tors} ($\resp$ Proposition \ref{prop IKE = serre in torf}) and Corollary \ref{cor ICE=CE} ($\resp$ the dual of it) that there exists some $1$-fold torsion ($\resp$ torsion-free) class $\ct_{1}$ ($\resp$ $\cf_{1}$) of $\md \Lambda$ such that 
	$\ct_{2}$ ($\resp$ $\cf_{2}$) is a Serre subcategory of $\ct_{1}$ ($\resp$ $\cf_{1}$). We show that $\cx$ is a torsion ($\resp$ torsion-free) class of $\ct_{1}$ ($\resp$ $\cf_{1}$). By Lemma \ref{lem extension closed}, $\cx$ is closed under conflations in $\ct_{1}$ ($\resp$ $\cf_{1}$). 
	It remains to show that $\cx$ is closed under admissible quotients in $\ct_{1}$ ($\resp$ admissible subobjects in $\cf_{1}$). Take any conflation $0 \to T \to X \to Q \to 0$ in $\ct_{1}$ ($\resp$ $0 \to S \to X \to F \to 0$ in $\cf_{1}$) with $X \in \cx$. Since $X \in \cx \subset \ct_{2}$ ($\resp$ $X \in \cx \subset \cf_{2}$) and $\ct_{2}$ ($\resp$ $\cf_{2}$) is a Serre subcategory of $\ct_{1}$ ($\resp$ $\cf_{1}$), 
	the conflation $0 \to T \to X \to Q \to 0$ ($\resp$ $0 \to S \to X \to F \to 0$) is a conflation in $\ct_{2}$ ($\resp$ $\cf_{2}$). This implies $Q \in \cx$ ($\resp$ $S \in \cx$) because $\cx$ is a torsion ($\resp$ torsion-free) class of $\ct_{2}$ ($\resp$ $\cf_{2}$). This completes the proof.

	(``only if'' part): This follows from Remark \ref{rem nfold implies n+1fold} (2).
\end{proof}

\section{Examples} \label{sec example}
Let $\Lambda$ be a finite dimensional algebra over an algebraically closed field $\mathbb{K}$.

\begin{exam}
	Let $m$ be a positive integer greater than 2 and $\Lambda := \mathbb{K} \mathsf{A}_{m} / I$, where $\mathsf{A}_{m}$ is the following quiver and $I$ is the ideal generated by all paths whose length is 2 of $\md \mathbb{K} \mathsf{A}_{m}$:
  \[
	\begin{tikzpicture}[auto]
		\node (1) at (0, 0) {$1$}; \node (2) at (2.0, 0) {$2$}; \node (3) at (4.0, 0) {}; 
		\node (4) at (5.0, 0) {\large $\cdots$}; \node (5) at (6.0, 0) {}; \node (6) at (8.0, 0) {$m$};
		\node at (8.2, -0.1) {.};
		\draw[->, thick] (6) -- (5); \draw[->, thick] (3) -- (2); \draw[->, thick] (2) -- (1); 
	\end{tikzpicture} 
	\]
	The AR quiver of $\md \Lambda$ is as follows:

	\[
	\begin{tikzpicture}[auto]
		\node (1) at (0, 0) {$P_{1}$}; \node (2) at (2.0, 0) {$S_{2}$}; \node (3) at (4.0, 0) {};
		\node (12) at (1.0, 1.0) {$P_{2}$}; \node (23) at (3.0, 1.0) {$P_{3}$};
		\node  at (5.0, 0.5) {\large $\cdots$}; \node (m-1) at (6.0, 0.0) {};
		\node (m-1m) at (7.0, 1.0) {$P_{m}$}; \node (m) at (8.0, 0.0) {$S_{m}$};
		\node at (8.3, -0.1) {.};

		\draw[->, thick, dotted] (3) -- (2); \draw[->, thick, dotted] (2) -- (1);
		\draw[->, thick] (1) -- (12); \draw[->, thick] (12) -- (2); \draw[->, thick] (23) -- (3);
		\draw[->, thick] (2) -- (23); \draw[->, thick] (m-1) -- (m-1m); \draw[->, thick] (m-1m) -- (m);
		\draw[->, thick, dotted] (m) -- (m-1);
	\end{tikzpicture} 
	\]
	Then $\cc := \add (P_{2} \oplus P_{3} \oplus \cdots \oplus P_{m})$ is $\mathbf{NOT}$ an $(m-1)$-fold torsion-free class but an $m$-fold torsion-free class of $\md \Lambda$. 

	Indeed, there is an exact sequence $0 \to P_{1} \to P_{2} \to \cdots \to P_{m}$ in $\md \Lambda$ and hence $\cc$ is not a $\K^{(m-2)}\E$-closed subcategory of $\md \Lambda$; in particular, not an $(m-1)$-fold torsion-free class of $\md \Lambda$.
	Since the global dimension of $\Lambda$ is $m-1$, in a similar way as the proof of Example \ref{ex KnE} (1) $(e) \Longrightarrow (a)$, it follows that $\cd := \add (P_{1} \oplus P_{2} \oplus \cdots \oplus P_{m})$ is precisely an $(m-1)$-fold torsion-free class of $\md \Lambda$. 
	We show that $\cc$ is a torsion-free class of $\cd$. Since $\cc$ consists of projective $\Lambda$-modules, it is closed under extensions in $\md \Lambda$. Thus it is closed under conflations in $\cd$ by Lemma \ref{lem extension closed}. It remains to show that $\cc$ is closed under admissible subobjects in $\cd$. 
	Take any conflation $0 \to D' \to C \to D'' \to 0$ in $\cd$ with $C \in \cc$. Since $D''$ is projective in $\md \Lambda$, the conflation is split and hence $C = D' \oplus D''$. 
	This implies $D' \in \cc$. We have the desired results.
\end{exam}


\small
\bibliographystyle{amsalpha}
\bibliography{myreferences}
\end{document}